\makeatletter

\documentclass[
    11pt,
    a4paper,
    oneside,
    openright,
    center,
    chapterbib,
    crosshair,
    fleqn,
    headcount,
    headline,
    indent,
    indentfirst=false,
    portrait,
    phonetic,
    oldernstyle,
    onecolumn,
    sfbold,
    upper,
]{amsart}

\let\th@plain\relax

\PassOptionsToPackage{T2A}{fontenc} 
\PassOptionsToPackage{utf8}{inputenc} 
\PassOptionsToPackage{cyrpart}{cyrillic}
\PassOptionsToPackage{british,ngerman,russian}{babel}
\PassOptionsToPackage{
    english,
    ngerman,
    russian,
    capitalise,
}{cleveref}
\PassOptionsToPackage{
    margin=10pt,
    position=bottom,
    justification=centering,
    font=small,
    labelformat=simple,
    labelfont={sc},
    labelsep={period},
    textfont={it},
}{caption}
\PassOptionsToPackage{
    margin=10pt,
    position=bottom,
    justification=centering,
    font=small,
    labelformat=parens,
    labelfont={bf},
    labelsep={space},
    textfont={it},
}{subcaption}
\PassOptionsToPackage{framemethod=TikZ}{mdframed}
\PassOptionsToPackage{
    bookmarks=true,
    bookmarksopen=false,
    bookmarksopenlevel=0,
    bookmarkstype=toc,
    raiselinks=true,
    colorlinks=true,
    hyperfigures=true,
    anchorcolor=red,
    citecolor=green,
    linkcolor=blue,
    urlcolor=blue,
}{hyperref} 
\PassOptionsToPackage{normalem}{ulem}
\PassOptionsToPackage{
    amsmath,
    thmmarks,
}{ntheorem}
\PassOptionsToPackage{
    table,
}{xcolor}
\PassOptionsToPackage{
    all,
    color,
    curve,
    frame,
    import,
    knot,
    line,
    movie,
    rotate,
    textures,
    tile,
    tips,
    web,
    xdvi,
}{xy}
\PassOptionsToPackage{
    reset,
    left=1in,
    right=1in,
    top=20mm,
    bottom=20mm,
    heightrounded,
}{geometry}
\PassOptionsToPackage{
    symbol*, 
    multiple, 
}{footmisc}
\PassOptionsToPackage{overload}{textcase}
\PassOptionsToPackage{indentafter}{titlesec}

\usepackage{amsfonts}
\usepackage{amsmath}
\usepackage{amssymb}
\usepackage{ntheorem} 
\usepackage{array}
\usepackage{babel}
\usepackage{bbding}
\usepackage{bm}
\usepackage{bbm}
\usepackage{bibentry}
\usepackage{booktabs}
\usepackage{bold-extra} 
\usepackage{calc}
\usepackage{cancel}
\usepackage{caption} 
\usepackage{changepage}
\usepackage{cjhebrew}
\usepackage{cmlgc}
\usepackage{colonequals}
\usepackage{color}
\usepackage{comment}
\usepackage{datetime}
\usepackage{dsfont}
\usepackage{etex}
\usepackage{etoolbox}
\usepackage{eurosym}
\usepackage{fancybox}
\usepackage{fancyhdr}
\usepackage{float}
\usepackage{fontenc}
\usepackage{footmisc}
\usepackage{fp}
\usepackage{geometry}
\usepackage{graphicx}
\usepackage{ifpdf}
\usepackage{ifthen}
\usepackage{ifoddpage}
\usepackage{ifnextok}  
\usepackage{inputenc}
\usepackage{latexsym}
\usepackage{lineno}
\usepackage{listings}
\usepackage{longtable}
\usepackage{lscape}
\usepackage{mathrsfs}
\usepackage{multicol}
\usepackage{multirow}
\usepackage{nameref}
\usepackage{nowtoaux}
\usepackage{paralist}
\usepackage{enumerate} 
\usepackage{pgf}
\usepackage{pgfplots}
\usepackage{phonetic}
\usepackage{proof}
\usepackage{qtree}
\usepackage{refcount}
\usepackage{savesym}
\usepackage{stmaryrd}
\usepackage{synttree}
\usepackage{subcaption}
\usepackage{suffix}
\usepackage{yfonts} 
\usepackage{textcase} 
\usepackage{tikz}
\usepackage{xy}
\usepackage{wrapfig}
\usepackage{xcolor}
\usepackage{xspace}
\usepackage{xstring}
\usepackage{hyperref}
\usepackage{arydshln}
\usepackage{cleveref} 

\pgfplotsset{compat=newest}
\usetikzlibrary{math}

\usetikzlibrary{
    angles,
    arrows,
    automata,
    calc,
    decorations,
    decorations.pathmorphing,
    decorations.pathreplacing,
    positioning,
    patterns,
    quotes,
}

\savesymbol{corresponds}
\savesymbol{Diamond}
\savesymbol{emptyset}
\savesymbol{ggg}
\savesymbol{int}
\savesymbol{lll}
\savesymbol{RectangleBold}
\savesymbol{langle}
\savesymbol{rangle}
\savesymbol{hookrightarrow}
\savesymbol{hookleftarrow}
\savesymbol{Asterisk}
\usepackage{mathabx}
\usepackage{wasysym}

\restoresymbol{x}{corresponds}
\restoresymbol{x}{Diamond}
\restoresymbol{x}{emptyset}
\restoresymbol{x}{ggg}
\restoresymbol{x}{int}
\restoresymbol{x}{lll}
\restoresymbol{x}{RectangleBold}
\restoresymbol{x}{langle}
\restoresymbol{x}{rangle}
\restoresymbol{x}{hookrightarrow}
\restoresymbol{x}{hookleftarrow}
\restoresymbol{x}{Asterisk}

\ifpdf
    \usepackage{pdfcolmk}
\fi

\usepackage{mdframed}

\DeclareFontFamily{U}{MnSymbolA}{}
\DeclareFontShape{U}{MnSymbolA}{m}{n}{
    <-6> MnSymbolA5
    <6-7> MnSymbolA6
    <7-8> MnSymbolA7
    <8-9> MnSymbolA8
    <9-10> MnSymbolA9
    <10-12> MnSymbolA10
    <12-> MnSymbolA12
}{}
\DeclareFontShape{U}{MnSymbolA}{b}{n}{
    <-6> MnSymbolA-Bold5
    <6-7> MnSymbolA-Bold6
    <7-8> MnSymbolA-Bold7
    <8-9> MnSymbolA-Bold8
    <9-10> MnSymbolA-Bold9
    <10-12> MnSymbolA-Bold10
    <12-> MnSymbolA-Bold12
}{}
\DeclareSymbolFont{MnSymA}{U}{MnSymbolA}{m}{n}
\DeclareMathSymbol{\lcirclearrowright}{\mathrel}{MnSymA}{252}
\DeclareMathSymbol{\lcirclearrowdown}{\mathrel}{MnSymA}{255}
\DeclareMathSymbol{\rcirclearrowleft}{\mathrel}{MnSymA}{250}
\DeclareMathSymbol{\rcirclearrowdown}{\mathrel}{MnSymA}{251}

\DeclareFontFamily{U}{MnSymbolC}{}
\DeclareSymbolFont{MnSyC}{U}{MnSymbolC}{m}{n}
\DeclareFontShape{U}{MnSymbolC}{m}{n}{
    <-6>  MnSymbolC5
    <6-7>  MnSymbolC6
    <7-8>  MnSymbolC7
    <8-9>  MnSymbolC8
    <9-10> MnSymbolC9
    <10-12> MnSymbolC10
    <12->   MnSymbolC12%
}{}
\DeclareMathSymbol{\powerset}{\mathord}{MnSyC}{180}
\DeclareMathSymbol{\righthalfcap}{\mathbin}{MnSyC}{186}

\DeclareMathAlphabet{\mathpzc}{OT1}{pzc}{m}{it}

\def\boolwahr{true}
\def\boolfalsch{false}
\def\boolleer{}

\let\boolinappendix\boolfalsch
\let\boolinmdframed\boolfalsch
\let\eqtagset\boolfalsch
\let\eqtaglabel\boolleer
\let\eqtagsymb\boolleer

\newcount\bufferctr
\newcount\bufferreplace

\newlength\rtab
\newlength\gesamtlinkerRand
\newlength\gesamtrechterRand
\newlength\ownspaceabovethm
\newlength\ownspacebelowthm
\def\secnumberingpt{.}
\def\secnumberingseppt{.}
\def\subsecnumberingseppt{}
\def\thmnumberingpt{.}
\def\thmnumberingseppt{}
\def\thmForceSepPt{.}

\definecolor{leer}{gray}{1}
\definecolor{boxgrau}{gray}{0.85}
\definecolor{dunkelgrau}{gray}{0.5}
\definecolor{maroon}{rgb}{0.6901961,0.1882353,0.3764706}
\definecolor{dunkelgruen}{rgb}{0.015625,0.363281,0.109375}
\definecolor{dunkelrot}{rgb}{0.5450980392,0,0}
\definecolor{dunkelblau}{rgb}{0,0,0.5450980392}
\definecolor{blau}{rgb}{0,0,1}
\definecolor{newresult}{rgb}{0.6,0.6,0.6}
\definecolor{improvedresult}{rgb}{0.9,0.9,0.9}
\definecolor{hervorheben}{rgb}{0,0.9,0.7}
\definecolor{starkesblau}{rgb}{0.1019607843,0.3176470588,0.8156862745}
\definecolor{achtung}{rgb}{1,0.5,0.5}
\definecolor{frage}{rgb}{0.5,1,0.5}
\definecolor{schreibweise}{rgb}{0,0.7,0.9}
\definecolor{axiom}{rgb}{0,0.3,0.3}

\definecolor{drawing_light_grey}{gray}{0.85}

\def\let@name#1#2{
    \expandafter\let\csname #1\expandafter\endcsname\csname #2\endcsname\relax
}
\DeclareRobustCommand\crfamily{\fontfamily{ccr}\selectfont}
\DeclareTextFontCommand{\textcr}{\crfamily}

\def\ifthenelseleer#1#2#3{\ifthenelse{\equal{#1}{}}{#2}{#1#3}}
\def\bedingtesspaceexpand#1#2#3{\ifthenelseleer{\csname #1\endcsname}{#3}{#2#3}}

\def\hraum{\null\hfill\null}

\def\nvraum{\@ifnextchar\bgroup{\nvraum@c}{\nvraum@bes}}
    \def\nvraum@c#1{\vspace*{-#1\baselineskip}}
    \def\nvraum@bes{\vspace*{-\baselineskip}}
\def\erlaubeplatz{\relax\ifmmode\else\@\xspace\fi}
\def\entferneplatz{\relax\ifmmode\else\expandafter\@gobble\fi}

\def\send@toaux#1{\@bsphack\protected@write\@auxout{}{\string#1}\@esphack}

\def\rlabel#1[#2]#3#4#5{#5\rlabel@aux{#1}[#2]{#3}{#4}{#5}}
    \def\rlabel@aux#1[#2]#3#4#5{%
        \send@toaux{\newlabel{#1}{{\@currentlabel}{\thepage}{{\unexpanded{#5}}}{#2.\csname the#2\endcsname}{}}}\relax%
    }

\def\tag@rawscheme#1#2[#3]#4#5{\@ifnextchar[{\tag@rawscheme@{#1}{#2}[#3]{#4}{#5}}{\tag@rawscheme@{#1}{#2}[#3]{#4}{#5}[*]}}
    \def\tag@rawscheme@#1#2[#3]#4#5[#6]{\@ifnextchar\bgroup{\tag@rawscheme@@{#1}{#2}[#3]{#4}{#5}[#6]}{\tag@rawscheme@@{#1}{#2}[#3]{#4}{#5}[#6]{}}}
    \def\tag@rawscheme@@#1#2[#3]#4#5[#6]#7{%
        \ifthenelse{\equal{#6}{*}}{%
            \ifthenelse{\equal{#7}{\boolleer}}{\refstepcounter{#3}#4\csname the#3\endcsname#5}{#4#7#5}%
        }{%
            \refstepcounter{#3}#4%
            \ifthenelse{\equal{#7}{\boolleer}}{\rlabel{#6}[#3]{#1}{#2}{\csname the#3\endcsname}}{\rlabel{#6}[#3]{#1}{#2}{#7}}%
            #5%
        }%
    }
\def\tag@scheme#1#2[#3]{\tag@rawscheme{#1}{#2}[#3]{\upshape(}{\upshape)}}

\def\eqtag@post#1{\makebox[0pt][r]{#1}}
\def\eqtag@pre{\tag@scheme{Eq}{Equation}[Xe]}
\def\eqtag{\@ifnextchar[{\eqtag@}{\eqtag@[*]}}
    \def\eqtag@[#1]{\@ifnextchar\bgroup{\eqtag@@[#1]}{\eqtag@@[#1]{}}}
    \def\eqtag@@[#1]#2{\eqtag@post{\eqtag@pre[#1]{#2}}}

\def\eqcref#1{\text{(\ref{#1})}}
\def\punktlabel#1{\label{it:#1:\beweislabel}}
\def\punktcref#1{\eqcref{it:#1:\beweislabel}}

\def\opfromto[#1]_#2^#3{\underset{#2}{\overset{#3}{#1}}}
\def\textoverset#1#2{\overset{\text{#1}}{#2}}

\def\eqcrefoverset#1#2{\textoverset{\eqcref{#1}}{#2}}

\def\mathclap#1{#1}
\def\oberunterset#1{\@ifnextchar^{\oberunterset@oben{#1}}{\oberunterset@unten{#1}}}
    \def\oberunterset@oben#1^#2_#3{\underset{\mathclap{#3}}{\overset{\mathclap{#2}}{#1}}}
    \def\oberunterset@unten#1_#2^#3{\underset{\mathclap{#2}}{\overset{\mathclap{#3}}{#1}}}
    \def\breitunderbrace#1_#2{\underbrace{#1}_{\mathclap{#2}}}
    \def\breitoverbrace#1^#2{\overbrace{#1}^{\mathclap{#2}}}
    \def\breitunderbracket#1_#2{\underbracket{#1}_{\mathclap{#2}}}
    \def\breitoverbracket#1^#2{\overbracket{#1}^{\mathclap{#2}}}

\def\generatenestedsecnumbering#1#2#3{%
    \expandafter\gdef\csname thelong#3\endcsname{%
        \expandafter\csname the#2\endcsname%
        \secnumberingpt%
        \expandafter\csname #1\endcsname{#3}%
    }%
    \expandafter\gdef\csname theshort#3\endcsname{%
        \expandafter\csname #1\endcsname{#3}%
    }%
}
\def\generatenestedthmnumbering#1#2#3{%
    \expandafter\gdef\csname the#3\endcsname{%
        \expandafter\csname the#2\endcsname%
        \thmnumberingpt%
        \expandafter\csname #1\endcsname{#3}%
    }%
    \expandafter\gdef\csname theshort#3\endcsname{%
        \expandafter\csname #1\endcsname{#3}%
    }%
}

\providecommand{\setcounternach}{}
\renewcommand{\setcounternach}[2]{\setcounter{#1}{#2}\addtocounter{#1}{-1}}

\def\forcepunkt#1{#1\IfEndWith{#1}{.}{}{.}}
\def\lateinabkuerzung#1#2{%
    \expandafter\gdef\csname #1\endcsname{\emph{#2}\@ifnextchar.{\entferneplatz}{\erlaubeplatz}}
}
\def\deutscheabkuerzung#1#2{%
    \expandafter\gdef\csname #1\endcsname{{#2}\@ifnextchar.{\entferneplatz}{\erlaubeplatz}}
}

\def\matrix#1{\left(\begin{array}{#1}}
    \def\endmatrix{\end{array}\right)}
\def\smatrix{\left(\begin{smallmatrix}}
    \def\endsmatrix{\end{smallmatrix}\right)}

\def\multiargrekursiverbefehl#1#2#3#4#5#6#7#8{%
    \expandafter\gdef\csname#1\endcsname #2##1#4{\csname #1@anfang\endcsname##1#3\egroup}
    \expandafter\def\csname #1@anfang\endcsname##1#3{#5##1\@ifnextchar\egroup{\csname #1@ende\endcsname}{#7\csname #1@mitte\endcsname}}
    \expandafter\def\csname #1@mitte\endcsname##1#3{#6##1\@ifnextchar\egroup{\csname #1@ende\endcsname}{#7\csname #1@mitte\endcsname}}
    \expandafter\def\csname #1@ende\endcsname##1{#8}
}
\multiargrekursiverbefehl{svektor}{[}{;}{]}{\begin{smatrix}}{}{\\}{\\\end{smatrix}}
\multiargrekursiverbefehl{vektor}{[}{;}{]}{\begin{matrix}{c}}{}{\\}{\\\end{matrix}}
\multiargrekursiverbefehl{vektorzeile}{}{,}{;}{}{&}{}{}
\multiargrekursiverbefehl{matlabmatrix}{[}{;}{]}{\begin{smatrix}\vektorzeile}{\vektorzeile}{;\\}{;\end{smatrix}}

\def\BeweisRichtung[#1]{\@ifnextchar\bgroup{\@BeweisRichtung@c[#1]}{\@BeweisRichtung@bes[#1]}}
    \def\@BeweisRichtung@bes[#1]{{\bfseries (#1)}}
    \def\@BeweisRichtung@c[#1]#2#3{#2~#1~#3}
\def\erzeugeBeweisRichtungBefehle#1#2{
    \expandafter\gdef\csname #1text\endcsname##1##2{\BeweisRichtung[#2]{##1}{##2}}
    \expandafter\gdef\csname #1\endcsname{%
        \@ifnextchar\bgroup{\csname #1@\endcsname}{\csname #1text\endcsname{}{}}%
    }
    \expandafter\gdef\csname #1@\endcsname##1##2{%
        \csname #1text\endcsname{\punktcref{##1}}{\punktcref{##2}}%
    }
}
\erzeugeBeweisRichtungBefehle{hinRichtung}{$\Rightarrow$}
\erzeugeBeweisRichtungBefehle{herRichtung}{$\Leftarrow$}
\erzeugeBeweisRichtungBefehle{hinherRichtung}{$\Leftrightarrow$}

\def\cal#1{\mathcal{#1}}
\def\brkt#1{\langle{}#1{}\rangle}
\def\mathfrak#1{\mbox{\usefont{U}{euf}{m}{n}#1}}

\def\rectangleblack{\text{\RectangleBold}}

\def\squarewhite{\Box}

\def\crefname@full#1#2#3#4#5{%
    \crefname{#1}{#2}{#3}
    \Crefname{#1}{#4}{#5}
}
\def\crefname@fullmod#1#2#3#4#5{%
    \crefname@full{#1}{#2}{#3}{#4}{#5}
    \crefname@full{#1@basic}{#2}{#3}{#4}{#5}
    \crefname@full{#1@withName}{#2}{#3}{#4}{#5}
}
\crefname@full{chapter}{chapter}{chapters}{Chapter}{Chapters}
\crefname@full{appendix}{appendix}{appendices}{Appendix}{Appendices}
\crefname@full{section}{section}{sections}{Section}{Sections}
\crefname@full{subsection}{section}{sections}{Section}{Sections}
\crefname@full{subsubsection}{section}{sections}{Section}{Sections}
\crefname@full{subsubsubsection}{section}{sections}{Section}{Sections}
\crefname@full{table}{table}{tables}{Table}{Tables}
\crefname@full{figure}{figure}{figures}{Figure}{Figures}
\crefname@full{subfigure}{figure}{figures}{Figure}{Figures}

\crefname@fullmod{thm}{theorem}{theorems}{Theorem}{Theorems}
\crefname@fullmod{thmStar}{theorem}{theorems}{Theorem}{Theorems}
\crefname@fullmod{conj}{conjecture}{conjectures}{Conjecture}{Conjectures}
\crefname@fullmod{cor}{corollary}{corollaries}{Corollary}{Corollaries}
\crefname@fullmod{defn}{definition}{definitions}{Definition}{Definitions}
\crefname@fullmod{conv}{convention}{conventions}{Convention}{Conventions}
\crefname@fullmod{e.g.}{example}{examples}{Example}{Examples}
\crefname@fullmod{prop}{proposition}{propositions}{Proposition}{Propositions}
\crefname@fullmod{proof}{proof}{proofs}{Proof}{Proofs}
\crefname@fullmod{lemm}{lemma}{lemmata}{Lemma}{Lemmata}
\crefname@fullmod{problem}{problem}{problems}{Problem}{Problems}
\crefname@fullmod{qstn}{question}{questions}{Question}{Questions}
\crefname@fullmod{rem}{remark}{remarks}{Remark}{Remarks}

\def\qedEIGEN#1{\@ifnextchar[{\qedEIGEN@c{#1}}{\qedEIGEN@bes{#1}}}
\def\qedEIGEN@bes#1{%
    \parfillskip=0pt
    \widowpenalty=10000
    \displaywidowpenalty=10000
    \finalhyphendemerits=0
    \leavevmode
    \unskip
    \nobreak
    \hfil
    \penalty50
    \hskip.2em
    \null
    \hfill
    #1
    \par%
}
\def\qedEIGEN@c#1[#2]{%
    \parfillskip=0pt
    \widowpenalty=10000
    \displaywidowpenalty=10000
    \finalhyphendemerits=0
    \leavevmode
    \unskip
    \nobreak
    \hfil
    \penalty50
    \hskip.2em
    \null
    \hfill
    {#1~{\small\bfseries\upshape (#2)}}%
    \par%
}
\def\qedVARIANT#1#2{
    \expandafter\def\csname ennde#1Sign\endcsname{#2}
    \expandafter\def\csname ennde#1\endcsname{\@ifnextchar[{\qedEIGEN@c{#2}}{\qedEIGEN@bes{#2}}} 
}
\qedVARIANT{OfProof}{$\squarewhite$}
\qedVARIANT{OfWork}{\rectangleblack}
\qedVARIANT{OfSomething}{$\dashv$}
\qedVARIANT{OnNeutral}{} 

\def\ra@pretheoremwork{
    \setlength{\theorempreskipamount}{\ownspaceabovethm}
    \setlength{\theorempostskipamount}{\ownspacebelowthm}
}
\def\rathmtransfer#1#2{
    \expandafter\def\csname #2\endcsname{\csname #1\endcsname}
    \expandafter\def\csname end#2\endcsname{\csname end#1\endcsname}
}

\def\ranewthm#1#2#3[#4]{
    \theoremstyle{\current@theoremstyle}
    \theoremseparator{\current@theoremseparator}
    \theoremprework{\ra@pretheoremwork}
    \@ifundefined{#1@basic}{\newtheorem{#1@basic}[#4]{#2}}{\renewtheorem{#1@basic}[#4]{#2}}
    \theoremstyle{\current@theoremstyle}
    \theoremseparator{\thmForceSepPt}
    \theoremprework{\ra@pretheoremwork}
    \@ifundefined{#1@withName}{\newtheorem{#1@withName}[#4]{#2}}{\renewtheorem{#1@withName}[#4]{#2}}
    \theoremstyle{nonumberplain}
    \theoremseparator{\thmForceSepPt}
    \theoremprework{\ra@pretheoremwork}
    \@ifundefined{#1@star@basic}{\newtheorem{#1@star@basic}[#4]{#2}}{\renewtheorem{#1@star@basic}[#4]{#2}}
    \theoremstyle{nonumberplain}
    \theoremseparator{\thmForceSepPt}
    \theoremprework{\ra@pretheoremwork}
    \@ifundefined{#1@star@withName}{\newtheorem{#1@star@withName}[#4]{#2}}{\renewtheorem{#1@star@withName}[#4]{#2}}
    \umbauenenv{#1}{#3}[#4]
    \umbauenenv{#1@star}{#3}[#4]
    \rathmtransfer{#1@star}{#1*}
}

\def\umbauenenv#1#2[#3]{%
    \expandafter\def\csname #1\endcsname{\relax%
        \@ifnextchar[{\csname #1@\endcsname}{\csname #1@\endcsname[*]}%
    }
    \expandafter\def\csname #1@\endcsname[##1]{\relax%
        \@ifnextchar[{\csname #1@@\endcsname[##1]}{\csname #1@@\endcsname[##1][*]}%
    }
    \expandafter\def\csname #1@@\endcsname[##1][##2]{%
        \ifx*##1%
            \def\enndeOfBlock{\csname end#1@basic\endcsname}
            \csname #1@basic\endcsname%
        \else%
            \def\enndeOfBlock{\csname end#1@withName\endcsname}
            \csname #1@withName\endcsname[##1]%
        \fi%
        \def\makelabel####1{%
            \gdef\beweislabel{####1}%
            \label{\beweislabel}%
        }%
        \ifx*##2%
            \def\enndeSymbol{\qedEIGEN{#2}}
        \else%
            \def\enndeSymbol{\qedEIGEN{#2}[##2]}
        \fi
    }
    \expandafter\gdef\csname end#1\endcsname{\enndeSymbol\enndeOfBlock}
}

    \def\current@theoremstyle{plain}
    \def\current@theoremseparator{\thmnumberingseppt}
    \theoremstyle{\current@theoremstyle}
    \theoremseparator{\current@theoremseparator}
    \theoremsymbol{}

\generatenestedthmnumbering{arabic}{section}{X}
\generatenestedthmnumbering{arabic}{section}{Xe}
\generatenestedthmnumbering{Roman}{section}{Xsp}

    \theoremheaderfont{\upshape\bfseries}
    \theorembodyfont{\slshape}

\ranewthm{thm}{Theorem}{\enndeOnNeutralSign}[X]
\ranewthm{lemm}{Lemma}{\enndeOnNeutralSign}[X]
\ranewthm{cor}{Corollary}{\enndeOnNeutralSign}[X]
\ranewthm{prop}{Proposition}{\enndeOnNeutralSign}[X]

    \theorembodyfont{\upshape}

\ranewthm{defn}{Definition}{\enndeOnNeutralSign}[X]
\ranewthm{conv}{Convention}{\enndeOnNeutralSign}[X]
\ranewthm{e.g.}{Example}{\enndeOnNeutralSign}[X]
\ranewthm{fact}{Fact}{\enndeOnNeutralSign}[X]
\ranewthm{problem}{Problem}{\enndeOnNeutralSign}[X]
\ranewthm{qstn}{Question}{\enndeOnNeutralSign}[X]
\ranewthm{rem}{Remark}{\enndeOnNeutralSign}[X]

    \theoremheaderfont{\itshape}
    \theorembodyfont{\upshape}

\ranewthm{proof@tmp}{Proof}{\enndeOfProofSign}[Xdisplaynone]
\rathmtransfer{proof@tmp*}{proof}

\def\behauptungbeleg@claim{%
    \iflanguage{british}{Claim}{%
    \iflanguage{english}{Claim}{%
    \iflanguage{ngerman}{Behauptung}{%
    \iflanguage{russian}{Утверждение}{%
    Claim%
    }}}}%
}
\def\behauptungbeleg@pf@kurz{%
    \iflanguage{british}{Pf}{%
    \iflanguage{english}{Pf}{%
    \iflanguage{ngerman}{Bew}{%
    \iflanguage{russian}{Доказательство}{%
    Pf%
    }}}}%
}
\def\behauptungbeleg{\@ifnextchar\bgroup{\behauptungbeleg@c}{\behauptungbeleg@bes}}
    \def\behauptungbeleg@c#1{\item[{\bfseries \behauptungbeleg@claim\erlaubeplatz #1.}]}
    \def\behauptungbeleg@bes{\item[{\bfseries \behauptungbeleg@claim.}]}
\def\belegbehauptung{\item[{\bfseries\itshape\behauptungbeleg@pf@kurz.}]}

\newdateformat{standardshort}{\oldstylenums{\THEYEAR}.\oldstylenums{\THEMONTH}.\oldstylenums{\THEDAY}}
\newdateformat{standardcompact}{\THEYEAR\twodigit{\THEMONTH}\twodigit{\THEDAY}}
\newdateformat{standardlong}{\THEYEAR\ \monthname\ \THEDAY}
\newcolumntype{\RECHTS}[1]{>{\raggedleft}p{#1}}
\newcolumntype{\LINKS}[1]{>{\raggedright}p{#1}}
\newcolumntype{m}{>{$}l<{$}}
\newcolumntype{C}{>{$}c<{$}}
\newcolumntype{L}{>{$}l<{$}}
\newcolumntype{R}{>{$}r<{$}}
\newcolumntype{0}{@{\hspace{0pt}}}
\newcolumntype{\LINKSRAND}{@{\hspace{\@totalleftmargin}}}
\newcolumntype{h}{@{\extracolsep{\fill}}}
\newcolumntype{i}{>{\itshape}}
\newcolumntype{t}{@{\hspace{\tabcolsep}}}
\newcolumntype{q}{@{\hspace{1em}}}
\newcolumntype{n}{@{\hspace{-\tabcolsep}}}
\newcolumntype{M}[2]{%
    >{\begin{minipage}{#2}\begin{math}}%
    {#1}%
    <{\end{math}\end{minipage}}%
}
\newcolumntype{T}[2]{%
    >{\begin{minipage}{#2}}%
    {#1}%
    <{\end{minipage}}%
}
\def\punkteumgebung@genbefehl#1#2#3{
    \punkteumgebung@genbefehl@{#1}{#2}{#3}{}{}
    \punkteumgebung@genbefehl@{multi#1}{#2}{#3}{
        \setlength{\columnsep}{10pt}%
        \setlength{\columnseprule}{0pt}%
        \begin{multicols}{\thecolumnanzahl}%
    }{\end{multicols}\nvraum{1}}
}
\def\punkteumgebung@genbefehl@#1#2#3#4#5{
    \expandafter\gdef\csname #1\endcsname{
        \@ifnextchar\bgroup{\csname #1@c\endcsname}{\csname #1@bes\endcsname}
    }
        \expandafter\def\csname #1@c\endcsname##1{
            \@ifnextchar[{\csname #1@c@\endcsname{##1}}{\csname #1@c@\endcsname{##1}[\z@]}
        }
        \expandafter\def\csname #1@c@\endcsname##1[##2]{
            \@ifnextchar[{\csname #1@c@@\endcsname{##1}[##2]}{\csname #1@c@@\endcsname{##1}[##2][\z@]}
        }
        \expandafter\def\csname #1@c@@\endcsname##1[##2][##3]{
            \let\alterlinkerRand\gesamtlinkerRand
            \let\alterrechterRand\gesamtrechterRand
            \addtolength{\gesamtlinkerRand}{##2}
            \addtolength{\gesamtrechterRand}{##3}
            \advance\linewidth -##2%
            \advance\linewidth -##3%
            \advance\@totalleftmargin ##2%
            \parshape\@ne \@totalleftmargin\linewidth%
            #4
            \begin{#2}[\upshape ##1]%
                \setlength{\parskip}{0.5\baselineskip}\relax%
                \setlength{\topsep}{\z@}\relax%
                \setlength{\partopsep}{\z@}\relax%
                \setlength{\parsep}{\parskip}\relax%
                \setlength{\itemsep}{#3}\relax%
                \setlength{\listparindent}{\z@}\relax%
                \setlength{\itemindent}{\z@}\relax%
        }
        \expandafter\def\csname #1@bes\endcsname{
            \@ifnextchar[{\csname #1@bes@\endcsname}{\csname #1@bes@\endcsname[\z@]}
        }
        \expandafter\def\csname #1@bes@\endcsname[##1]{
            \@ifnextchar[{\csname #1@bes@@\endcsname[##1]}{\csname #1@bes@@\endcsname[##1][\z@]}
        }
        \expandafter\def\csname #1@bes@@\endcsname[##1][##2]{
            \let\alterlinkerRand\gesamtlinkerRand
            \let\alterrechterRand\gesamtrechterRand
            \addtolength{\gesamtlinkerRand}{##1}
            \addtolength{\gesamtrechterRand}{##2}
            \advance\linewidth -##1%
            \advance\linewidth -##2%
            \advance\@totalleftmargin ##1%
            \parshape\@ne \@totalleftmargin\linewidth%
            #4
            \begin{#2}%
                \setlength{\parskip}{0.5\baselineskip}\relax%
                \setlength{\topsep}{\z@}\relax%
                \setlength{\partopsep}{\z@}\relax%
                \setlength{\parsep}{\parskip}\relax%
                \setlength{\itemsep}{#3}\relax%
                \setlength{\listparindent}{\z@}\relax%
                \setlength{\itemindent}{\z@}\relax%
        }
    \expandafter\gdef\csname end#1\endcsname{%
        \end{#2}#5
        \setlength{\gesamtlinkerRand}{\alterlinkerRand}
        \setlength{\gesamtlinkerRand}{\alterrechterRand}
    }
}

\def\ritempunkt{{\Large \textbullet}} 
\setdefaultitem{\ritempunkt}{\ritempunkt}{\ritempunkt}{\ritempunkt}
\punkteumgebung@genbefehl{itemise}{compactitem}{\parskip}{}{}
\punkteumgebung@genbefehl{kompaktitem}{compactitem}{\z@}{}{}
\punkteumgebung@genbefehl{enumerate}{compactenum}{\parskip}{}{}
\punkteumgebung@genbefehl{kompaktenum}{compactenum}{\z@}{}{}

\def\displaysum_#1{\@ifnextchar^{\displaysum@both_{#1}}{\displaysum@@sub{#1}}}
    \def\displaysum@both_#1^#2{\displaysum@@subsup{#1}{#2}}
    \def\displaysum@@sub#1{\mathop{\displaystyle\csname sum\endcsname_{#1}}}
    \def\displaysum@@subsup#1#2{\mathop{\displaystyle\csname sum\endcsname_{#1}^{#2}}}
\def\displaysup_#1{\@ifnextchar^{\displaysup@both_{#1}}{\displaysup@@sub{#1}}}
    \def\displaysup@both_#1^#2{\displaysup@@subsup{#1}{#2}}
    \def\displaysup@@sub#1{\mathop{\displaystyle\csname sup\endcsname_{#1}}}
    \def\displaysup@@subsup#1#2{\mathop{\displaystyle\csname sup\endcsname_{#1}^{#2}}}
\def\displaymin_#1{\@ifnextchar^{\displaymin@both_{#1}}{\displaymin@@sub{#1}}}
    \def\displaymin@both_#1^#2{\displaymin@@subsup{#1}{#2}}
    \def\displaymin@@sub#1{\mathop{\displaystyle\csname min\endcsname_{#1}}}
    \def\displaymin@@subsup#1#2{\mathop{\displaystyle\csname min\endcsname_{#1}^{#2}}}
\def\displaymax_#1{\@ifnextchar^{\displaymax@both_{#1}}{\displaymax@@sub{#1}}}
    \def\displaymax@both_#1^#2{\displaymax@@subsup{#1}{#2}}
    \def\displaymax@@sub#1{\mathop{\displaystyle\csname max\endcsname_{#1}}}
    \def\displaymax@@subsup#1#2{\mathop{\displaystyle\csname max\endcsname_{#1}^{#2}}}
\def\displaylim_#1{\@ifnextchar^{\displaylim@both_{#1}}{\displaylim@@sub{#1}}}
    \def\displaylim@both_#1^#2{\displaylim@@subsup{#1}{#2}}
    \def\displaylim@@sub#1{\mathop{\displaystyle\csname lim\endcsname_{#1}}}
    \def\displaylim@@subsup#1#2{\mathop{\displaystyle\csname lim\endcsname_{#1}^{#2}}}
\def\displayliminf_#1{\@ifnextchar^{\displayliminf@both_{#1}}{\displayliminf@@sub{#1}}}
    \def\displayliminf@both_#1^#2{\displayliminf@@subsup{#1}{#2}}
    \def\displayliminf@@sub#1{\mathop{\displaystyle\csname liminf\endcsname_{#1}}}
    \def\displayliminf@@subsup#1#2{\mathop{\displaystyle\csname liminf\endcsname_{#1}^{#2}}}
\def\displaylimsup_#1{\@ifnextchar^{\displaylimsup@both_{#1}}{\displaylimsup@@sub{#1}}}
    \def\displaylimsup@both_#1^#2{\displaylimsup@@subsup{#1}{#2}}
    \def\displaylimsup@@sub#1{\mathop{\displaystyle\csname limsup\endcsname_{#1}}}
    \def\displaylimsup@@subsup#1#2{\mathop{\displaystyle\csname limsup\endcsname_{#1}^{#2}}}

\def\matrix#1{\left(\begin{array}[mc]{#1}}
    \def\endmatrix{\end{array}\right)}
\def\smatrix{\left(\begin{smallmatrix}}
    \def\endsmatrix{\end{smallmatrix}\right)}

\def\multiargrekursiverbefehl#1#2#3#4#5#6#7#8{%
    \expandafter\gdef\csname#1\endcsname #2##1#4{\csname #1@anfang\endcsname##1#3\egroup}
    \expandafter\def\csname #1@anfang\endcsname##1#3{#5##1\@ifnextchar\egroup{\csname #1@ende\endcsname}{#7\csname #1@mitte\endcsname}}
    \expandafter\def\csname #1@mitte\endcsname##1#3{#6##1\@ifnextchar\egroup{\csname #1@ende\endcsname}{#7\csname #1@mitte\endcsname}}
    \expandafter\def\csname #1@ende\endcsname##1{#8}
}
\multiargrekursiverbefehl{svektor}{[}{;}{]}{\begin{smatrix}}{}{\\}{\\\end{smatrix}}
\multiargrekursiverbefehl{vektor}{[}{;}{]}{\begin{matrix}{c}}{}{\\}{\\\end{matrix}}
\multiargrekursiverbefehl{vektorzeile}{}{,}{;}{}{&}{}{}
\multiargrekursiverbefehl{matlabmatrix}{[}{;}{]}{\begin{smatrix}\vektorzeile}{\vektorzeile}{;\\}{;\end{smatrix}}

\def\underbracenodisplay#1{%
    \mathop{\vtop{\m@th\ialign{##\crcr
    $\hfil\displaystyle{#1}\hfil$\crcr
    \noalign{\kern3\p@\nointerlineskip}%
    \upbracefill\crcr\noalign{\kern3\p@}}}}\limits%
}

\def\maths[#1]#2{%
    \ifthenelse{\equal{\boolinmdframed}{\boolwahr}}{}{\begin{escapeeinzug}}
    \noindent%
    \let\eqtagset\boolfalsch
    \let\eqtaglabel\boolleer
    \let\eqtagsymb\boolleer
    \let\alteqtag\eqtag
    \def\eqtag{\@ifnextchar[{\eqtag@loc@}{\eqtag@loc@[*]}}%
    \def\eqtag@loc@[##1]{\@ifnextchar\bgroup{\eqtag@loc@@[##1]}{\eqtag@loc@@[##1]{}}}%
    \def\eqtag@loc@@[##1]##2{%
        \gdef\eqtagset{\boolwahr}
        \gdef\eqtaglabel{##1}
        \gdef\eqtagsymb{##2}
    }%
    \def\verticalalign{}%
        \IfBeginWith{#1}{t}{\def\verticalalign{t}}{}%
        \IfBeginWith{#1}{m}{\def\verticalalign{c}}{}%
        \IfBeginWith{#1}{b}{\def\verticalalign{b}}{}%
    \def\horizontalalign{\null\hfill\null}%
        \IfEndWith{#1}{l}{}{\null\hfill\null}%
        \IfEndWith{#1}{r}{\def\horizontalalign{}}{}%
    \begin{math}
    \begin{array}[\verticalalign]{0#2}%
}
    \def\endmaths{%
        \end{array}
        \end{math}\horizontalalign%
        \let\eqtag\alteqtag
        \ifthenelse{\equal{\eqtagset}{\boolwahr}}{\eqtag[\eqtaglabel]{\eqtagsymb}}{}
        \ifthenelse{\equal{\boolinmdframed}{\boolwahr}}{}{\end{escapeeinzug}}%
    }

\def\longmaths[#1]#2{\relax
    \let\altarraystretch\arraystretch
    \renewcommand\arraystretch{1.2}\relax
    \begin{longtable}[#1]{\LINKSRAND #2}
}
    \def\endlongmaths{
        \end{longtable}
        \renewcommand\arraystretch{\altarraystretch}
    }

\def\einzug{\@ifnextchar[{\indents@}{\indents@[\z@]}}
    \def\indents@[#1]{\@ifnextchar[{\indents@@[#1]}{\indents@@[#1][\z@]}}
    \def\indents@@[#1][#2]{%
        \begin{list}{}{\relax
            \setlength{\topsep}{\z@}\relax
            \setlength{\partopsep}{\z@}\relax
            \setlength{\parsep}{\parskip}\relax
            \setlength{\listparindent}{\z@}\relax
            \setlength{\itemindent}{\z@}\relax
            \setlength{\leftmargin}{#1}\relax
            \setlength{\rightmargin}{#2}\relax
            \let\alterlinkerRand\gesamtlinkerRand
            \let\alterrechterRand\gesamtrechterRand
            \addtolength{\gesamtlinkerRand}{#1}
            \addtolength{\gesamtrechterRand}{#2}
        }\relax
            \item[]\relax
    }
        \def\endeinzug{%
            \setlength{\gesamtlinkerRand}{\alterlinkerRand}
            \setlength{\gesamtlinkerRand}{\alterrechterRand}
            \end{list}%
        }

\def\escapeeinzug{\begin{einzug}[-\gesamtlinkerRand][-\gesamtrechterRand]}
    \def\endescapeeinzug{\end{einzug}}

\def\programmiercode{
    \modulolinenumbers[1]
    \begin{einzug}[\rtab][\rtab]%
    \begin{linenumbers}%
        \fontfamily{cmtt}\fontseries{m}\fontshape{u}\selectfont%
        \setlength{\parskip}{1\baselineskip}%
        \setlength{\parindent}{0pt}%
}
    \def\endprogrammiercode{
        \end{linenumbers}
        \end{einzug}
    }

\def\schattiertebox@genbefehl#1#2#3{
    \expandafter\gdef\csname #1\endcsname{%
        \@ifnextchar[{\csname #1@args\endcsname}{\csname #1@args\endcsname[#3]}
    }
        \expandafter\def\csname #1@args\endcsname[##1]{%
            \@ifnextchar[{\csname #1@args@l\endcsname[##1]}{\csname #1@args@n\endcsname[##1]}
        }
        \expandafter\def\csname #1@args@l\endcsname[##1][##2]{%
            \@ifnextchar[{\csname #1@args@l@r\endcsname[##1][##2]}{\csname #1@args@l@n\endcsname[##1][##2]}
        }
        \expandafter\def\csname #1@args@n\endcsname[##1]{%
            \let\boolinmdframed\boolwahr
            \begin{mdframed}[#2leftmargin=0,rightmargin=0,outermargin=0,innermargin=0,##1]
        }
        \expandafter\def\csname #1@args@l@n\endcsname[##1][##2]{%
            \let\boolinmdframed\boolwahr
            \begin{mdframed}[#2leftmargin=##2/2,rightmargin=##2/2,outermargin=##2/2,innermargin=##2/2,##1]
        }
        \expandafter\def\csname #1@args@l@r\endcsname[##1][##2][##3]{%
            \let\boolinmdframed\boolwahr
            \begin{mdframed}[#2leftmargin=##2,rightmargin=##3,outermargin=##2,innermargin=##3,##1]
        }
    \expandafter\gdef\csname end#1\endcsname{%
        \end{mdframed}
        \let\boolinmdframed\boolfalsch
    }
}
    \schattiertebox@genbefehl{schattiertebox}{
        splittopskip=0,%
        splitbottomskip=0,%
        frametitleaboveskip=0,%
        frametitlebelowskip=0,%
        skipabove=1\baselineskip,%
        skipbelow=1\baselineskip,%
        linewidth=2pt,%
        linecolor=black,%
        roundcorner=4pt,%
    }{
        backgroundcolor=leer,%
        nobreak=true,%
    }

    \schattiertebox@genbefehl{schattierteboxdunn}{
        splittopskip=0,%
        splitbottomskip=0,%
        frametitleaboveskip=0,%
        frametitlebelowskip=0,%
        skipabove=1\baselineskip,%
        skipbelow=1\baselineskip,%
        linewidth=1pt,%
        linecolor=black,%
        roundcorner=2pt,%
    }{
        backgroundcolor=leer,%
        nobreak=true,%
    }

\def\tikzsetzepfeil#1{%
    \begin{tikzpicture}[remember picture,overlay,>=latex]%
        \draw #1;%
    \end{tikzpicture}%
}

\def\tikzsetzekreise[#1]#2#3{%
    \tikzsetzepfeil{%
    [rounded corners,#1]%
        ([shift={(-\tabcolsep,0.75\baselineskip)}]#2)%
        rectangle%
        ([shift={(\tabcolsep,-0.5\baselineskip)}]#3)
    }%
}

\tikzset{
    >=stealth,
    auto,
    node distance=1cm,
    thick,
    main node/.style={
        circle,draw,font=\sffamily\Large\bfseries,minimum size=0pt
    },
    state/.style={minimum size=0pt}
    loop above right/.style={loop,out=30,in=60,distance=0.5cm},
    loop above left/.style={above left,out=150,in=120,loop},
    loop below right/.style={below right,out=330,in=300,loop},
    loop below left/.style={below left,out=240,in=210,loop},
    itria/.style={
        draw,dashed,shape border uses incircle,
        isosceles triangle,shape border rotate=90,yshift=-1.45cm
    },
    rtria/.style={
        draw,dashed,shape border uses incircle,
        isosceles triangle,isosceles triangle apex angle=90,
        shape border rotate=-45,yshift=0.2cm,xshift=0.5cm
    },
    ritria/.style={
        draw,dashed,shape border uses incircle,
        isosceles triangle,isosceles triangle apex angle=110,
        shape border rotate=-55,yshift=0.1cm
    },
    litria/.style={
        draw,dashed,shape border uses incircle,
        isosceles triangle,isosceles triangle apex angle=110,
        shape border rotate=235,yshift=0.1cm
    }
}

\renewenvironment{cases}[0]{\left\{\begin{array}[c]{0lcl}}{\end{array}\right.}

\lateinabkuerzung{cf}{cf.}
\lateinabkuerzung{Cf}{Cf.}
\lateinabkuerzung{idest}{i.e.}
\lateinabkuerzung{Idest}{I.e.}
\lateinabkuerzung{exempli}{e.g.}
\lateinabkuerzung{Exempli}{E.g.}
\lateinabkuerzung{etcetera}{etc.}
\lateinabkuerzung{etAlia}{et al.}
\lateinabkuerzung{viz}{viz.}
\deutscheabkuerzung{usw}{usw.}
\deutscheabkuerzung{oBdA}{o.\,B.\,d.\,A.}
\deutscheabkuerzung{OBdA}{O.\,B.\,d.\,A.}
\deutscheabkuerzung{oae}{o.\,\"A.}
\deutscheabkuerzung{oE}{o.\,E.}
\deutscheabkuerzung{OE}{O.\,E.}
\deutscheabkuerzung{og}{o.\,g.}
\deutscheabkuerzung{obengen}{o.\,g.}
\deutscheabkuerzung{obenst}{o.\,s.}
\deutscheabkuerzung{untenst}{u.\,s.}
\deutscheabkuerzung{Tfae}{T.\,f.\,a.\,e.}
\deutscheabkuerzung{tfae}{t.\,f.\,a.\,e.}
\deutscheabkuerzung{wrt}{wrt.}
\deutscheabkuerzung{randomvar}{r.\,v.}
\deutscheabkuerzung{iid}{i.\,i.\,d.}
\deutscheabkuerzung{almostEverywhere}{a.\,e.}
\deutscheabkuerzung{almostSurely}{a.\,s.}
\deutscheabkuerzung{withoutlog}{w.\,l.\,o.\,g.} 
\deutscheabkuerzung{Withoutlog}{W.\,l.\,o.\,g.} 
\deutscheabkuerzung{respectively}{resp.}
\newcommand{\usesinglequotes}[1]{`#1'}

\providecommand{\topInterior}{}
\renewcommand{\topInterior}[1]{\mathop{\textup{int}}(#1)}

\providecommand{\card}{}
\renewcommand{\card}[1]{\lvert #1 \rvert}

\providecommand{\signum}{}
\renewcommand{\signum}[1]{\mathop{\mathrm{sgn}}(#1)}

\providecommand{\abs}{}
\renewcommand{\abs}[1]{\lvert #1 \rvert}
\providecommand{\absLong}{}
\renewcommand{\absLong}[1]{\Big| #1 \Big|}
\providecommand{\MonoidAsTuple}{}
\renewcommand{\MonoidAsTuple}[3]{(#1, #2, #3)}
\providecommand{\norm}{}
\renewcommand{\norm}[1]{\lVert #1 \rVert}
\providecommand{\normLong}{}
\renewcommand{\normLong}[1]{\Big\| #1 \Big\|}

\providecommand{\opSpectrum}{}
\renewcommand{\opSpectrum}[1]{\mathop{\sigma}(#1)}

\providecommand{\opResolventSet}{}
\renewcommand{\opResolventSet}[1]{\mathop{\rho}(#1)}

\providecommand{\opRegRevExponents}{}
\renewcommand{\opRegRevExponents}[1]{\mathop{\mathrm{reg}}(#1)}

\providecommand{\isPartition}{}
\renewcommand{\isPartition}[2]{#1 \in \mathop{\mathrm{Part}}(#2)}

\def\Cts{\@ifnextchar_{\Cts@tief}{\Cts@tief_{}}}
    \def\Cts@tief_#1#2{\@ifnextchar\bgroup{\Cts@two_{#1}{#2}}{\Cts@one_{#1}{#2}}}
    \def\Cts@one_#1#2{C_{#1}\big(#2\big)}
    \def\Cts@two_#1#2#3{C_{#1}\big(#2,~#3\big)}

\newcommand{\onematrix}[0]{\text{\upshape\bfseries I}}
\newcommand{\zeromatrix}[0]{\mathbf{0}}
\newcommand{\onevector}[0]{\mathbf{1}}
\newcommand{\zerovector}[0]{\mathbf{0}}

\newcommand{\AlgebraUpperTr}[0]{\mathcal{G}}

\def\naturals{\mathbb{N}}
\def\naturalsPos{\mathbb{N}}
\def\naturalsZero{\mathbb{N}_{0}}
\def\integers{\mathbb{Z}}

\def\reals{\mathbb{R}}
\def\realsPos{\reals_{>0}}
\def\realsNonNeg{\reals_{\geq 0}}
\def\complex{\mathbb{C}}
\def\Torus{\mathbb{T}}

\def\BRAKET#1#2{\langle{}#1,~#2{}\rangle}
\def\brkt#1{\langle{}#1{}\rangle}

\newcommand{\Cnought}[0]{\ensuremath{C_{0}}}

\def\without{\mathbin{\setminus}}

\def\da{\partial}

\let\altphi\phi
\let\altvarphi\varphi
    \def\phi{\altvarphi}
    \def\varphi{\altphi}

\def\quer#1{\overline{#1}}

\def\lim{\mathop{\ell\mathrm{im}}}

\def\supp{\mathop{\textup{supp}}}
\def\dim{\mathop{\textup{dim}}}
\def\dom{\mathop{\textup{dom}}}

\def\ran{\mathop{\textup{ran}}}

\def\Re{\mathop{\mathfrak{R}\mathrm{e}}}
\def\Imag{\mathop{\mathfrak{I}\mathrm{m}}}
\def\imagunit{\imath}

\def\BoundedOpsSymbol{\mathfrak{L}}
\def\BoundedOps#1{\@ifnextchar\bgroup{\BoundedOps@two{#1}}{\mathop{\BoundedOpsSymbol}(#1)}}
    \def\BoundedOps@two#1#2{\mathop{\BoundedOpsSymbol}(#1,#2)}
\def\BoundedOpsInv#1{\@ifnextchar\bgroup{\BoundedOps@two{#1}}{\mathop{\BoundedOpsSymbol}(#1)^{\times}}}
    \def\BoundedOpsInv@two#1#2{\mathop{\BoundedOpsSymbol}(#1,#2)^{\times}}
\def\HilbertRaum{\mathcal{H}}

\def\topSOT{\text{\upshape \scshape sot}}

\@ifundefined{setcitestyle}{%
}{%
    \setcitestyle{numeric-comp,open={[},close={]}}
}

\renewcommand{\arraystretch}{1}
\def\firstparagraph{\noindent}
\def\continueparagraph{\noindent}

    \generatenestedsecnumbering{arabic}{section}{subsection}
    \generatenestedsecnumbering{arabic}{subsection}{subsubsection}
    \def\theunitnamesection{\thesection}

    \def\sectionname{}

    \let\appendix@orig\appendix
    \def\appendix{%
        \appendix@orig%
        \let\boolinappendix\boolwahr
        \addcontentsline{toc}{part}{\appendixname}%
        \addtocontents{toc}{\protect\setcounter{tocdepth}{0}}
        \def\sectionname{Appendix}%
        \def\theunitnamesection{\Alph{section}}%
    }
    \def\notappendix{%
        \let\boolinappendix\boolfalse
        \addtocontents{toc}{\protect\setcounter{tocdepth}{1 }}
        \def\sectionname{}%
        \def\theunitnamesection{\arabic{section}}%
    }

    \def\@seccntformat#1{%
        \protect\textup{%
            \protect\@secnumfont
            \expandafter\protect\csname format#1\endcsname%
            \csname the#1\endcsname
            \expandafter\protect\csname format#1@pt\endcsname%
            \space
        }%
    }

    \def\formatsection@text{\centering\Large\scshape}
    \def\formatsection@pt{\secnumberingseppt}
    \def\section{\@startsection{section}{1}{\z@}{.7\linespacing\@plus\linespacing}{.5\linespacing}{\formatsection@text}}

    \def\formatsubsection@text{\flushleft\bfseries\scshape}
    \def\formatsubsection@pt{\subsecnumberingseppt}
    \def\subsection{\@startsection{subsection}{2}{\z@}{\z@}{\z@\hspace{1em}}{\formatsubsection@text}}

    \renewcommand{\paragraph}[1]{%
        {\bfseries\itshape #1}\:%
    }

\DefineFNsymbols*{customAlphabet}{abcdefghijklmnopqrstuvwxyz}
\setfnsymbol{customAlphabet}

\def\footnotemark[#1]{\text{\textsuperscript{\getrefnumber{#1}}}}

\def\rafootnotectr{20}
\providecommand{\incrftnotectr}{}
\renewcommand{\incrftnotectr}[1]{%
    \addtocounter{#1}{1}%
    \ifnum\value{#1}>\rafootnotectr\relax
        \setcounter{#1}{0}%
    \fi%
}
\providecommand{\footnoteref}{}
\renewcommand{\footnoteref}[1]{\protected@xdef\@thefnmark{\ref{#1}}\@footnotemark}
\let\@old@footnotetext\footnotetext
\def\footnotetext[#1]#2{%
    \incrftnotectr{footnote}%
    \@old@footnotetext[\value{footnote}]{\label{#1}#2}%
}

\def\kopfzeiledefault{
    \lhead[]{}
    \lhead[]{}
    \chead[]{}
    \rhead[]{}
    \lfoot[]{}
    \cfoot{\footnotesize\thepage}
    \rfoot[]{}
}

\def\aktuellesfont{\csname lmodern\endcsname}
\def\documentfont{%
    \gdef\aktuellesfont{\csname lmodern\endcsname}%
    \fontfamily{lmr}\fontseries{m}\selectfont%
    \renewcommand{\sfdefault}{phv}%
    \renewcommand{\ttdefault}{pcr}%
    \renewcommand{\rmdefault}{cmr}
    \renewcommand{\bfdefault}{bx}%
    \renewcommand{\itdefault}{it}%
    \renewcommand{\sldefault}{sl}%
    \renewcommand{\scdefault}{sc}%
    \renewcommand{\updefault}{n}%
}

\allowdisplaybreaks

\def\startdocumentlayoutoptions{
    \selectlanguage{british}
    \setlength{\parskip}{0.25\baselineskip}
    \setlength{\parindent}{2em}
    \kopfzeiledefault
    \documentfont
    \normalsize
}

\providecommand{\highlightTerm}{}
\renewcommand{\highlightTerm}[1]{\emph{#1}}
\providecommand{\highlightForReview}{}
\renewcommand{\highlightForReview}[1]{\bgroup\color{blue}#1\egroup}

\def\@adminfootnotes{%
    \let\@makefnmark\relax
    \let\@thefnmark\relax
    \ifx\@empty\@date\else%
        \@footnotetext{\@setdate}%
    \fi%
    \ifx\@empty\@subjclass\else%
        \@footnotetext{\@setsubjclass}%
    \fi
    \ifx\@empty\@keywords\else%
        \@footnotetext{\@setkeywords}%
    \fi
    \ifx\@empty\thankses\else%
        \@footnotetext{\def\par{\let\par\@par}\@setthanks}%
    \fi
}

\def\@settitle{\Large\bfseries\scshape\@title}

\def\@maketitle{%
  \normalfont\normalsize
  \@adminfootnotes
  \@mkboth{\@nx\shortauthors}{\@nx\shorttitle}%
  \global\topskip42\p@\relax
  {\centering\@settitle}
  \ifx\@empty\authors\else{\centering\small\@setauthors}\fi
  \ifx\@empty\@date\else{\vtop{\centering\small\@date\@@par}}\fi
  \ifx\@empty\@dedicatory%
  \else%
    \baselineskip\p@
    \vtop{\centering{\footnotesize\itshape\@dedicatory\@@par}%
    \global\dimen@i\prevdepth}\prevdepth\dimen@i%
  \fi
  \@setabstract
  \normalsize
  \if@titlepage
    \newpage
  \else
    \dimen@34\p@\advance\dimen@-\baselineskip
  \fi
}

\def\addresseshere{%
  \bgroup
  \setlength{\parindent}{0pt}
  \enddoc@text
  \egroup
  \let\enddoc@text\relax
}

\makeatother

\begin{document}
\startdocumentlayoutoptions

\thispagestyle{plain}



\def\abstractname{Abstract}
\begin{abstract}
    Consider $d$ commuting $\Cnought$-semigroups
    (or equivalently: $d$-parameter $\Cnought$-semigroups)
    over a Hilbert space
    for $d \in \naturals$.
    In the literature
    (\cf \cite{Nagy1970,Slocinski1974,Slocinski1982,Ptak1985,LeMerdy1996DilMultiParam,Shamovich2017dilationsMultiParam}),
    conditions are provided to classify the existence of
    unitary and regular unitary dilations.
    Some of these conditions require inspecting values of the semigroups,
    some provide only sufficient conditions,
    and others involve verifying sophisticated properties of the generators.
    By focussing on semigroups with bounded generators,
    we establish a simple and natural condition on the generators,
    \viz
        \emph{complete dissipativity},
        which naturally extends the basic notion of the dissipativity of the generators.
    Using examples of non-doubly commuting semigroups,
    this property can be shown to be strictly stronger than dissipativity.
    As the first main result, we demonstrate that
    complete dissipativity completely characterises
    the existence of regular unitary dilations,
    and extend this to the case of arbitrarily many commuting $\Cnought$-semigroups.
    We furthermore show that all multi-parameter $\Cnought$-semigroups (with bounded generators)
    admit a weaker notion of regular unitary dilations,
    and provide simple sufficient norm criteria for complete dissipativity.
    The paper concludes with an application to the von Neumann polynomial inequality problem,
    which we formulate for the semigroup setting and solve negatively for all $d \geq 2$.
\end{abstract}



\title[Dilations of commuting $\Cnought$-semigroups with bounded generators and the von Neumann polynomial inequality]{Dilations of commuting $\Cnought$-semigroups with bounded generators and the von Neumann polynomial inequality}
\author{Raj Dahya}
\email{raj.dahya@web.de}
\address{Fakult\"at f\"ur Mathematik und Informatik\newline
Universit\"at Leipzig, Augustusplatz 10, D-04109 Leipzig, Germany}

\maketitle



\setcounternach{section}{1}



\section[Introduction]{Introduction}
\label{sec:introduction}


\firstparagraph
Dynamical systems can often be described by
evolving contractive operators over
    Hilbert or Banach spaces.
Characterising the possibility of embedding these into larger systems
described by surjective isometries
began in 1953 with the research of Sz.-Nagy, \etAlia in \cite{Nagy1953},
in which the unitary (power) dilation of contractions
and of $1$-parameter contractive $\Cnought$-semigroups
over Hilbert spaces is presented.
In 1955, Stinespring \cite{Stinespring1955} introduced dilation to the non-commutative setting
of Banach and $C^{\ast}$-algebras, thus opening the way for results
for more sophisticated dynamical systems.
For an overview of this development, see \exempli \cite{Averson2010DilationOverview,Shalit2021DilationBook}.
With the backdrop of these theoretical frameworks,
intensive research has been conducted over the decades to yield concrete results
for families of operators (see \exempli \cite{Ando1963pairContractions}),
classes of semigroups (see \exempli \cite{Slocinski1974,Slocinski1982,Ptak1985,LeMerdy1996DilMultiParam,Shamovich2017dilationsMultiParam}),
and dynamical systems over $C^{\ast}$- and $W^{\ast}$-algebras (see \exempli \cite{Tevzadze1999markhov,Fagnola2003qsde,Fagnola2006bookQSDE,Izumi2012E0semigroups,Laca2022lcmDilation}).

In this paper we focus on the semigroup setting.
Consider $d$ commuting $\Cnought$-semigroup $T_{1},T_{2},\ldots,T_{d}$ over a Hilbert space $\HilbertRaum$
for some $d\in\naturals$.
We say that $T_{1},T_{2},\ldots,T_{d}$ have a \highlightTerm{simultaneous unitary dilation}
if there is
    a Hilbert space $\HilbertRaum^{\prime}$,
    a bounded operator (necessarily an isometry)
        ${r\in\BoundedOps{\HilbertRaum}{\HilbertRaum^{\prime}}}$,
    and $d$ commuting unitary $\Cnought$-semigroups $U_{1},U_{2},\ldots,U_{d}$ over $\HilbertRaum^{\prime}$
    (which can be uniquely extended to commuting unitary representations of $(\reals,+,0)$),
such that
    $%
        \prod_{i=1}^{d}T(t_{i}) = r^{\ast}(\prod_{i=1}^{d}U(t_{i}))r
    $
holds for all $\mathbf{t}=(t_{i})_{i=1}^{d}\in\realsNonNeg^{d}$.
We call this a \highlightTerm{simultaneous regular unitary dilation} if
the stronger condition
    $%
        (\prod_{i=1}^{d}T(t_{i}^{-}))^{\ast}
        (\prod_{i=1}^{d}T(t_{i}^{+}))
        = r^{\ast}(\prod_{i=1}^{d}U(t_{i}))r
    $
holds for all $\mathbf{t}=(t_{i})_{i=1}^{d}\in\reals^{d}$,
where $t^{+}=\max\{t,0\}$ and $t^{-}=\max\{-t,0\}$
denote the positive and negative parts of $t$ for any $t\in\reals$.
One reason to consider regular unitary dilations,
is that they are more directly related to the notion of \usesinglequotes{positive definite functions},
which in turn characterise group representation of locally compact Hausdorff groups
(%
    \cf
    \cite[Theorem~I.7.1~b)]{Nagy1970},
    \cite[Notes, p.~52]{Pisier2001bookCBmaps},
    and
    \cite[\S{}3.3]{Folland2015bookHarmonicAnalysis}%
).

In
    \cite[Theorem~I.8.1]{Nagy1970},
    \cite{Slocinski1974},
    \cite[Theorem~2]{Slocinski1982},
    and
    \cite[Theorem~2.3]{Ptak1985},
it was proved that $T_{1},T_{2},\ldots,T_{d}$ have a simultaneous unitary dilation
if they are contractive and $d\in\{1,2\}$.
In \cite[Theorem~3.2]{Ptak1985} a general condition on $T_{1},T_{2},\ldots,T_{d}$
which we shall refer to as \emph{Brehmer positivity}
(see \Cref{defn:brehmer-operators:sig:article-dilation-problem-raj-dahya})
was found for the existence of simultaneous regular unitary dilations.
However, to verify this condition one needs to consider values of the semigroups.
In \cite[Theorem~2.2 and Theorem~3.1]{LeMerdy1996DilMultiParam} Le~Merdy fully classified the existence of a \emph{weaker notion} of simultaneous unitary dilations,
and applied this to commuting families of bounded analytic $\Cnought$-semigroups.
More recently, Shamovich and Vinnivok established in \cite{Shamovich2017dilationsMultiParam} sufficient conditions on generators for the existence of simultaneous unitary dilations.
These conditions are quite sophisticated and involve proving the existence
of embeddings of the generators of the marginals.



Note that there is a natural correspondence between $d$ commuting $\Cnought$-semigroups,
    $T_{1},T_{2},\ldots,T_{d}$
and $d$-parameter $\Cnought$-semigroups, $T$,
\idest $\topSOT$-continuous morphisms between the algebraic structures
    $(\realsNonNeg^{d}, +, \zerovector)$
    and
    $\MonoidAsTuple{\BoundedOps{\HilbertRaum}}{\circ}{\onematrix}$.
This correspondence is realised via the constructions
    $T(\mathbf{t}) = \prod_{i=1}^{d}T_{i}(t_{i})$
    for all $\mathbf{t}\in\realsNonNeg^{d}$
and the co-ordinate maps
    $T_{i}(t) = T(0,0,\ldots,\underset{i}{t},\ldots,0)$
    for $t\in\realsNonNeg$ and $i\in\{1,2,\ldots,d\}$.
In this way, the $T_{i}$ may be viewed as
the \highlightTerm{marginal semigroups} (or simply: the \highlightTerm{marginals}) of $T$.
It is well known that their generators $A_{1},A_{2},\ldots,A_{d}$ commute (even if they are unbounded).
See for example \cite[Proposition 1.1.8--9]{Butzer1967semiGrApproximationsBook}.
It is also straightforward to see that $T$ is contractive/unitary/isometric
if and only if each of the $T_{i}$ are.
Hence one may interchangeably refer to
    commuting families of $d$ contractive/unitary/isometric $\Cnought$-semigroups
    and
    $d$-parameter contractive/unitary/isometric $\Cnought$-semigroups.
For convenience, we shall primarily use the multi-parameter presentation throughout this paper.

In our research, we focus primarily on multi-parameter $\Cnought$-semigroups with bounded generators.
In \S{}\ref{sec:definitions} we introduce the special conditions
of \emph{complete dissipativity} and \emph{complete super dissipativity} on the generators of $T$
(see \Cref{defn:dissipation-operators-and-conditions:sig:article-dilation-problem-raj-dahya}).
In \S{}\ref{sec:algebra} we develop algebraic identities involving these notions.
In \S{}\ref{sec:results} the main classification result is proved:

\begin{thm}[Classfication via complete dissipativity]
\makelabel{thm:classification:sig:article-dilation-problem-raj-dahya}
    Let $d\in\naturalsPos$
    and $T$ be a $d$-parameter $\Cnought$-semigroup over $\HilbertRaum$
    with bounded generators.
    Then the following statements are equivalent.

    \begin{kompaktenum}{\bfseries (1)}[\rtab]
        \item\punktlabel{1}
            The semigroup $T$ has a regular unitary dilation.
        \item\punktlabel{2}
            The generators of $T$ are completely dissipative.
        \item\punktlabel{3}
            There is a net $(T^{(\alpha)})_{\alpha\in\cal{I}}$
            consisting of
            regularly unitarily dilatable
            $d$-parameter $\Cnought$-semigroups over $\HilbertRaum$,
            such that
                $%
                    (T^{(\alpha)})_{\alpha\in\cal{I}}
                    \longrightarrow T
                $
            uniformly in norm on compact subsets of $\realsNonNeg^{d}$,
                \idest
                ${%
                    \sup_{\mathbf{t} \in L}\norm{T^{(\alpha)}(\mathbf{t}) - T(\mathbf{t})}
                        \underset{\alpha}{\longrightarrow}
                            0
                }$
                for all compact $L \subseteq \realsNonNeg^{d}$.
        \item\punktlabel{3dash}
            There is a net $(T^{(\alpha)})_{\alpha\in\cal{I}}$
            consisting of
            regularly unitarily dilatable
            $d$-parameter $\Cnought$-semigroups over $\HilbertRaum$,
            such that
                $%
                    (T^{(\alpha)})_{\alpha\in\cal{I}}
                    \longrightarrow T
                $
            uniformly in the \topSOT-topology on compact subsets of $\realsNonNeg^{d}$,
                \idest
                ${%
                    \sup_{\mathbf{t} \in L}\norm{(T^{(\alpha)}(\mathbf{t}) - T(\mathbf{t}))\xi}
                        \underset{\alpha}{\longrightarrow}
                            0
                }$
                for all $\xi\in\HilbertRaum$
                and
                all compact $L \subseteq \realsNonNeg^{d}$.
    \end{kompaktenum}

    \nvraum{1}

\end{thm}

As a consequence of this classification, we further show that all multi-parameter $\Cnought$-semigroups
with bounded generators admit weaker notions of regular unitary dilations.

\begin{cor}[Regular unitary dilation up to exponential equivalence]
\makelabel{cor:dilatable-up-to-exponential-modification:sig:article-dilation-problem-raj-dahya}
    Let $d\in\naturalsPos$
    and $T$ be a $d$-parameter $\Cnought$-semigroup over $\HilbertRaum$
    with bounded generators.
    Then for some $\boldsymbol{\omega} \in \realsNonNeg^{d}$,
    the modified $d$-parameter $C_{0}$-semigroup
    $%
        \tilde{T}
        \colonequals (
            e^{
                -\BRAKET{\mathbf{t}}{\boldsymbol{\omega}}
            }
            T(\mathbf{t})
        )_{\mathbf{t}\in\realsNonNeg^{d}}
    $
    has a regular unitary dilation.
\end{cor}

In the remainder of \S{}\ref{sec:results} we extend \Cref{thm:classification:sig:article-dilation-problem-raj-dahya}
to arbitrarily many commuting $\Cnought$-semigroups (see \Cref{cor:general-classification:sig:article-dilation-problem-raj-dahya}).
We further explore the set of $\boldsymbol{\omega} \in \reals^{d}$
for which \Cref{cor:dilatable-up-to-exponential-modification:sig:article-dilation-problem-raj-dahya} holds
and provide simple norm conditions sufficient for the existence of regular unitary dilations:

\begin{thm}[Sufficient norm conditions for regular unitary dilations]
\makelabel{thm:relative-norm-bound-generators-implies-dilation:sig:article-dilation-problem-raj-dahya}
    Let $d\in\naturalsPos$
    and $T$ a (necessarily contractive) $d$-parameter $\Cnought$-semigroup
    over a Hilbert space $\HilbertRaum$
    with bounded generators
        $A_{1},A_{2},\ldots,A_{d}$.
    If the generators satisfy
        $%
            \frac{
                \norm{A_{i} + \omega_{i}\cdot\onematrix}
            }{\omega_{i}}
            \leq 2^{1/d} - 1
        $
    for all $i\in\{1,2,\ldots,d\}$
    and some $\boldsymbol{\omega}=(\omega_{i})_{i=1}^{d}\in\realsPos^{d}$,
    then $T$ has a regular unitary dilation.
\end{thm}

In \S{}\ref{sec:examples}, we investigate complete (super) dissipativity for concrete classes of semigroups.
In the case of normal semigroups, the notions coincide
with the more basic properties of dissipativity and negative spectral bounds respectively.
And by working with a naturally definable class of non-doubly commuting generators,
we demonstrate that our notions are in general not equivalent to these properties.

The paper concludes in \S{}\ref{sec:applications}
with a non-trivial application of complete dissipativity to the von Neumann inequality problem.
We provide a natural generalisation of the problem to our context,
defining the
    \highlightTerm{regular polynomial bounds}
and the
    \highlightTerm{regular von Neumann polynomial inequality problem for multi-parameter $\Cnought$-semigroups}
(see \Cref{defn:polynomial-inequalities-semigroups:sig:article-dilation-problem-raj-dahya}).
We then establish a second characterisation of regular unitary dilations:

\begin{thm}[Classification via polynomial bounds]
\makelabel{thm:regular-polynomial-inequality-iff-completely-dissipative:sig:article-dilation-problem-raj-dahya}
    Let $d\in\naturalsPos$
    and $T$ be a $d$-parameter $\Cnought$-semigroup
    over a Hilbert space $\HilbertRaum$
    with bounded generators. Then the following are equivalent:

    \begin{kompaktenum}{\bfseries (1)}[\rtab]
        \item\punktlabel{1}
            The semigroup $T$ has a regular unitary dilation.
        \item\punktlabel{2}
            $T$ satisfies regular polynomial bounds.
        \item\punktlabel{3}
            $T$ satisfies regular polynomial bounds in a neighbourhood of $\zerovector$.
    \end{kompaktenum}

    \nvraum{1}

\end{thm}

Using this, we negatively solve the regular von Neumann polynomial inequality problem
for multi-parameter $\Cnought$-semigroups:

\begin{cor}
\makelabel{cor:counter-examples-regular-polynomial-inequality:sig:article-dilation-problem-raj-dahya}
    Let $\HilbertRaum$ be a Hilbert space with $\dim(\HilbertRaum) \geq 2$
    and let $d\in\naturals$ with $d \geq 2$.
    Then there exist $d$-parameter contractive $\Cnought$-semigroups
    with bounded generators which have strictly negative spectral bounds,\footnote{
        The \highlightTerm{spectral bound} of a linear operator
            $A:\dom(A)\subseteq\HilbertRaum\to\HilbertRaum$
        is given by
            $\sup\{\Re \lambda \mid \lambda\in\opSpectrum{A}\}$
        (\cf \cite[Definition~1.12]{EngelNagel2000semigroupTextBook}).
    }
    for which regular polynomial bounds fail.
\end{cor}




\section[Definitions]{Definitions}
\label{sec:definitions}


\firstparagraph
In this paper we fix the following notation.
Let $d\in\naturalsPos$.
Let
    $A,A_{1},A_{2},\ldots,A_{d}\in\BoundedOps{\HilbertRaum}$ be any bounded operators,
    let $C,C_{1},C_{2},K\subseteq\{1,2,\ldots,d\}$,
    and
    let $\pi=(\pi(i))_{i}$ be a (possibly empty) finite sequence of indices from $\{1,2,\ldots,d\}$.

\begin{kompaktitem}
    \item
        We write
            $\realsNonNeg = \{r\in\reals \mid r\geq 0\}$,
            $\realsPos = \{r\in\reals \mid r > 0\}$,
            and
            $\complex^{\times} = \complex \setminus \{0\}$.
        The set $\Torus$ denotes the unit circle in the complex plane $\{z\in\complex \mid \abs{z} = 1\}$.
    \item
        In any algebraic context empty sums and products shall always be taken
        to be the additive and multiplicative identities respectively.
    \item
        Depending on the context
        $\onematrix$ shall denote the identity operator on a space
        and
        $\zeromatrix$ denotes the zero operator on a space or between spaces.
    \item
        In the context of $\reals^{d}$
        let
            $\zerovector = (0,0,\ldots,0)\in\reals^{d}$
            and
            $\onevector = (1,1,\ldots,1)\in\reals^{d}$,
        and let
            $\mathbf{e}_{i} = (0,0,\ldots,\underset{i}{1},\ldots,0)$
        denote the canonical unit vectors for $i\in\{1,2,\ldots,d\}$.
        Also set
            $\mathbf{e}_{C} \colonequals \sum_{i\in C}\mathbf{e}_{i}$.
        For $\mathbf{t}=(t_{i})_{i=1}^{d}\in\reals^{d}$ let
            $%
                \mathbf{t}^{+}
                \colonequals (t_{i}^{+})_{i=1}^{d}
                = (\max\{t_{i},0\})_{i=1}^{d}
                \in \realsNonNeg^{d}
            $
        and
            $%
                \mathbf{t}^{-}
                \colonequals (t_{i}^{-})_{i=1}^{d}
                = (\max\{-t_{i},0\})_{i=1}^{d}
                \in \realsNonNeg^{d}
            $
        denote the positive and negative parts respectively.
        We further define the support
            ${\supp(\mathbf{t}) \colonequals \{i\in\{1,2,\ldots,d\} \mid t_{i} \neq 0\}}$
        as usual.
    \item
        For $\boldsymbol{\omega} = (\omega_{i})_{i=1}^{d} \in \reals^{d}$
        we shall denote $\omega_{K} \colonequals \prod_{i \in K}\omega_{i}$.
    \item
        $\Re A \coloneq \frac{1}{2}(A + A^{\ast})$
        and
        $\Imag A \coloneq \frac{1}{2\imagunit}(A - A^{\ast})$
        denote the self-adjoint operators, referred to as the real and imaginary parts of the operator $A$.
    \item
        We adopt the notation
            $A(\pi) \colonequals \prod_{i=1}^{|\pi|}A_{\pi(i)}$
        and
            $A(C) \colonequals \prod_{i \in C}A_{i}$
            where the order of the indices can be taken to be ascending.
        If these operators commute, then the order of multiplication is irrelevant.
    \item
        We use the notation $A_{i}^{-}$ for $-A_{i}$
        for each $i$
        and define $A^{-}(\pi)$ and $A^{-}(C)$ as above.
    \item
        We write $\isPartition{(C_{1},C_{2})}{K}$ to denote that $\{C_{1},C_{2}\}$ is a partition of $K$.
        We shall frequently compute sums over partitions of products.
        If $a_{1},a_{2},\ldots,a_{d}\in\cal{A}$ are commuting elements
        of some algebra $\cal{A}$ with unit,
        then using binomial expansion one obtains
            $%
                \sum_{\isPartition{(C_{1},C_{2})}{K}}\prod_{i\in C_{1}}a_{i}
                = \sum_{C \subseteq K}\prod_{i\in K}
                    \begin{cases}
                        a_{i} &: &i \in C\\
                        1 &: &i \notin C\\
                    \end{cases}
                = \prod_{i \in K}(1 + a_{i})
            $.
        This observation shall be repeatedly used in computations.
\end{kompaktitem}

To avoid confusion throughout this paper we shall consistently call
self-adjoint elements $a\in\cal{A}$ of a unital $C^{\ast}$-algebra $\cal{A}$
    \highlightTerm{positive}
    (in symbols $a \geq \zeromatrix$)
    if
        $a=b^{\ast}b$ for some $b\in\cal{A}$,
    which holds
    if and only if
        $\opSpectrum{a} \subseteq \realsNonNeg$.
We call $a$ \highlightTerm{strictly positive}
    if
        ${a - c\onematrix \geq \zeromatrix}$
        for some $c\in\realsPos$.
For self-adjoint bounded operators $A\in\BoundedOps{\HilbertRaum}$
one equivalently has that $A$ is
    \highlightTerm{positive}
    if and only if
        $\BRAKET{A\xi}{\xi} \geq 0$
        for all $\xi\in\HilbertRaum$
    and
    \highlightTerm{strictly positive}
    if and only if
        $\BRAKET{A\xi}{\xi} \geq c\norm{\xi}^{2}$
        for all $\xi\in\HilbertRaum$
        and some $c\in\realsPos$.\footnote{%
    The first property is also called \highlightTerm{positive semi-definite}.
}



\subsection[Notions of dilation]{Notions of dilation}
\label{sec:definitions:dilation}

\firstparagraph
As alluded to in the introduction, there is a natural correspondence between
commuting families of (contractive/isometric/unitary) $\Cnought$-semigroups
and multi-parameter (contractive/isometric/unitary) $\Cnought$-semigroups.
Working with the multi-parameter presentation,
we provide definitions of (regular) unitary dilations corresponding to those in the introduction.\footnote{
    Since a multi-parameter $\Cnought$-semigroup is a single object
    as opposed to a collection,
    we drop the term \usesinglequotes{simultaneous}.
}

\begin{defn}
    A \highlightTerm{unitary dilation} of $T$ is a tuple $(U,\HilbertRaum^{\prime},r)$,
    where
        $U$ is a $d$-parameter unitary $\Cnought$-semigroup over a Hilbert space $\HilbertRaum^{\prime}$
        and
        $r\in\BoundedOps{\HilbertRaum}{\HilbertRaum^{\prime}}$ (necessarily isometric),
    such that
        $%
            T(\mathbf{t}) = r^{\ast}U(\mathbf{t})r
        $
    holds for all $\mathbf{t}\in\realsNonNeg^{d}$.
\end{defn}

Note that any $d$-parameter unitary $\Cnought$-semigroup $U$ can
always be (uniquely) extended to an $\topSOT$-continuous representation of
    $(\reals^{d},+,\zerovector)$
via the definition $U(\mathbf{t})=U(\mathbf{t}^{-})^{\ast}U(\mathbf{t}^{+})$
for all $\mathbf{t}\in\reals^{d}$.
By capturing this behaviour,
one obtains the following stronger notion of dilation:

\begin{defn}
\makelabel{defn:dilation-regular:sig:article-dilation-problem-raj-dahya}
    A \highlightTerm{regular unitary dilation} of $T$ is a tuple $(U,\HilbertRaum^{\prime},r)$,
    where
        $U$ is an $\topSOT$-continuous unitary representation
            of $(\reals^{d},+,\zerovector)$
            over a Hilbert space $\HilbertRaum^{\prime}$
        and
        $r\in\BoundedOps{\HilbertRaum}{\HilbertRaum^{\prime}}$ (necessarily isometric),
    such that
        $%
            T(\mathbf{t}^{-})^{\ast}T(\mathbf{t}^{+})
            = r^{\ast}U(\mathbf{t})r
        $
    holds for all $\mathbf{t}\in\reals^{d}$.
\end{defn}

Clearly,
    the existence of a regular unitary dilation
    implies
    the existence of a unitary dilation.
This implication is strict
(\cf \S{}\ref{sec:examples:non-doubly-comm} and \Cref{rem:regular-dilation-stronger-than-unitary-dilation:sig:article-dilation-problem-raj-dahya}).

\begin{defn}
\makelabel{defn:brehmer-operators:sig:article-dilation-problem-raj-dahya}
    The self-adjoint \highlightTerm{Brehmer operators} associated to $T$
    shall be defined via

        \begin{maths}[mc]{c}
            B_{T,K}(t)
                \colonequals
                    \displaystyle
                    \sum_{C \subseteq K}
                        (-1)^{\card{C}}T(t\mathbf{e}_{C})^{\ast}T(t\mathbf{e}_{C})
        \end{maths}

    \continueparagraph
    for all $K\subseteq\{1,2,\ldots,d\}$ and $t\in\realsNonNeg$.
    We say that $T$ satisfies the \highlightTerm{Brehmer positivity criterion},
    if for sufficiently small $t\in\realsPos$
    it holds that $B_{T,K}(t) \geq \zeromatrix$
    for all $K\subseteq\{1,2,\ldots,d\}$.
\end{defn}

\begin{thm}[Ptak, 1985]
\makelabel{thm:brehmer-iff-dilation:sig:article-dilation-problem-raj-dahya}
    Let $d\in\naturalsPos$
    and $T$ be a $d$-parameter $\Cnought$-semigroup over $\HilbertRaum$.
    Then the following statements are equivalent:

    \begin{kompaktenum}{\bfseries (1)}[\rtab]
        \item $T$ has a regular unitary dilation.
        \item $T$ satisfies the Brehmer positivity criterion.
        \item $B_{T,K}(t) \geq \zeromatrix$
            for all $K\subseteq\{1,2,\ldots,d\}$
            and all  $t\in\realsPos$.
    \end{kompaktenum}

    \nvraum{1}

\end{thm}

A proof of this can be found in \cite[Theorem~3.2]{Ptak1985}
and is based on the more general characterisation of regular unitary dilations
of operator-valued functions defined on topological groups provided in \cite[Theorem~I.7.1]{Nagy1970},
which in turn is rooted in the correspondence between
    \usesinglequotes{positive definite functions}
    and group representations
(\cf \cite[Proposition~3.14, Theorem~3.20, and Proposition~3.35]{Folland2015bookHarmonicAnalysis}).
Note that the Brehmer positivity criterion in Ptak's paper is more weakly defined
in terms of a sequence of $t$-values converging to $0$.
However, in the proof it is stated that
    the existence of a regular unitary dilation
    implies that
    $B_{T,K}(t) \geq \zeromatrix$ for all $t\in\realsPos$
(\cf (a)~$\Rightarrow$~(b) in \cite[Theorem~3.2]{Ptak1985}).
Thus the above formulation is equivalent to Ptak's.

In general, it can be difficult to verify this condition,
as we need to consider the values of the semigroup at all points
(or at least at points close to $0$).
For semigroups with bounded generators, however,
\Cref{thm:classification:sig:article-dilation-problem-raj-dahya}
establishes a simpler condition purely in terms of the generators.

Note that in the discrete case,
the existence of simultaneous regular unitary dilations for families of commuting contractions
is equivalent to a similarly defined notion of Brehmer positivity
(\cf \cite[Theorem~I.9.1]{Nagy1970}).
In \cite{Barik2022} the authors inspect slightly weaker conditions
and construct simultaneous regular \emph{isometric} dilations
for families of commuting contractions satisfying this.

In addition to these definitions there are weaker variations of dilation
that one may consider.
These all revolve around dilations of certain natural modifications.
In the discrete setting,
    \idest for tuples of commuting bounded operators over $\HilbertRaum$,
there is the notion of $\rho$-dilations (see \exempli \cite[\S{}I.11]{Nagy1970}).
In the continuous setting of multi-parameter $\Cnought$-semigroups one has the following:

\begin{defn}
    Say that $T$ has a \highlightTerm{(regular) unitary dilation up to similarity}
    if it can be written as
        $
            T = s\,\tilde{T}(\cdot)\,s^{-1}
            \colonequals (
                s\,\tilde{T}(\mathbf{t})\,s^{-1}
            )_{\mathbf{t}\in\realsNonNeg^{d}}
        $
    for some invertible
        $s \in \BoundedOps{\HilbertRaum}$
    and some $d$-parameter $\Cnought$-semigroup $\tilde{T}$ over $\HilbertRaum$,
    where $\tilde{T}$ has a (regular) unitary dilation.
\end{defn}

\begin{defn}
    Say that $T$ has a \highlightTerm{(regular) unitary dilation up to exponential equivalence}
    if it can be written as
        $
            T = e^{\BRAKET{\cdot}{\boldsymbol{\omega}}} \tilde{T}(\cdot)
            \colonequals (
                e^{\BRAKET{\mathbf{t}}{\boldsymbol{\omega}}}
                \tilde{T}(\mathbf{t})
            )_{\mathbf{t}\in\realsNonNeg^{d}}
        $
    for some
        $\boldsymbol{\omega}\in\reals^{d}$
    and some $d$-parameter $\Cnought$-semigroup $\tilde{T}$ over $\HilbertRaum$,
    where $\tilde{T}$ has a (regular) unitary dilation.
\end{defn}

In \cite[Theorem~2.2]{LeMerdy1996DilMultiParam} Le~Merdy provides a characterisation of $\Cnought$-semigroups
which are similar to $\Cnought$-semigroups that have a unitary dilation.
The characterisation involves
demonstrating the complete boundedness
of a certain functional calculus map
induced by the resolvents of the generators of the marginal semigroups.
In \cite[Theorem~3.1]{LeMerdy1996DilMultiParam} $d$-parameter contractive $\Cnought$-semigroups
with bounded analytic marginal semigroups
are shown to satisfy this criterion.

By focussing on the smaller class of
    multi-parameter $\Cnought$-semigroups with bounded generators,
we obtain a considerably simpler condition.
This condition is not fulfilled by all multi-parameter $\Cnought$-semigroups,
as shall be demonstrated in \S{}\ref{sec:examples:non-doubly-comm}.
Nonetheless in a similar vein to \cite{LeMerdy1996DilMultiParam}, we demonstrate in
\Cref{cor:dilatable-up-to-exponential-modification:sig:article-dilation-problem-raj-dahya}
that regular unitary dilatability \emph{up to exponential equivalence}
always holds.



\subsection[Dissipativity conditions on the generators]{Dissipativity conditions on the generators}
\label{sec:definitions:dissipation}

\firstparagraph
It turns out that we can capture the existence of regular unitary dilations
in terms of a stronger but natural notion of dissipativity.
Recall that an operator $A\in\BoundedOps{\HilbertRaum}$ is called
\highlightTerm{dissipative}
if
    $\Re \BRAKET{A\xi}{\xi} \leq 0$ for all $\xi \in \HilbertRaum$.%
\footnote{
    Dissipativity is typically defined for unbounded operators
    (\cf \cite[\S{}VI.5~(8)]{Yosida1995fa}),
    but we shall not require this level of generality.
}

\begin{rem}
\makelabel{rem:lumer-phillips:sig:article-dilation-problem-raj-dahya}
    Let $A\in\BoundedOps{\HilbertRaum}$.
    Since the spectrum $\opSpectrum{A}$ is bounded,
    the resolvent set satisfies $\opResolventSet{A}\cap\realsPos\neq\emptyset$.
    The Lumer-Phillips form of the Hille-Yosida theorem
    (see \cite[Theorem~3.3]{Goldstein1985semigroups})
    thus yields that $A$
    is the generator of a contractive $\Cnought$-semigroup
    if and only if $A$ is dissipative.
\end{rem}

It shall be convenient to furthermore call $A\in\BoundedOps{\HilbertRaum}$
\highlightTerm{super dissipative}
if
    $\Re \BRAKET{A\xi}{\xi} \leq -\omega\norm{\xi}^{2}$ for all $\xi \in \HilbertRaum$
    and some $\omega\in\realsPos$.
Now, as an alternative formulation of these terms,
observe that $A$ is dissipative
if and only if $\Re A \leq \zeromatrix$ and $A$ is super dissipative
if and only if $\Re A \leq -\omega\onematrix$
for some $\omega\in\realsPos$.
Based on these observations, we introduce the following generalised notions:

\begin{defn}
\makelabel{defn:dissipation-operators-and-conditions:sig:article-dilation-problem-raj-dahya}
    The self-adjoint \highlightTerm{dissipation operators}
    associated to the generators $A_{1},A_{2},\ldots,A_{d}$ of $T$
    shall be defined by

        \begin{maths}[mc]{c}
            S_{T,K}
                \colonequals
                    2^{-\card{K}}
                    \displaystyle
                    \sum_{\isPartition{(C_{1},C_{2})}{K}}
                        A^{-}(C_{1})^{\ast}A^{-}(C_{2})
        \end{maths}

    \continueparagraph
    for all $K\subseteq\{1,2,\ldots,d\}$.
    For $k\in\naturalsZero$ we shall refer to
        $S_{T,K}$ for $K\subseteq\{1,2,\ldots,d\}$ with $\card{K}=k$
    as the \highlightTerm{$k$\textsuperscript{th} order dissipation operators}.
    Setting

        \begin{maths}[mc]{c}
            \beta_{T}
                \colonequals
                    \displaystyle
                    \min_{K \subseteq \{1,2,\ldots,d\}}
                        \opSpectrum{S_{T,K}},
        \end{maths}

    \continueparagraph
    we say that the generators of $T$ are
    \highlightTerm{completely dissipative} if $\beta_{T} \geq 0$
    (equivalently: $S_{T,K} \geq \zeromatrix$ for each $K\subseteq\{1,2,\ldots,d\}$)
    and \highlightTerm{completely super dissipative} if $\beta_{T} > 0$.
\end{defn}

Observe that for $K=\emptyset$ one has
    $
        S_{T,\emptyset}
        = 2^{-0}A^{-}(\emptyset)^{\ast}A^{-}(\emptyset)
        = \onematrix
    $,
since $\{\emptyset,\emptyset\}$ is the only partition of $K$.
And for $K = \{\alpha\}$ for some $\alpha\in\{1,2,\ldots,d\}$
one has

    \begin{maths}[mc]{c}
        S_{T,\{\alpha\}}
        = 2^{-1}(A^{-}(\{\alpha\})^{\ast}A^{-}(\emptyset) + A^{-}(\emptyset)^{\ast}A^{-}(\{\alpha\}))
        = 2^{-1}((A_{\alpha}^{-})^{\ast}\cdot \onematrix + \onematrix \cdot A_{\alpha}^{-})
        = -\Re A_{\alpha},\\
    \end{maths}

\continueparagraph
since $\{\{\alpha\},\emptyset\}$ is the only partition of $K$.
So if $T$ has completely dissipative generators,
then for each $\alpha\in\{1,2,\ldots,d\}$
    $
        0 \leq \beta_{T}
        \leq \min(\opSpectrum{-\Re A_{\alpha}})
        = -\max(\opSpectrum{\Re A_{\alpha}})
    $,
which, by the correspondence between numerical ranges and spectra for self-adjoint operators
(\cf \cite[Theorem~1.2.1--4]{Gustafson1997numericalRange}), implies
    $%
        \Re \BRAKET{A_{\alpha}\xi}{\xi}
        = \BRAKET{(\Re A_{\alpha})\xi}{\xi}
        \leq 0
    $
    for all $\xi\in\HilbertRaum$,
\idest $A_{\alpha}$ is dissipative.
In a similar way, we see that the complete super dissipativity of the generators of $T$
implies that each of the generators must be super dissipative.
Furthermore, if the generators are normal,
then super dissipativity is equivalent to strict negativity of the spectral bounds.

Hence complete dissipativity extends the notion of dissipativity,
and complete super dissipativity extends the notion of super dissipativity
as well as, in the case of normal generators, strict negativity of the spectral bounds.
In \S{}\ref{sec:examples:non-doubly-comm} it shall be shown
that complete (super) dissipativity
is a strictly stronger property
than (super) dissipativity.




\section[Algebraic identities]{Algebraic identities}
\label{sec:algebra}


\firstparagraph
We now lay out some usual algebraic identities for the computation of dissipation operators.
To start, we observe the following basic recursive relations.
For
    $t\in\realsNonNeg$,
    $K \subseteq \{1,2,\ldots,d\}$,
    and
    $\alpha \in \{1,2,\ldots,d\} \without K$
one has

    \begin{maths}[mc]{rcl}
    \eqtag[eq:recursion:brehmer:sig:article-dilation-problem-raj-dahya]
        B_{T,\emptyset}(t)
            &= &(-1)^{0}T(\zerovector)^{\ast}T(\zerovector)
            = 1\cdot \onematrix^{\ast}\onematrix
            = \onematrix~\text{and}\\
        B_{T,K\cup\{\alpha\}}(t)
            &= &\begin{array}[t]{0l}
                    \sum_{C \subseteq K}
                        (-1)^{\card{C \cup \{\alpha\}}}
                        T(t \mathbf{e}_{C \cup \{\alpha\}})^{\ast}
                        T(t \mathbf{e}_{C \cup \{\alpha\}})\\
                    + \sum_{C \subseteq K}
                        (-1)^{\card{C}}
                        T(t \mathbf{e}_{C})^{\ast}
                        T(t \mathbf{e}_{C})\\
                \end{array}\\
            &= &\begin{array}[t]{0l}
                    \sum_{C \subseteq K}
                        (-1)^{\card{C} + 1}
                        T_{\alpha}(t)^{\ast}T(t \mathbf{e}_{C})^{\ast}
                        T(t \mathbf{e}_{C})T_{\alpha}\\
                    + \sum_{C \subseteq K}
                        (-1)^{\card{C}}
                        T(t \mathbf{e}_{C})^{\ast}
                        T(t \mathbf{e}_{C})\\
                \end{array}\\
            &= &B_{T,K}(t) - T_{\alpha}(t)^{\ast}B_{T,K}(t)T_{\alpha}(t)\\
    \end{maths}

\continueparagraph
as well as

    \begin{maths}[mc]{rcl}
    \eqtag[eq:recursion:dissipation:sig:article-dilation-problem-raj-dahya]
        S_{T,\emptyset}
            &= &2^{-0}
                A^{-}(\emptyset)^{\ast}A^{-}(\emptyset)
            = \onematrix^{\ast} \onematrix
            = \onematrix,~\text{and}\\
        S_{T,K\cup\{\alpha\}}
        &= &2^{-\card{K \cup \{\alpha\}}}
            \sum_{
                \isPartition{%
                    (C_{1},C_{2})
                }{%
                    K \cup \{\alpha\}
                }
            }
                A^{-}(C_{1})^{\ast}A^{-}(C_{2})\\
        &= &\begin{array}[t]{l}
            2^{-\card{K}-1}
            \sum_{\isPartition{(C_{1},C_{2})}{K}}
                A^{-}(C_{1}\cup\{\alpha\})^{\ast}A^{-}(C_{2})\\
            + 2^{-\card{K}-1}
            \sum_{\isPartition{(C_{1},C_{2})}{K}}
                A^{-}(C_{1})^{\ast}A^{-}(C_{2}\cup\{\alpha\})\\
            \end{array}\\
        &= &\begin{array}[t]{l}
            \frac{1}{2}
            2^{-\card{K}}
            \sum_{\isPartition{(C_{1},C_{2})}{K}}
                (A_{\alpha}^{-})^{\ast}A^{-}(C_{1})^{\ast}A^{-}(C_{2})\\
            + \frac{1}{2}
            2^{-\card{K}}
            \sum_{\isPartition{(C_{1},C_{2})}{K}}
                A^{-}(C_{1})^{\ast}A^{-}(C_{2})A^{-}_{\alpha}\\
            \end{array}\\
        &=
            &-\frac{1}{2}(A_{\alpha}^{\ast}S_{T,K} + S_{T,K}A_{\alpha})\\
        &=
            &\frac{1}{2}\{-\Re A_{\alpha},\:S_{T,K}\}
            + \imagunit \frac{1}{2}[\Imag A_{\alpha},\:S_{T,K}],\\
    \end{maths}

\continueparagraph
where $[\cdot,\:\cdot]$ und $\{\cdot,\:\cdot\}$
denote the algebraic \emph{commutator} and \emph{anti-commutator} operations.\footnote{
    For any algebra $\cal{A}$, which is at least a $\integers$-module,
    and for any $a,b\in\cal{A}$ one defines
    $[a,\:b] \colonequals ab - ba$
    and
    $\{a,\:b\} \colonequals ab + ba$.
}
Note that the above recursive relations for $B_{T,\cdot}(\cdot)$
also hold for semigroups with unbounded generators.



\subsection[Dissipation operators under shifts]{Dissipation operators under shifts}
\label{sec:algebra:basic}

\firstparagraph
Shifting the generators by constant multiples of the identity
lead to self-similarities in the dissipation operators.
In this subsection we consider an arbitrary $d$-parameter $\Cnought$-semigroup $T$
with marginals
    $T_{1},T_{2},\ldots,T_{d}$
which have bounded generators
    $A_{1},A_{2},\ldots,A_{d}$
on which we place no further assumptions.
For arbitrary $\boldsymbol{\omega} \in \reals^{d}$
we define the $1$-dimensional continuous representation
    $E_{\boldsymbol{\omega}} \colonequals (e^{\BRAKET{\mathbf{t}}{\boldsymbol{\omega}}})_{\mathbf{t}\in\reals^{d}}$
of the group $(\reals^{d},+,\zerovector)$
and define the modification

    \begin{maths}[mc]{rcl}
        E_{\boldsymbol{\omega}} \cdot T
            &\colonequals
                &(
                    e^{
                        -\BRAKET{\mathbf{t}}{\boldsymbol{\omega}}
                    }
                    T(\mathbf{t})
                )_{\mathbf{t} \in \realsNonNeg^{d}},\\
    \end{maths}

\continueparagraph
which is clearly a $d$-parameter $\Cnought$-semigroup over $\HilbertRaum$
whose marginals have bounded generators
    $\tilde{A}_{i} = A_{i} - \omega_{i}\onematrix$
for each $i\in\{1,2,\ldots,d\}$.

\begin{defn}
    \makelabel{defn:scale-shift:sig:article-dilation-problem-raj-dahya}
    Call the $d$-parameter $\Cnought$-semigroup $E_{\boldsymbol{\omega}} \cdot T$
    the \highlightTerm{$\boldsymbol{\omega}$-shift} of $T$.
\end{defn}

\begin{prop}[Self-similarity under shifts]
\makelabel{prop:algebraic-identity-scaled-shifts:sig:article-dilation-problem-raj-dahya}
    Let $T$ be a $d$-parameter $\Cnought$-semigroup over $\HilbertRaum$
    with bounded generators $A_{1},A_{2},\ldots,A_{d}$.
    Further consider the $\boldsymbol{\omega}$-shift $E_{\boldsymbol{\omega}} \cdot T$ of $T$
    for some $\boldsymbol{\omega} \in \reals^{d}$.
    Then

        \begin{maths}[mc]{c}
        \eqtag[eq:self-similarity-dissipation-operators:sig:article-dilation-problem-raj-dahya]
            S_{E_{\boldsymbol{\omega}} \cdot T, K}
                =
                    \displaystyle
                    \sum_{K^{\prime}\subseteq K}
                        \omega_{K \without K^{\prime}}
                        S_{T,K^{\prime}}\\
        \end{maths}

    \continueparagraph
    holds for all $K \subseteq \{1,2,\ldots,d\}$.
\end{prop}

    \begin{proof}
        We prove this by induction over the size of $K$
        and using the recursion laid out in \eqcref{eq:recursion:dissipation:sig:article-dilation-problem-raj-dahya}.
        For $K = \emptyset$, the left- and right-hand sides of \eqcref{eq:self-similarity-dissipation-operators:sig:article-dilation-problem-raj-dahya}
        are both $\onematrix$ and thus equal.
        Let $K \subseteq \{1,2,\ldots,d\}$ and $\alpha\in\{1,2,\ldots,d\} \without K$.
        Assume that \eqcref{eq:self-similarity-dissipation-operators:sig:article-dilation-problem-raj-dahya}
        holds for $K$. Then

            \begin{longmaths}[mc]{RCL}
                S_{E_{\boldsymbol{\omega}} \cdot T, K \cup \{\alpha\}}
                    &\eqcrefoverset{eq:recursion:dissipation:sig:article-dilation-problem-raj-dahya}{=}
                        &-\frac{1}{2}(
                            \tilde{A}_{\alpha}S_{E_{\boldsymbol{\omega}} \cdot T, K}
                            + S_{E_{\boldsymbol{\omega}} \cdot T, K}\tilde{A}_{\alpha}^{\ast}
                        )\\
                    &=
                        &-\frac{1}{2}(
                            (A_{\alpha} - \omega_{\alpha}\onematrix)
                            S_{E_{\boldsymbol{\omega}} \cdot T, K}
                            + S_{E_{\boldsymbol{\omega}} \cdot T, K}
                                (A_{\alpha}^{\ast} - \omega_{\alpha}\onematrix)
                        )\\
                    &\textoverset{ind.}{=}
                        &-\frac{1}{2}
                        \displaystyle
                        \sum_{K^{\prime} \subseteq K}
                            \omega_{K \without K^{\prime}}
                            \cdot
                            (
                                (A_{\alpha} - \omega_{\alpha}\onematrix)
                                S_{T,K^{\prime}}
                                +
                                S_{T,K^{\prime}}
                                (A_{\alpha}^{\ast} - \omega_{\alpha}\onematrix)
                            )\\
                    &=
                        &\displaystyle
                        \sum_{K^{\prime} \subseteq K}
                            \omega_{K \without K^{\prime}}\omega_{\alpha}
                                S_{T,K^{\prime}}
                        -
                        \frac{1}{2}
                        \displaystyle
                        \sum_{K^{\prime}\subseteq K}
                            \omega_{K \without K^{\prime}}
                            \cdot
                            (
                                A_{\alpha}
                                S_{T,K^{\prime}}
                                +
                                S_{T,K^{\prime}}
                                A_{\alpha}^{\ast}
                            )\\
                    &\eqcrefoverset{eq:recursion:dissipation:sig:article-dilation-problem-raj-dahya}{=}
                        &\displaystyle
                        \sum_{K^{\prime}\subseteq K}
                            \omega_{K \without K^{\prime}}\omega_{\alpha}
                                S_{T,K^{\prime}}
                        +
                        \displaystyle
                        \sum_{K^{\prime}\subseteq K}
                            \omega_{K \without K^{\prime}}
                                S_{T,K^{\prime}\cup\{\alpha\}}\\
                    &=
                        &\displaystyle
                        \sum_{K^{\prime}\subseteq K}
                            \omega_{(K\cup\{\alpha\}) \without K^{\prime}}
                                S_{T,K^{\prime}}
                        +
                        \displaystyle
                        \sum_{K^{\prime}\subseteq K}
                            \omega_{(K\cup\{\alpha\}) \without (K^{\prime}\cup\{\alpha\})}
                                S_{T,K^{\prime}\cup\{\alpha\}}\\
                    &=
                        &\displaystyle
                        \sum_{K^{\prime} \subseteq K\cup\{\alpha\}}
                            \omega_{(K\cup\{\alpha\}) \without K^{\prime}}
                                S_{T,K^{\prime}},\\
            \end{longmaths}

        \continueparagraph
        \idest \eqcref{eq:self-similarity-dissipation-operators:sig:article-dilation-problem-raj-dahya}
        holds for $K\cup\{\alpha\}$.
        Hence the claim holds by induction.
    \end{proof}



\subsection[Dissipation operators under inversions]{Dissipation operators under inversions}
\label{sec:algebra:inverse}

\firstparagraph
This subsection is independent of the rest of the paper,
and provides a connection to inverse infinitesimal generators,
which are of relatively recent interest in semigroup theory
in particular in the unbounded case (\cf \cite{GomilkoZwart2007invGenRu,Zwart2007invGenEn}).

Let $P\subseteq\{1,2,\ldots,d\}$.
We shall assume that the generators $A_{1},A_{2},\ldots,A_{d}$ of $T$ are bounded,
and that $A_{i}$ is invertible for each $i \in P$.
We then define
    $A_{P,i} \colonequals A_{i}^{-1}$ for $i\in P$
and
    $A_{P,i} \colonequals A_{i}$ for $i\in\{1,2,\ldots,d\} \without P$.
Since the $A_{i}$ commute, one clearly has that the $A_{P,i}$ commute.
We can thus define the $d$-parameter $\Cnought$-semigroup

    \begin{maths}[mc]{rcl}
        T_{P}
            \colonequals
                (
                    e^{
                        \sum_{i=1}^{d}
                        t_{i}A_{P,i}
                    }
                )_{\mathbf{t}\in\realsNonNeg^{d}}\\
    \end{maths}

\continueparagraph
over $\HilbertRaum$, which has $A_{P,1},A_{P,2},\ldots,A_{P,d}$ as its generators.

\begin{defn}
    Call $T_{P}$ the \highlightTerm{$P$-inversion} associated to the semigroup $T$.
\end{defn}

Observe that for $K\subseteq\{1,2,\ldots,d\}$

    \begin{longmaths}[t]{CL}
            &A^{-}(K \cap P)^{\ast}
            S_{T_{P},K}
            A^{-}(K \cap P)\\
        = &2^{-\card{K}}
            \displaystyle
            \sum_{\isPartition{(C_{1},C_{2})}{K}}
                A^{-}(K \cap P)^{\ast}
                \Big(
                    \displaystyle
                    \prod_{i\in C_{1}}A_{P,i}^{-}
                \Big)^{\ast}
                \Big(
                    \displaystyle
                    \prod_{i\in C_{2}}A_{P,i}^{-}
                \Big)
                A^{-}(K \cap P)\\
        = &2^{-\card{K}}
            \displaystyle
            \sum_{\isPartition{(C_{1},C_{2})}{K}}
            \begin{array}[t]{0l}
                \Big(
                    \displaystyle
                    \prod_{i\in C_{1} \without P}A^{-}_{i}
                    \displaystyle
                    \prod_{i\in C_{1} \cap P}(A^{-}_{i})^{-1}
                    \displaystyle
                    \prod_{i \in K \cap P}A^{-}_{i}
                \Big)^{\ast}\\
                \cdot\Big(
                    \displaystyle
                    \prod_{i\in C_{2} \without P}A^{-}_{i}
                    \displaystyle
                    \prod_{i \in C_{2} \cap P}(A^{-}_{i})^{-1}
                    \displaystyle
                    \prod_{i\in K \cap P}A^{-}_{i}
                \Big)\\
            \end{array}\\
        = &2^{-\card{K}}
            \displaystyle
            \sum_{\isPartition{(C_{1},C_{2})}{K}}
                \Big(
                    \displaystyle
                    \prod_{i\in C_{1} \without P}A^{-}_{i}
                    \displaystyle
                    \prod_{i \in (K \without C_{1}) \cap P}A^{-}_{i}
                \Big)^{\ast}
                \cdot\Big(
                    \displaystyle
                    \prod_{i\in C_{2} \without P}A^{-}_{i}
                    \displaystyle
                    \prod_{i\in (K \without C_{2}) \cap P}A^{-}_{i}
                \Big)\\
        = &2^{-\card{K}}
            \displaystyle
            \sum_{\isPartition{(C_{1},C_{2})}{K}}
                \Big(
                    \displaystyle
                    \prod_{i\in C_{1} \without P}A^{-}_{i}
                    \displaystyle
                    \prod_{i \in C_{2} \cap P}A^{-}_{i}
                \Big)^{\ast}
                \cdot\Big(
                    \displaystyle
                    \prod_{i\in C_{2} \without P}A^{-}_{i}
                    \displaystyle
                    \prod_{i\in C_{1} \cap P}A^{-}_{i}
                \Big)\\
        = &2^{-\card{K}}
            \displaystyle
            \sum_{\isPartition{(C_{1},C_{2})}{K}}
                A^{-}((C_{1} \without P)\cup(C_{2} \cap P))^{\ast}
                A^{-}((C_{2} \without P)\cup(C_{1} \cap P))\\
        = &2^{-\card{K}}
            \displaystyle
            \sum_{\isPartition{(\tilde{C}_{1},\tilde{C}_{2})}{K}}
                A^{-}(\tilde{C}_{1})^{\ast}
                A^{-}(\tilde{C}_{2})
        \textoverset{Defn}{=}
            S_{T,K},\\
    \end{longmaths}

\continueparagraph
where equality between the final lines hold,
since
  $(C_{1},\:C_{2})\mapsto((C_{1} \without P)\cup(C_{2} \cap P),\:(C_{2} \without P)\cup(C_{1} \cap P))$
clearly constitutes a bijection on the set of partitions of $K$.

Due to the above relation and the invertibility of the generators,
if $\beta_{T_{P}} \geq 0$, it follows that

    \begin{maths}[mc]{rcl}
        \BRAKET{S_{T,K}\xi}{\xi}
            &=
                &\BRAKET{
                    A^{-}(K\cap P)^{\ast}S_{T_{P},K}A^{-}(K\cap P)
                    \xi
                }{\xi}\\
            &=
                &\BRAKET{
                    S_{T_{P},K}
                    \,A^{-}(K\cap P)\xi
                }{A^{-}(K\cap P)\xi}\\
            &\geq
                &\min \opSpectrum{S_{T_{P},K}} \cdot \norm{A^{-}(K\cap P)\xi}^{2}\\
            &\geq
                &\beta_{T_{P}} \cdot \norm{A^{-}(K\cap P)\xi}^{2}\\
            &\geq
                &\beta_{T_{P}} \cdot \norm{A^{-}(K\cap P)^{-1}}^{-2}\norm{\xi}^{2}\\
    \end{maths}

\continueparagraph
for all $\xi\in\HilbertRaum$ and $K\subseteq\{1,2,\ldots,d\}$.
Hence we have shown

    \begin{maths}[mc]{c}
    \eqtag[eq:p-inversion-inequality:sig:article-dilation-problem-raj-dahya]
        \beta_{T_{P}} \geq 0
            \Longrightarrow
                \beta_{T}
                \geq
                    \dfrac{
                        \beta_{T_{P}}
                    }{
                        \displaystyle
                        \max_{K \subseteq \{1,2,\ldots,d\}}
                            \norm{A^{-}(K\cap P)^{-1}}^{2}
                    }
                =
                    \dfrac{
                        \beta_{T_{P}}
                    }{
                        \displaystyle
                        \max_{K \subseteq P}
                            \norm{A^{-}(K)^{-1}}^{2}
                    }
                \geq 0\\
    \end{maths}

\continueparagraph
and, since $(T_{P})_{P}=T$ clearly holds, it follows analogously that

    \begin{maths}[mc]{c}
    \eqtag[eq:p-inversion-dual-inequality:sig:article-dilation-problem-raj-dahya]
        \beta_{T} \geq 0
            \Longrightarrow
                \beta_{T_{P}}
                \geq
                    \dfrac{
                        \beta_{(T_{P})_{P}}
                    }{
                        \displaystyle
                        \max_{K \subseteq P}
                           \norm{A_{P}^{-}(K)^{-1}}^{2}
                    }
                =
                    \dfrac{
                        \beta_{T}
                    }{
                        \displaystyle
                        \max_{K \subseteq P}
                            \norm{A^{-}(K)}^{2}
                    }
                \geq 0.\\
    \end{maths}

The implications and inequalities in
    \eqcref{eq:p-inversion-inequality:sig:article-dilation-problem-raj-dahya}
    and
    \eqcref{eq:p-inversion-dual-inequality:sig:article-dilation-problem-raj-dahya}
yield
    $\beta_{T}\geq 0$ if and only if $\beta_{T_{P}}\geq 0$,
as well as
    $\beta_{T} > 0$ if and only if $\beta_{T_{P}} > 0$.
Putting this together yields

\begin{prop}
    Let $T$ be a $d$-parameter $\Cnought$-semigroup over $\HilbertRaum$
    with bounded generators $A_{1},A_{2},\ldots,A_{d}$.
    Further let $P \subseteq \{1,2,\ldots,d\}$
    and assume that $A_{i}$ is invertible for $i \in P$.
    Then $T$ is completely dissipative (\respectively completely super dissipative)
    if and only if $T_{P}$ is.
\end{prop}




\section[Main results]{Main results}
\label{sec:results}


\firstparagraph
Our goal is to relate regular unitary dilatability to the notions of complete (super) dissipativity.
To achieve this, we first establish relations between the Brehmer and dissipation operators.



\begin{prop}
\makelabel{prop:asymptotic-behaviour-brehmer-ops:sig:article-dilation-problem-raj-dahya}
    Let $T$ be a $d$-parameter $\Cnought$-semigroup over $\HilbertRaum$
    with bounded generators $A_{1},A_{2},\ldots,A_{d}$.
    Then for all $t\in\realsNonNeg$ and $K\subseteq\{1,2,\ldots,d\}$

        \begin{maths}[mc]{rcl}
        \eqtag[eq:brehmer-dissipation-operators:sig:article-dilation-problem-raj-dahya]
            B_{T,K}(t)
                &= &2^{\card{K}}t^{\card{K}}S_{T,K}
                    + (-1)^{\card{K}}t^{\card{K}+1}\Delta_{T,K}(t),\\
        \end{maths}

    \continueparagraph
    where $\Delta_{T,K}$ is an operator-valued function which is norm bounded on compact subsets of $\realsNonNeg$.
\end{prop}

    \begin{proof}
        Let $K\subseteq\{1,2,\ldots,d\}$ and $t\in\realsNonNeg$ be arbitrary.
        Relying on the boundedness of the generators
        we can expand the expression defining the Brehmer operator $B_{T,K}(t)$
        in terms of an absolute convergent power series:

            \begin{longmaths}[mc]{RCL}
                B_{T,K}(t)
                    &= &\displaystyle
                        \sum_{C\subseteq K}
                            (-1)^{\card{C}}
                            T(t\mathbf{e}_{C})^{\ast}
                            T(t\mathbf{e}_{C})\\
                    &= &\displaystyle
                        \sum_{C \subseteq K}
                            (-1)^{\card{C}}
                            \Big(
                                e^{t\sum_{i \in C}A_{i}}
                            \Big)^{\ast}
                            e^{t\sum_{i \in C}A_{i}}\\
                    &= &\displaystyle
                        \sum_{k,l\in\naturalsZero}
                        \frac{t^{k+l}}{k!l!}
                        \displaystyle
                            \sum_{C\subseteq K}
                            (-1)^{\card{C}}
                            \Big(
                                \sum_{i \in C}A_{i}^{\ast}
                            \Big)^{k}
                            \Big(
                                \sum_{i \in C}A_{i}
                            \Big)^{l}\\
                    &= &\displaystyle
                        \sum_{k,l\in\naturalsZero}
                        \frac{t^{k+l}}{k!l!}
                        \displaystyle
                            \sum_{C\subseteq K}
                        \sum_{\substack{
                            (\pi_{1},\pi_{2})\in K^{k}\times K^{l}:\hraum\\
                            \ran(\pi_{1}) \cup \ran(\pi_{2}) \subseteq C\hraum
                        }}
                            (-1)^{\card{C}}
                            \cdot
                            \underbrace{
                                \Big(
                                    \displaystyle
                                    \prod_{i=1}^{k}
                                    A_{\pi_{1}(i)}
                                \Big)^{\ast}
                                \displaystyle
                                \prod_{j=1}^{l}
                                    A_{\pi_{2}(j)}
                            }_{= A(\pi_{1})^{\ast}A(\pi_{2})}\\
                    &= &\displaystyle
                        \sum_{k,l\in\naturalsZero}
                        \frac{t^{k+l}}{k!l!}
                        \displaystyle
                        \sum_{(\pi_{1},\pi_{2})\in K^{k}\times K^{l}}
                        \underbrace{
                            \displaystyle\sum_{\substack{
                                C\subseteq K:\\
                                C \supseteq \ran(\pi_{1}) \cup \ran(\pi_{2})
                            }}
                                (-1)^{\card{C}}
                        }_{(\ast)}
                        A(\pi_{1})^{\ast}A(\pi_{2}).\\
            \end{longmaths}

        We can simplify ($\ast$) using the binomial expansion.
        Observe for any $P \subseteq K$ that
            $%
                \sum_{C\subseteq K:~C \supseteq P}
                    (-1)^{\card{C}}
                =
                    \sum_{C^{\prime} \subseteq K \setminus P}
                        (-1)^{\card{C^{\prime} \cup P}}
                =
                    (-1)^{\card{P}}
                    \sum_{C^{\prime}\subseteq K \setminus P}
                        (-1)^{\card{C^{\prime}}}
            $,
        which equals
            ${(-1)^{\card{P}}\cdot (1 + (-1))^{\card{K \setminus P}}} = 0$
            if $P \neq K$
        and otherwise
            $(-1)^{\card{K}}$
            if $P=K$.
        Thus the sum ($\ast$) vanishes except when
            $\ran(\pi_{1}) \cup \ran(\pi_{2}) = K$,
        in which case it equals $(-1)^{\card{K}}$.
        Applying this to the above computation thus yields:

            \begin{longmaths}[mc]{RCL}
                B_{T,K}(t)
                    &= &(-1)^{\card{K}}
                        \displaystyle\sum_{k,l\in\naturalsZero}
                        \displaystyle\sum_{\substack{
                            (\pi_{1},\pi_{2})\in K^{k}\times K^{l}:\hraum\\
                            \ran(\pi_{1})\cup \ran(\pi_{2})=K\hraum
                        }}
                            \frac{t^{k+l}}{k!l!}
                            A(\pi_{1})^{\ast}A(\pi_{2})\\
                    &= &(-1)^{\card{K}}
                        \displaystyle\sum_{\substack{
                            k,l\in\naturalsZero:\hraum\\
                            k+l \geq \card{K}\hraum
                        }}
                        \displaystyle\sum_{\substack{
                            (\pi_{1},\pi_{2})\in K^{k}\times K^{l}:\hraum\\
                            \ran(\pi_{1})\cup \ran(\pi_{2})=K\hraum
                        }}
                            \frac{t^{k+l}}{k!l!}
                            A(\pi_{1})^{\ast}A(\pi_{2})\\
                    &= &\begin{array}[t]{0l}
                        t^{\card{K}}
                        \underbrace{
                            \displaystyle\sum_{\substack{
                                k,l\in\naturalsZero:\hraum\\
                                k+l = \card{K}\hraum
                            }}
                                \frac{1}{k!l!}
                            \displaystyle\sum_{\substack{
                                (\pi_{1},\pi_{2})\in K^{k}\times K^{l}:\hraum\\
                                \ran(\pi_{1})\cup \ran(\pi_{2})=K\hraum
                            }}
                                A^{-}(\pi_{1})^{\ast}A^{-}(\pi_{2})
                        }_{\equalscolon B_{T,K}^{0}}\\
                        + (-1)^{\card{K}}t^{\card{K}+1}
                        \underbrace{
                            \displaystyle\sum_{n=\card{K}+1}^{\infty}
                            \displaystyle\sum_{\substack{
                                k,l\in\naturalsZero:\hraum\\
                                k+l = n\hraum
                            }}
                            \displaystyle\sum_{\substack{
                                (\pi_{1},\pi_{2})\in K^{k}\times K^{l}:\hraum\\
                                \ran(\pi_{1})\cup \ran(\pi_{2})=K\hraum
                            }}
                                \frac{t^{n-\card{K}-1}}{k!l!}
                                A(\pi_{1})^{\ast}A(\pi_{2})
                        }_{\equalscolon \Delta_{T,K}}.\\
                    \end{array}\\
            \end{longmaths}

        To complete the proof, it remains to simplify the first sum in the last expression,
        and to show that $\Delta_{T,K}$ is norm-bounded on compact subsets of $\realsNonNeg$.

        Observe that for $k,l\in\naturalsZero$ with $k+l=\card{K}$
        a pair of sequences
            $(\pi_{1},\pi_{2})\in K^{k}\times K^{l}$
        satisfy
            $\ran(\pi_{1})\cup \ran(\pi_{2})=K$
        if and only if $\pi_{1},\pi_{2}$
        are injective functions mapping to the respective parts of a partition
            $(C_{1},C_{2})$ of $K$
        with $\card{C_{1}}=k$ and $\card{C_{2}}=l$.
        In this case, by commutativity of the generators, one has that
            $A^{-}(\pi_{1})=A^{-}(C_{1})$ and $A^{-}(\pi_{2})=A^{-}(C_{2})$.
        Conversely, for any partition $(C_{1},C_{2})$ of $K$
        setting $k=\card{C_{1}}$ and $l=\card{C_{2}}$
        there exist
            $k!$ sequences $\pi_{1} \in K^{k}$ with $\ran(\pi_{1})=C_{1}$
        and
            $l!$ sequences $\pi_{2} \in K^{l}$ with $\ran(\pi_{2})=C_{2}$.
        Hence $B_{T,K}^{0}$ simplifies to

            \begin{longmaths}[mc]{RCL}
                B_{T,K}^{0}
                    &= &\displaystyle
                        \displaystyle\sum_{\substack{
                            k,l\in\naturalsZero:\hraum\\
                            k+l = \card{K}\hraum
                        }}
                            \frac{1}{k!l!}
                        \displaystyle
                        \sum_{\substack{
                            \isPartition{(C_{1},C_{2})}{K}:\hraum\\
                            \card{C_{1}}=k,~\card{C_{2}}=l\hraum
                        }}
                            k!l! \cdot A^{-}(C_{1})^{\ast}A^{-}(C_{2})\\
                    &= &\displaystyle
                        \sum_{\isPartition{(C_{1},C_{2})}{K}}
                            A^{-}(C_{1})^{\ast}A^{-}(C_{2})\\
                    &= &2^{\card{K}}S_{T,K},\\
            \end{longmaths}

        \continueparagraph
        and thus when inserted into the above computation,
        one sees that
        \eqcref{eq:brehmer-dissipation-operators:sig:article-dilation-problem-raj-dahya}
        holds.

        For the second term, using the boundedness of the generators,
        one may fix
            $M \colonequals \max_{i \in \{1,2,\ldots,d\}}\norm{A_{i}} \in \realsNonNeg$
        and obtain the following norm bound on $\Delta_{T,K}(t)$:

            \begin{longmaths}[mc]{RCL}
                \norm{\Delta_{T,K}(t)}
                    &\leq
                    &\displaystyle\sum_{n=\card{K}+1}^{\infty}
                    \frac{t^{n-\card{K}-1}}{n!}
                    \displaystyle\sum_{k=0}^{n}
                    \displaystyle\sum_{(\pi,\pi^{\prime})\in K^{k}\times K^{n-k}}
                        \frac{n!}{k!(n-k)!}M^{n}\\
                    &=
                    &\displaystyle\sum_{n=\card{K}+1}^{\infty}
                    \frac{t^{n-\card{K}-1}}{n!}
                    \displaystyle\sum_{k=0}^{n}
                        \binom{n}{k}
                        \card{K}^{k}\card{K}^{n-k}M^{n}\\
                    &=
                    &\displaystyle\sum_{n=\card{K}+1}^{\infty}
                        \frac{t^{n-\card{K}-1}}{n!}2^{n}\card{K}^{n}M^{n},\\
            \end{longmaths}

        \continueparagraph
        whereby the expression on the right constitutes
        a power series with infinite converge radius.
        It follows that $\norm{\Delta_{T,K}(\cdot)}$ is bounded on compact subsets of $\realsNonNeg$.
    \end{proof}

\begin{lemm}
\makelabel{lemm:completely-super-dissipative-implies-brehmer:sig:article-dilation-problem-raj-dahya}
    Let $T$ be a $d$-parameter $\Cnought$-semigroup over $\HilbertRaum$
    with bounded generators $A_{1},A_{2},\ldots,A_{d}$.
    If the generators of $T$ are completely super dissipative,
    then $T$ has a regular unitary dilation.
\end{lemm}

    \begin{proof}
        By \Cref{thm:brehmer-iff-dilation:sig:article-dilation-problem-raj-dahya} it is necessary and sufficient to show that
        $T$ satisfies the Brehmer positivity criterion.
        By \Cref{prop:asymptotic-behaviour-brehmer-ops:sig:article-dilation-problem-raj-dahya}
        one has the connection between the Brehmer and dissipation operators:
            $%
                B_{T,K}(t)
                = 2^{\card{K}}t^{\card{K}}S_{T,K}
                + (-1)^{\card{K}}t^{\card{K}+1}\Delta_{T,K}(t)
            $
        for each $K\subseteq\{1,2,\ldots,d\}$ and $t\in\realsNonNeg$,
        where $\Delta_{T,K}$ is norm-bounded on compact subsets of $\realsNonNeg$.
        Set
            $%
                C \colonequals \sup\{
                    1,\,\norm{\Delta_{T,K}(t)}
                    \mid
                    t\in [0,\:1],
                    ~K\subseteq\{1,2,\ldots,d\}
                \}
            $,
        which is finite and positive.

        Since the generators of $T$ are completely super dissipative,
        then $\beta_{T} > 0$.
        Set $t^{\ast} \colonequals \min\{1,\,\frac{\beta_{T}}{2C}\} \in \realsPos$.
        Applying the spectral theory of self-adjoint operators yields
            $
                \BRAKET{S_{T,K}\xi}{\xi}
                \geq
                    \min(\opSpectrum{S_{T,K}})
                    \norm{\xi}^{2}
                \geq
                    \beta_{T}\norm{\xi}^{2}
            $
        for all $\xi\in\HilbertRaum$.
        For $t\in (0,\:t^{\ast})$ one thus obtains

            \begin{longmaths}[mc]{RCL}
                \BRAKET{B_{T,K}(t)\xi}{\xi}
                    &=
                        &t^{\card{K}} \BRAKET{2^{\card{K}}S_{T,K}\xi}{\xi}
                        + (-1)^{\card{K}}t^{\card{K}+1} \BRAKET{\Delta_{T,K}(t)\xi}{\xi}\\
                    &\geq
                        &2^{0}t^{\card{K}}\beta_{T}\norm{\xi}^{2}
                        - t^{\card{K}+1} \norm{\Delta_{T,K}(t)}\norm{\xi}^{2}\\
                    &\geq &t^{\card{K}} (\beta_{T} - C t) \norm{\xi}^{2}
                        \quad\text{(since $t \in (0,\:t^{\ast}) \subseteq [0,\:1]$)}\\
                    &\geq &t^{\card{K}} (\beta_{T} - \frac{\beta_{T}}{2}) \norm{\xi}^{2}
                        \quad\text{(since $t < t^{\ast} \leq \frac{\beta_{T}}{2C}$)}\\
                    &= &t^{\card{K}}\frac{\beta_{T}}{2}\norm{\xi}^{2}\\
            \end{longmaths}

        \continueparagraph
        for all $\xi\in\HilbertRaum$ and all $K\subseteq\{1,2,\ldots,d\}$.
        So $B_{T,K}(t) \geq \zeromatrix$ for all $t\in(0,\:t^{\ast})$ and all $K\subseteq\{1,2,\ldots,d\}$.
        Hence $T$ satisfies the Brehmer positivity criterion.
    \end{proof}

Complete super dissipativity shall prove to be a strictly stronger condition
than the existence of regular unitary dilations
(see \S{}\ref{sec:examples:doubly-comm}).
The correct condition turns out to be complete dissipativity.
To prove this, we shall make use of approximations via
semigroups with completely super dissipative generators.
And to enable this, we use the following lemma
to establish sufficient conditions for shifts to guarantee this property.

\begin{lemm}
\makelabel{lemm:self-similarities-imply-perturbations-yield-positive-beta:sig:article-dilation-problem-raj-dahya}
    Let $T$ be a $d$-parameter $\Cnought$-semigroup over $\HilbertRaum$
    with bounded generators $A_{1},A_{2},\ldots,A_{d}$.
    If $\beta_{T} < 0$,
    then the $\boldsymbol{\omega}$-shift $E_{\boldsymbol{\omega}} \cdot T$ of $T$
    is completely super dissipative
    for all $\boldsymbol{\omega} \in \realsPos^{d}$,
    provided

        \begin{maths}[bc]{rcl}
        \eqtag[eq:condition-shift-complete-super-dissipative:sig:article-dilation-problem-raj-dahya]
            \displaystyle
            \prod_{i=1}^{d}
                (1 + \omega_{i}^{-1})
            < 1 + \frac{1}{\abs{\beta_{T}}}
        \end{maths}

    \continueparagraph
    holds. And if $\beta_{T} \geq 0$,
    then $E_{\boldsymbol{\omega}} \cdot T$
    is completely super dissipative
    for all $\boldsymbol{\omega} \in \realsPos^{d}$.
\end{lemm}

    \begin{proof}
        Let $\omega\in\realsPos^{d}$ be arbitrary.
        Note that
            $\omega_{K} = \prod_{i\in K}\omega_{i} > 0$
        and
            $S_{T,K} \geq \beta_{T}\onematrix \geq -\beta_{T}^{-}\onematrix$
        for all $K \subseteq \{1,2,\ldots,d\}$.
        Applying \eqcref{eq:self-similarity-dissipation-operators:sig:article-dilation-problem-raj-dahya}
        from \Cref{prop:algebraic-identity-scaled-shifts:sig:article-dilation-problem-raj-dahya}
        to $T$ yields

            \begin{longmaths}[mc]{RCL}
                S_{E_{\boldsymbol{\omega}} \cdot T, K}
                    &=
                        &\displaystyle
                        \sum_{K^{\prime}\subseteq K}
                            \omega_{K \without K^{\prime}}^{-1}
                            S_{T,K^{\prime}}\\
                    &=
                        &\omega_{K}\onematrix
                        + \omega_{K}
                        \displaystyle
                        \sum_{\substack{
                            K^{\prime}\subseteq K:\hraum\\
                            K^{\prime}\neq\emptyset\hraum
                        }}
                            \omega_{K^{\prime}}^{-1}
                            S_{T,K^{\prime}}\\
                    &\geq
                        &\omega_{K}
                        \onematrix
                        - \omega_{K}
                        \displaystyle
                        \sum_{\substack{
                            K^{\prime}\subseteq K:\hraum\\
                            K^{\prime}\neq\emptyset\hraum
                        }}
                            \beta_{T}^{-}
                            \omega_{K^{\prime}}^{-1}
                            \cdot
                            \onematrix\\
                    &=
                        &\omega_{K}
                        \cdot
                        \Big(
                            1
                            - \beta_{T}^{-}
                            \cdot
                            \Big(
                                \displaystyle
                                \sum_{K^{\prime}\subseteq K}
                                \omega_{K^{\prime}}^{-1}
                                - 1
                            \Big)
                        \Big)
                        \onematrix\\
                    &=
                        &\omega_{K}
                        \cdot
                        \Big(
                            1
                            - \beta_{T}^{-}
                            \cdot
                            \Big(
                                \displaystyle
                                \prod_{i \in K}
                                (1 + \omega_{i}^{-1})
                                - 1
                            \Big)
                        \Big)
                        \onematrix\\
                    &\geq
                        &\omega_{K}
                        \cdot
                        \underbrace{
                            \Big(
                                1
                                - \beta_{T}^{-}
                                \cdot
                                \Big(
                                    \displaystyle
                                    \prod_{i=1}^{d}
                                    (1 + \omega_{i}^{-1})
                                    - 1
                                \Big)
                            \Big)
                        }_{\equalscolon \gamma}
                        \onematrix\\
            \end{longmaths}

        \continueparagraph
        for each $K\subseteq\{1,2,\ldots,d\}$.

        Now if $\beta_{T} \geq 0$, then $\beta_{T}^{-} = 0$ and hence $\gamma = 1 - 0 > 0$.
        Otherwise $\beta_{T} < 0$ and one has $\beta_{T}^{-} = \abs{\beta_{T}} > 0$.
        In this case, since \eqcref{eq:condition-shift-complete-super-dissipative:sig:article-dilation-problem-raj-dahya} is assumed,
        one has
            $\gamma > 1 - 1 = 0$.
        Since
            $%
                S_{E_{\boldsymbol{\omega}} \cdot T, K}
                    \geq
                        \omega_{K}\gamma\onematrix
            $
        for $K \subseteq \{1,2,\ldots,d\}$,
        one thus obtains
            $
                \beta_{E_{\boldsymbol{\omega}} \cdot T}
                    \geq
                        \min_{K\subseteq\{1,2,\ldots,d\}}\omega_{K}\gamma
                    > 0
            $,
        whence $E_{\boldsymbol{\omega}} \cdot T$ is completely super dissipative.
    \end{proof}



\subsection[Classification of regular unitary dilatability]{Classification of regular unitary dilatability}
\label{sec:applications:classification}

\firstparagraph
We can now prove \Cref{thm:classification:sig:article-dilation-problem-raj-dahya}.

    \def\beweislabel{thm:classification:sig:article-dilation-problem-raj-dahya}
    \begin{proof}[of \Cref{\beweislabel}]
        As per usual, we let
            $T_{1},T_{2},\ldots,T_{d}$
        denote the marginals of $T$
        and
            $A_{1},A_{2},\ldots,A_{d}$
        their respective generators.

        \paragraph{\hinRichtung{1}{2}:}
        By \Cref{prop:asymptotic-behaviour-brehmer-ops:sig:article-dilation-problem-raj-dahya}
        one has
            $%
                B_{T,K}(t)
                = 2^{\card{K}}t^{\card{K}}S_{T,K}
                + (-1)^{\card{K}}t^{\card{K}+1}\Delta_{T,K}(t)
            $
        where $\Delta_{T,K}$ is norm-bounded on compact subsets of $\realsNonNeg$
        for each $K\subseteq\{1,2,\ldots,d\}$ and $t\in\realsNonNeg$.
        Thus
            $%
                C \colonequals \sup\{
                    1,\,\norm{\Delta_{T,K}(t)}
                    \mid
                    t\in [0,\:1],
                    ~K\subseteq\{1,2,\ldots,d\}
                \}
            $
        is a finite, positive real.

        We shall prove the implication by contraposition.
        So, suppose that $\beta_{T} < 0$.
        By definition there must exist some
            $\tilde{K} \subseteq \{1,2,\ldots,d\}$
        such that
            $\min(\opSpectrum{S_{T,\tilde{K}}}) = \beta_{T} < 0$.
        By the correspondence between the numerical range of self-adjoint operators
        and their spectra (\cf \cite[Theorem~1.2.1--4]{Gustafson1997numericalRange}),
        it holds that
            $%
                \beta_{T}
                = \inf\{
                    \BRAKET{S_{T,\tilde{K}}\xi}{\xi}
                    \mid
                    \xi \in\HilbertRaum,~\norm{\xi}=1
                \}
            $.
        Set $t^{\ast} \colonequals \min\{1,\,-\frac{2^{\card{\tilde{K}}}\beta_{T}}{2C}\} \in \realsPos$.
        For $t\in (0,\:t^{\ast})$ one computes

            \begin{longmaths}[mc]{RCL}
                \BRAKET{B_{T,\tilde{K}}(t)\xi}{\xi}
                    &=
                        &t^{\card{\tilde{K}}} \BRAKET{2^{\card{\tilde{K}}}S_{T,\tilde{K}}\xi}{\xi}
                        + (-1)^{\card{\tilde{K}}}t^{\card{\tilde{K}}+1} \BRAKET{\Delta_{T,\tilde{K}}(t)\xi}{\xi}\\
                    &\leq
                        &t^{\card{\tilde{K}}}2^{\card{\tilde{K}}}
                            \BRAKET{S_{T,\tilde{K}}\xi}{\xi}
                        + t^{\card{\tilde{K}}+1}
                            \norm{\Delta_{T,\tilde{K}}(t)}
                            \norm{\xi}^{2}\\
                    &\leq &t^{\card{\tilde{K}}}\Big(
                            2^{\card{\tilde{K}}}
                            \BRAKET{S_{T,\tilde{K}}\xi}{\xi}
                            + C t \cdot 1^{2}
                        \Big)
                        \quad\text{(since $t \in (0,\:t^{\ast}) \subseteq [0,\:1]$)}\\
                    &\leq &t^{\card{\tilde{K}}}\Big(
                            2^{\card{\tilde{K}}}
                            \BRAKET{S_{T,\tilde{K}}\xi}{\xi}
                            - \frac{2^{\card{\tilde{K}}}\beta_{T}}{2}
                        \Big)
                        \quad\text{(since $t < t^{\ast} \leq -\frac{2^{\tilde{K}}\beta_{T}}{2C}$)}\\
                    &= &(2t)^{\card{\tilde{K}}}(\BRAKET{S_{T,\tilde{K}}\xi}{\xi} - \frac{\beta_{T}}{2})\\
            \end{longmaths}

        \continueparagraph
        for all $\xi\in\HilbertRaum$ with $\norm{\xi}=1$.
        Taking infima over such vectors $\xi$ thus yields

            \begin{maths}[mc]{c}
                \inf_{\xi}
                    \BRAKET{B_{T,\tilde{K}}(t)\xi}{\xi}
                \leq
                    (2t)^{\card{\tilde{K}}}(\inf_{\xi}\BRAKET{S_{T,\tilde{K}}\xi}{\xi} - \frac{\beta_{T}}{2})
                = (2t)^{\card{\tilde{K}}}(\beta_{T} - \frac{\beta_{T}}{2})
                < 0,\\
            \end{maths}

        \continueparagraph
        whence there must exist some $\xi\in\HilbertRaum$ with $\norm{\xi}=1$
        such that $\BRAKET{B_{T,\tilde{K}}(t)\xi}{\xi} < 0$.
        In particular, $B_{T,\tilde{K}}(t)$ is not a positive operator.
        Since this is the case for all $t\in(0,t^{\ast})$,
        the Brehmer positivity criterion fails
        and by \Cref{thm:brehmer-iff-dilation:sig:article-dilation-problem-raj-dahya}
        $T$ does not have a regular unitary dilation.

        \paragraph{\hinRichtung{2}{3}:}
        Directly order $\realsPos^{d}$ via
            $\boldsymbol{\omega}^{\prime}\succeq\boldsymbol{\omega}$
            $:\Leftrightarrow$
            $\forall{i\in\{1,2,\ldots,d\}:~}\omega^{\prime}_{i}\leq\omega_{i}$
        for each $\boldsymbol{\omega},\boldsymbol{\omega}^{\prime}\in\realsPos^{d}$,
        and consider the net
            $(T^{(\boldsymbol{\omega})} \colonequals E_{\boldsymbol{\omega}} \cdot T)_{\boldsymbol{\omega}\in\realsPos^{d}}$.
        Since $T$ has completely dissipative generators,
        by \Cref{lemm:self-similarities-imply-perturbations-yield-positive-beta:sig:article-dilation-problem-raj-dahya},
        this net consists of $d$-parameter $\Cnought$-semigroups
        with completely super dissipative generators,
        and further by \Cref{lemm:completely-super-dissipative-implies-brehmer:sig:article-dilation-problem-raj-dahya},
        these semigroups have regular unitary dilations.

        We now demonstrate uniform norm-convergence on compact subsets of $\realsNonNeg^{d}$.
        To this end fix an arbitrary compact subset $L \subseteq \realsNonNeg^{d}$.
        By compactness, $L\subseteq\prod_{i=1}^{d}[0,\:C]$ for some $C\in\realsNonNeg$.
        Since $\topSOT$-continuous $1$-parameter $\Cnought$-semigroups
        are norm bounded on compact subsets,
        there exists some $M\in\realsNonNeg$ such that
            $\norm{T_{i}(t)} \leq M$ for all $t \in [0,\:C]$
        and all $i\in \{1,2,\ldots,d\}$.
        Putting this together yields

            \begin{maths}[mc]{c}
                \displaystyle
                \sup_{\mathbf{t} \in L}
                \norm{
                    T^{(\boldsymbol{\omega})}(\mathbf{t})
                    - T(\mathbf{t})
                }
                    =
                        \displaystyle
                        \sup_{\mathbf{t} \in L}
                        \abs{e^{-\BRAKET{\mathbf{t}}{\boldsymbol{\omega}}} - 1}
                        \norm{T(\mathbf{t})}
                    \leq
                        M^{d} \cdot (1 - e^{-C\sum_{i=1}^{d}\omega_{i}})\\
            \end{maths}

        \continueparagraph
        for each $\boldsymbol{\omega}\in\realsPos^{d}$.
        The right- and thus the left-hand side of the above clearly converge to $0$
        as ${\boldsymbol{\omega} \longrightarrow \zerovector}$.
        So ${T^{(\boldsymbol{\omega})} \underset{\boldsymbol{\omega}}{\longrightarrow} T}$
        uniformly in norm on compact subsets of $\realsNonNeg^{d}$.
        Hence \punktcref{3} holds.

        \paragraph{\hinRichtung{3}{3dash}:} This trivially holds.

        \paragraph{\hinRichtung{3dash}{1}:}
        By \Cref{thm:brehmer-iff-dilation:sig:article-dilation-problem-raj-dahya},
        it is necessary and sufficient to show that
            $B_{T, K}(t) \geq \zeromatrix$
            for all $K\subseteq\{1,2,\ldots,d\}$
            and all $t\in\realsPos$.
        So fix arbitrary
            $K\subseteq\{1,2,\ldots,d\}$
            and
            $t\in\realsPos$.
        Since each of the $T^{(\alpha)}$ are regularly unitarily dilatable,
        by \Cref{thm:brehmer-iff-dilation:sig:article-dilation-problem-raj-dahya},
        one has that
            $B_{T^{(\alpha)}, K}(t) \geq \zeromatrix$
        for each $\alpha\in\cal{I}$.
        Now, for $\xi\in\HilbertRaum$ one has

            \begin{longmaths}[mc]{RCL}
                \abs{\BRAKET{(
                    B_{T^{(\alpha)}, K}(t)
                    - B_{T, K}(t)
                )\xi}{\xi}}
                    &=
                        &\absLong{
                        \displaystyle
                        \sum_{C \subseteq K}
                            (-1)^{|C|}
                            (
                                \norm{T^{(\alpha)}(t\mathbf{e}_{C})\xi}^{2}
                                - \norm{T(t\mathbf{e}_{C})\xi}^{2}
                            )
                        }\\
                    &\leq
                        &\displaystyle
                        \sum_{C \subseteq K}
                            \abs{
                                \norm{T^{(\alpha)}(t\mathbf{e}_{C})\xi}^{2}
                                - \norm{T(t\mathbf{e}_{C})\xi}^{2}
                            },\\
            \end{longmaths}

        \continueparagraph
        which clearly converges to $0$, since by assumption
            ${(T^{(\alpha)})_{\alpha} \longrightarrow T}$
        uniformly in the $\topSOT$-topology
        on compact subsets of $\realsNonNeg^{d}$
        (\exempli $\{t\mathbf{e}_{C}\}$).
        Since
            $B_{T^{(\alpha)}, K}(t) \geq \zeromatrix$
        for all $\alpha\in\cal{I}$,
        the above computation implies
            $\BRAKET{B_{T, K}(t)\xi}{\xi} \geq 0$
        for all $\xi\in\HilbertRaum$,
        \idest $B_{T,K}(t) \geq \zeromatrix$.
    \end{proof}

Having established \Cref{thm:classification:sig:article-dilation-problem-raj-dahya},
we may now prove \Cref{cor:dilatable-up-to-exponential-modification:sig:article-dilation-problem-raj-dahya}.

    \def\beweislabel{cor:dilatable-up-to-exponential-modification:sig:article-dilation-problem-raj-dahya}
    \begin{proof}[of \Cref{\beweislabel}]
        By \Cref{lemm:self-similarities-imply-perturbations-yield-positive-beta:sig:article-dilation-problem-raj-dahya}
        we can find $\boldsymbol{\omega}\in\realsPos^{d}$ such that
        the $\boldsymbol{\omega}$-shift $E_{\boldsymbol{\omega}} \cdot T$ of $T$
        has completely (super!) dissipative generators.
        By \Cref{thm:classification:sig:article-dilation-problem-raj-dahya}
        it follows that $\tilde{T}=E_{\boldsymbol{\omega}} \cdot T$
        has a regular unitary dilation.
    \end{proof}

\begin{rem}
    Observe that, if we let
        (\ref*{it:3dash:thm:classification:sig:article-dilation-problem-raj-dahya}')
    denote the same statement as
        \eqcref{it:3dash:thm:classification:sig:article-dilation-problem-raj-dahya}
        in \Cref{thm:classification:sig:article-dilation-problem-raj-dahya},
    except with \emph{%
        uniform $\topSOT$-convergence on compact subsets of $\realsNonNeg^{d}$%
    } replaced by the weaker condition of \emph{%
        pointwise $\topSOT$-convergence%
    },
    then the above proof makes clear that
        \eqcref{it:3dash:thm:classification:sig:article-dilation-problem-raj-dahya}
        ~$\Rightarrow$~%
        (\ref*{it:3dash:thm:classification:sig:article-dilation-problem-raj-dahya}')%
        ~$\Rightarrow$~%
        \eqcref{it:1:thm:classification:sig:article-dilation-problem-raj-dahya}.
    So the equivalences in the theorem also hold with
        \eqcref{it:3dash:thm:classification:sig:article-dilation-problem-raj-dahya}
    replaced by the weaker statement
        (\ref*{it:3dash:thm:classification:sig:article-dilation-problem-raj-dahya}').
    Observe further that the statement and proof of
        \eqcref{it:3dash:thm:classification:sig:article-dilation-problem-raj-dahya}%
        ~$\Rightarrow$~%
        \eqcref{it:1:thm:classification:sig:article-dilation-problem-raj-dahya}
    do not require $T$ to have bounded generators.
    It would be interesting to know whether the concept of complete dissipativity
    can be well-defined without the assumption of bounded generators,
    and whether
        \eqcref{it:1:thm:classification:sig:article-dilation-problem-raj-dahya}%
        ~$\Rightarrow$~%
        \eqcref{it:2:thm:classification:sig:article-dilation-problem-raj-dahya}
        and
        \eqcref{it:2:thm:classification:sig:article-dilation-problem-raj-dahya}%
        ~$\Rightarrow$~%
        \eqcref{it:3dash:thm:classification:sig:article-dilation-problem-raj-dahya}
    in the theorem remain true.
\end{rem}



\subsection[Infinitely many commuting $\Cnought$-semigroups]{Infinitely many commuting $\Cnought$-semigroups}
\label{sec:results:infinite}

\firstparagraph
We now extend the notion of complete dissipativity and the classification result
to arbitrarily many commuting $\Cnought$-semigroups.

Let $I$ be an arbitrary non-empty set
and $(T_{i})_{i \in I}$ be commuting $\Cnought$-semigroups over a Hilbert space $\HilbertRaum$,
with bounded generators $(A_{i})_{i \in I} \subseteq \BoundedOps{\HilbertRaum}$.
For finite subsets $K \subseteq I$ and $t\in\realsNonNeg$
we define

    \begin{maths}[mc]{rcl}
        B_{(T_{i})_{i \in I}, K}(t)
            &\colonequals
                &\displaystyle
                \sum_{C \subseteq K}
                    (-1)^{\card{C}}
                    \Big(
                        \displaystyle
                        \prod_{i \in C}T_{i}(t)
                    \Big)^{\ast}
                    \Big(
                        \displaystyle
                        \prod_{i \in C}T_{i}(t)
                    \Big),
        ~\text{and}\\
        S_{(T_{i})_{i \in I}, K}
            &\colonequals
                &2^{-\card{K}}
                \displaystyle
                \sum_{\mathclap{\isPartition{(C_{1},C_{2})}{K}}}
                    \Big(
                        \displaystyle
                        \prod_{i \in C_{1}}A^{-}_{i}
                    \Big)^{\ast}
                    \Big(
                        \displaystyle
                        \prod_{i \in C_{2}}A^{-}_{i}
                    \Big),\\
    \end{maths}

\continueparagraph
which clearly coincide with the \highlightTerm{Brehmer} and \highlightTerm{dissipation operators}
in \Cref{defn:brehmer-operators:sig:article-dilation-problem-raj-dahya}
and \Cref{defn:dissipation-operators-and-conditions:sig:article-dilation-problem-raj-dahya}
respectively in the finite case of $I=\{1,2,\ldots,d\}$ for some $d\in\naturalsPos$.
We shall say that the generators of $(T_{i})_{i \in I}$ are \highlightTerm{completely dissipative}
if $S_{(T_{i})_{i \in I}, K} \geq \zeromatrix$ for all finite subsets $K \subseteq I$.
Again, this coincides with \Cref{defn:dissipation-operators-and-conditions:sig:article-dilation-problem-raj-dahya}
in the finite case.

We say that $(T_{i})_{i \in I}$ has a \highlightTerm{simultaneous regular unitary dilation},
if there exist
    commuting unitary $\Cnought$-semigroups $(U_{i})_{i \in I}$ over some Hilbert space $\HilbertRaum^{\prime}$
    and $r\in\BoundedOps{\HilbertRaum}{\HilbertRaum^{\prime}}$ (necessarily isometric),
such that

    \begin{maths}[mc]{c}
        \Big(
            \displaystyle
            \prod_{i \in \supp(\mathbf{t}^{-})}
            T_{i}(t_{i}^{-})
        \Big)^{\ast}
        \Big(
            \displaystyle
            \prod_{i \in \supp(\mathbf{t}^{+})}
            T_{i}(t_{i}^{+})
        \Big)
            =
                r^{\ast}\,(\prod_{i \in K}U_{i}(t_{i}))\,r
    \end{maths}

\continueparagraph
for all $\mathbf{t}\in\reals^{K}$ and finite $K\subseteq\{1,2,\ldots,d\}$,
where we extend the unitary semigroups to representations of $(\reals,+,0)$
in the usual way.
This clearly coincides with the definition of a regular unitary dilation in the finite case
(\cf \Cref{defn:dilation-regular:sig:article-dilation-problem-raj-dahya} and the introduction).

Using these definitions we obtain the following classification as a simple consequence
of the classification result in the finite case:

\begin{cor}[General classification of regular unitary dilatability]
\makelabel{cor:general-classification:sig:article-dilation-problem-raj-dahya}
    Let $I$ be an arbitrary non-empty set
    and $(T_{i})_{i \in I}$ be $\Cnought$-semigroups over a Hilbert space $\HilbertRaum$
    with bounded generators $(A_{i})_{i \in I} \subseteq \BoundedOps{\HilbertRaum}$.
    Then $(T_{i})_{i \in I}$ has a regular unitary dilation
    if and only if the generators of $(T_{i})_{i \in I}$ are completely dissipative.
\end{cor}

This proof relies on the fact that \Cref{thm:brehmer-iff-dilation:sig:article-dilation-problem-raj-dahya}
is formulated for arbitrarily large families of commuting $\Cnought$-semigroups
in \cite{Ptak1985}.

    \begin{proof}[of \Cref{\beweislabel}]
        For the \usesinglequotes{only if}-direction,
        suppose that $(T_{i})_{i \in I}$ has regular unitary dilation.
        Then using this dilation, clearly every finite subset
        $K \subseteq I$ has a regular unitary dilation,
        and thus by the classification theorem (\Cref{thm:classification:sig:article-dilation-problem-raj-dahya}),
        $(T_{i})_{i \in K}$ has completely dissipative generators.
        In particular,
            $
                S_{(T_{i})_{i \in I}, K}
                = S_{(T_{i})_{i \in K},K}
                \geq \zeromatrix
            $.
        Since this holds for all finite $K \subseteq I$,
        by definition, the generators of $(T_{i})_{i \in I}$ are completely dissipative.

        Towards the \usesinglequotes{if}-direction,
        suppose that the generators of $(T_{i})_{i \in I}$ are completely dissipative.
        Then clearly for every finite subset $K \subseteq I$
        the generators of $(T_{i})_{i \in K}$ are completely dissipative,
        and thus by the classification theorem (\Cref{thm:classification:sig:article-dilation-problem-raj-dahya})
            $(T_{i})_{i \in K}$ has a regular unitary dilation.
        By \Cref{thm:brehmer-iff-dilation:sig:article-dilation-problem-raj-dahya}
        it follows that the Brehmer operators associated to
            $(T_{i})_{i \in K}$ are positive for all $t\in\realsPos$.
        Thus
            $%
                B_{(T_{i})_{i \in I}, K}(t)
                = B_{(T_{i})_{i \in K}, K}(t)
                \geq \zeromatrix
            $
        for all $t\in\realsPos$.
        So $B_{(T_{i})_{i \in I}, K}(t) \geq \zeromatrix$ for all $t\in\realsPos$
        and for all finite $K \subseteq I$.
        \Idest, the commuting system $(T_{i})_{i \in I}$ of $\Cnought$-semigroups
        satisfies the Brehmer positivity criterion.
        By \cite[Theorem~3.2]{Ptak1985} it follows that $(T_{i})_{i \in I}$ has a regular unitary dilation.
    \end{proof}



\subsection[Sufficient norm conditions]{Sufficient norm conditions}
\label{sec:applications:approximation}

\firstparagraph
As it is not always so easy to check the spectral values of (the dissipation) operators,
we wish to provide simple norm-conditions in terms on the generators,
which guarantee complete (super) dissipativity.
We shall show that it suffices to control the relative norm-bounds of the generators.
To this end we introduce the following definition.

\begin{defn}
    Let $A \in \BoundedOps{\HilbertRaum}$ and $r\in\realsNonNeg$.
    Say that $A$ is \highlightTerm{$r$-circular}
        (\respectively \highlightTerm{strictly $r$-circular})
    if
        $\frac{\norm{A - \omega\onematrix}}{\omega} \leq r$
        (\respectively $\frac{\norm{A - \omega\onematrix}}{\omega} < r$)
    for some $\omega\in\realsPos$.
    If $r=1$ we shall say \highlightTerm{(strictly) circular} instead of (strictly) $1$-circular.
\end{defn}

These notions are visualised in \Cref{fig:examples-circular-operators:sig:article-dilation-problem-raj-dahya}
in the case of normal operators.


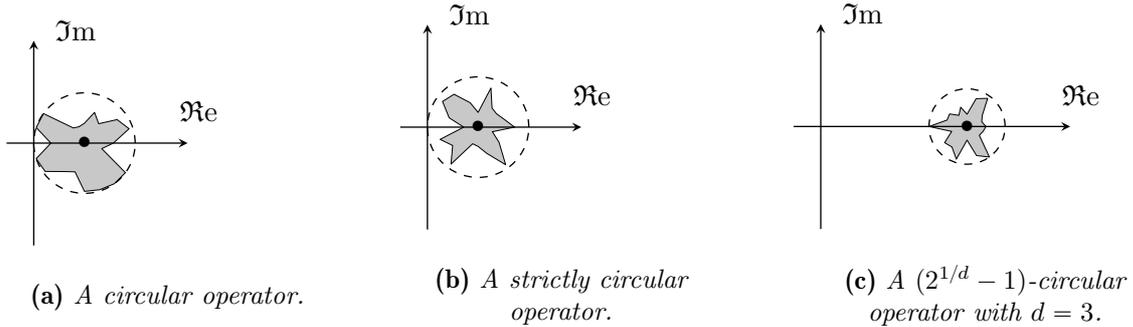
\begin{figure}[!htb]
    \centering

    \begin{subfigure}[m]{0.3\textwidth}
        \begin{tikzpicture}[node distance=1cm, thick]
            \filldraw[%
                fill={rgb,255:white,200},%
                line join=round,%
                line width=0pt,%
            ] plot coordinates{
                    (1.0469281295827075, 0.0)
                    (1.253656252801211, 0.1907244780187156)
                    (1.0958183011555136, 0.3117969134191595)
                    (0.8534834207198587, 0.2571312027199234)
                    (0.7989876233732314, 0.40724203007907367)
                    (0.6666666666666666, 0.26369468179681915)
                    (0.5959319737114303, 0.21769900001547754)
                    (0.5209343863456557, 0.20058327586287647)
                    (0.1273220037500351, 0.3918568348616488)
                    (0.03262898913656431, 0.20601132958329832)
                    (0.22531061771713112, 5.405052726303712e-17)
                    (0.0326289891365642, -0.20601132958329818)
                    (0.1495523092089749, -0.37570557253517917)
                    (0.3925507615072684, -0.37728817597498643)
                    (0.5428397918923332, -0.38109993395282743)
                    (0.6666666666666665, -0.6402268374604754)
                    (0.8675583478505868, -0.6182810199352111)
                    (1.0585235015283152, -0.5393446629166316)
                    (1.2060113295832982, -0.3918568348616488)
                    (0.8938513439532182, -0.07381677633271723)
                    (1.0469281295827075, 0.0)
                }--cycle;
                \node (M) at (0.6666666666666666, 0) {$\bullet$};
                \node (yAxisMin) at (0, -1.5) {};
                \node[label=right:{$\Imag$}] (yAxisMax) at (0, 1.5) {};
                \node (xAxisMin) at (-0.5, 0) {};
                \node[label=above:{$\Re$}] (xAxisMax) at (2.1666666666666665, 0) {};
                \draw[->,line width=0.5pt] (yAxisMin) edge (yAxisMax);
                \draw[->,line width=0.5pt] (xAxisMin) edge (xAxisMax);
                \draw[line width=0.5pt,dashed] (0.6666666666666666, 0) circle [radius=0.6666666666666666];
            \end{tikzpicture}
        \caption{A circular operator.}
    \end{subfigure}
    \hraum
    \begin{subfigure}[m]{0.3\textwidth}
        \begin{tikzpicture}[node distance=1cm, thick]
            \filldraw[%
                fill={rgb,255:white,200},%
                line join=round,%
                line width=0pt,%
            ] plot coordinates{
                    (1.146331968320027, 0.0)
                    (1.0024484724874236, 0.1091021223478171)
                    (0.9094479191694153, 0.17639090494565485)
                    (0.8500538376435264, 0.25241078657890564)
                    (0.8350567546488601, 0.5182514016062891)
                    (0.6666666666666666, 0.19528163984234617)
                    (0.5613203535807796, 0.32422261348654446)
                    (0.3498737574232762, 0.436028032816068)
                    (0.19846622943692338, 0.34016752927811306)
                    (0.2856482058008025, 0.12380040256366612)
                    (0.47787036547177597, 2.3120878594918247e-17)
                    (0.16682713545611955, -0.16240770864612916)
                    (0.3318434241116225, -0.24326332508089385)
                    (0.3101307110551556, -0.4907296433015965)
                    (0.5952399212409928, -0.21982891851060418)
                    (0.6666666666666666, -0.2559332489741431)
                    (0.7793707442349662, -0.3468674841044786)
                    (1.0291342389136915, -0.49889381319788434)
                    (0.9414119500086713, -0.19961413271684927)
                    (0.8541547843848645, -0.06091858225627627)
                    (1.146331968320027, 0.0)
                }--cycle;
                \node (M) at (0.6666666666666666, 0) {$\bullet$};
                \node (yAxisMin) at (0, -1.5) {};
                \node[label=right:{$\Imag$}] (yAxisMax) at (0, 1.5) {};
                \node (xAxisMin) at (-0.5, 0) {};
                \node[label=above:{$\Re$}] (xAxisMax) at (2.1666666666666665, 0) {};
                \draw[->,line width=0.5pt] (yAxisMin) edge (yAxisMax);
                \draw[->,line width=0.5pt] (xAxisMin) edge (xAxisMax);
                \draw[line width=0.5pt,dashed] (0.6666666666666666, 0) circle [radius=0.6666666666666666];
            \end{tikzpicture}
        \caption{A strictly circular operator.}
    \end{subfigure}
    \hraum
    \begin{subfigure}[m]{0.35\textwidth}
        \begin{tikzpicture}[node distance=1cm, thick]
            \filldraw[%
                fill={rgb,255:white,200},%
                line join=round,%
                line width=0pt,%
            ] plot coordinates{
                    (2.1732463996870477, 0.0)
                    (2.1168332864292765, 0.06276546407855715)
                    (2.123273659398018, 0.14502704917698198)
                    (2.1948737861287, 0.37329230532691254)
                    (2.043498685394311, 0.36882231472010757)
                    (1.9236610509315362, 0.16241351113441463)
                    (1.8404801334152308, 0.25600454044706566)
                    (1.816237650985143, 0.14785562552175952)
                    (1.6943452121897367, 0.1666077091911368)
                    (1.660883853249897, 0.0853814872476526)
                    (1.4236610509315362, 6.123233995736766e-17)
                    (1.7219194584572441, -0.06554981694428971)
                    (1.628859284753245, -0.21418602045962085)
                    (1.7256661949441483, -0.2725165401273342)
                    (1.78289934598526, -0.4332199819778745)
                    (1.9236610509315362, -0.1804837745231156)
                    (2.048960215269385, -0.3856311753044148)
                    (2.2175536770777726, -0.4045084971874738)
                    (2.0848539069794016, -0.11711346512942039)
                    (2.132210613830592, -0.0677618606266666)
                    (2.1732463996870477, 0.0)
                }--cycle;
                \node (M) at (1.9236610509315362, 0) {$\bullet$};
                \node (yAxisMin) at (0, -1.5) {};
                \node[label=right:{$\Imag$}] (yAxisMax) at (0, 1.5) {};
                \node (xAxisMin) at (-0.5, 0) {};
                \node[label=above:{$\Re$}] (xAxisMax) at (3.4236610509315364, 0) {};
                \draw[->,line width=0.5pt] (yAxisMin) edge (yAxisMax);
                \draw[->,line width=0.5pt] (xAxisMin) edge (xAxisMax);
                \draw[line width=0.5pt,dashed] (1.9236610509315362, 0) circle [radius=0.5];
            \end{tikzpicture}
        \caption{A $(2^{1/d}-1)$-circular operator with $d=3$.}
    \end{subfigure}

    \caption{%
        Visual examples of the spectra of normal operators
        under the definition of (strict) $r$-circularity,
        depicted as shaded regions of the complex plane
        contained in a disc around $\omega$.
    }
    \label{fig:examples-circular-operators:sig:article-dilation-problem-raj-dahya}
\end{figure}


\begin{prop}
\makelabel{prop:strict-circularity-implies-complete-super-dissipativity:sig:article-dilation-problem-raj-dahya}
    If each of the dissipation operators associated to the generators of $T$
    are (strictly) circular,
    then the generators of $T$ are completely (super) dissipative.
\end{prop}

    \begin{proof}
        First observe that if a self-adjoint operator
            $S \in \BoundedOps{\HilbertRaum}$
        is (strictly) circular,
        then for some $\omega\in\realsPos$
        one has
            $%
                S \geq \omega\onematrix - \norm{S - \omega\onematrix}\onematrix
                = \delta\onematrix
            $
        where $\delta = \omega\cdot(1 - \frac{\norm{S - \omega\onematrix}}{\omega}) \geq 0$
            (\respectively $\delta > 0$).
        Thus since each $S_{T,K}$ is (strictly) circular,
        one has that $S_{T,K} \geq \delta_{K}\onematrix$
        for some $\delta_{K}\in\realsNonNeg$ (\respectively $\delta_{K}\in\realsPos$)
        for each $K\subseteq\{1,2,\ldots,d\}$.
        It follows that
            $\beta_{T} \geq \min\{\delta_{K} \mid K\subseteq\{1,2,\ldots,d\}\} \geq 0$
            (\respectively $\beta_{T} > 0$)
        whence by definition the generators of $T$ are completely (super) dissipative.
    \end{proof}

Relying on the self-similarity of dissipation operators under shifts
the following result provides sufficient means to guarantee (strict) circularity.

\begin{prop}
\makelabel{prop:dissipation-operators-in-terms-of-deviations:sig:article-dilation-problem-raj-dahya}
    Let $T$ be a $d$-parameter $\Cnought$-semigroup over $\HilbertRaum$
    with bounded generators $A_{1},A_{2},\ldots,A_{d}$.
    Fix $\boldsymbol{\omega} \in (\reals \without \{0\})^{d}$
    and let $\tilde{T} \colonequals E_{-\boldsymbol{\omega}} \cdot T$
    be the $(-\boldsymbol{\omega})$-shift of $T$.
    Then

        \begin{maths}[mc]{c}
        \eqtag[eq:dissipation-operators-in-terms-of-deviations-norm:sig:article-dilation-problem-raj-dahya]
            \frac{\norm{S_{T,K} - \omega_{K}\onematrix}}{\abs{\omega_{K}}}
                \leq
                    \displaystyle
                    \sum_{\substack{
                        K^{\prime}\subseteq K:\hraum\\
                        K^{\prime} \neq \emptyset\hraum
                    }}
                        \textstyle
                        \frac{
                            \norm{S_{\tilde{T},K^{\prime}}}
                        }{\abs{\omega_{K^{\prime}}}}
                \leq
                    \displaystyle
                    \prod_{i \in K}
                    \Big(
                        1
                        +
                        \textstyle
                        \frac{\norm{A_{i} + \omega_{i}\onematrix}}{\abs{\omega_{i}}}
                    \Big)
                    - 1\\
        \end{maths}

    \continueparagraph
    holds for all $K \subseteq \{1,2,\ldots,d\}$.
\end{prop}

    \begin{proof}
        We shall apply \Cref{prop:algebraic-identity-scaled-shifts:sig:article-dilation-problem-raj-dahya}
        to $\tilde{T}$ and $\boldsymbol{\omega}$.
        Observe that $T = E_{\boldsymbol{\omega}} \cdot \tilde{T}$
        and $\tilde{T}$ has generators
            $\tilde{A}_{i} \colonequals A_{i} + \omega_{i}\onematrix$
        for each $i\in\{1,2,\ldots,d\}$.
        Let $K \subseteq \{1,2,\ldots,d\}$.
        Since $S_{\tilde{T},\emptyset} = \onematrix$,
        applying
            \eqcref{eq:self-similarity-dissipation-operators:sig:article-dilation-problem-raj-dahya}
        to $\tilde{T}$ and $\boldsymbol{\omega}$
        yields

            \begin{maths}[mc]{c}
                \omega_{K}^{-1}
                \cdot (S_{T,K} - \omega_{K}\onematrix)
                =
                    \omega_{K}^{-1}
                    \cdot (S_{E_{\boldsymbol{\omega}} \cdot \tilde{T},K} - \omega_{K}\onematrix)
                =
                    \displaystyle
                    \sum_{\substack{
                        K^{\prime} \subseteq K:\hraum\\
                        K^{\prime} \neq \emptyset\hraum
                    }}
                        \omega_{K^{\prime}}^{-1}
                        S_{\tilde{T},K^{\prime}},
            \end{maths}

        \continueparagraph
        which implies the first inequality.
        Now, since
            $S_{\tilde{T},K^{\prime}}$
        is simply an average of products of the operators
            $\{
                (\tilde{A}^{-}_{i})^{\ast},
                \tilde{A}^{-}_{i}
                \mid
                i \in K^{\prime}
            \}$
        with each index from $K^{\prime}$ occurring exactly once in each product,
        one has
            $
                \norm{S_{\tilde{T},K^{\prime}}}
                \leq \prod_{i \in K^{\prime}}\norm{\tilde{A}^{-}_{i}}
                = \prod_{i \in K^{\prime}}
                    \norm{A_{i} + \omega_{i}\onematrix}
            $
        for each $K^{\prime} \subseteq K$,
        whence

            \begin{maths}[bc]{rcl}
                \displaystyle
                \sum_{\substack{
                    K^{\prime} \subseteq K:\hraum\\
                    K^{\prime} \neq \emptyset\hraum
                }}
                    \textstyle
                    \frac{
                        \norm{S_{\tilde{T},K^{\prime}}}
                    }{\abs{\omega_{K^{\prime}}}}
                \leq
                    \displaystyle
                    \sum_{\substack{
                        K^{\prime} \subseteq K:\hraum\\
                        K^{\prime} \neq \emptyset\hraum
                    }}
                        \displaystyle
                        \prod_{i \in K^{\prime}}
                            \frac{\norm{A_{i} + \omega_{i}\onematrix}}{\abs{\omega_{i}}}
                =
                    \displaystyle
                    \prod_{i \in K}
                        \Big(
                            1
                            +
                            \textstyle
                            \frac{\norm{A_{i} + \omega_{i}\onematrix}}{\abs{\omega_{i}}}
                        \Big)
                    - 1,
            \end{maths}

        \continueparagraph
        which proves the second inequality.
    \end{proof}

Thus, if the generators are (strictly) $r$-circular,
then the dissipation operators are at least (strictly) $((1 + r)^{d}-1)$-circular.
In this way, we may prove \Cref{thm:relative-norm-bound-generators-implies-dilation:sig:article-dilation-problem-raj-dahya}.

    \def\beweislabel{thm:relative-norm-bound-generators-implies-dilation:sig:article-dilation-problem-raj-dahya}
    \begin{proof}[of \Cref{\beweislabel}]
        By \Cref{thm:classification:sig:article-dilation-problem-raj-dahya}
        it suffices to show that the generators of $T$ are completely dissipative.
        And by \Cref{prop:strict-circularity-implies-complete-super-dissipativity:sig:article-dilation-problem-raj-dahya}
        it suffices in turn to show that the dissipation operators associated to $T$
        are circular.
        Using $\boldsymbol{\omega} \in \realsPos^{d}$ as in the assumption
        we now consider the $\boldsymbol{\omega}$-shift
            ${E_{\boldsymbol{\omega}} \cdot T}$ of $T$.
        By \Cref{prop:dissipation-operators-in-terms-of-deviations:sig:article-dilation-problem-raj-dahya},
        in order to obtain the circularity of $T$,
        it suffices to show that the right-hand expression in
        \eqcref{eq:dissipation-operators-in-terms-of-deviations-norm:sig:article-dilation-problem-raj-dahya}
        is bounded by $1$ for each $K \subseteq \{1,2,\ldots,d\}$.
        Applying the inequality in the assumption yields
            $
                \prod_{i \in K}
                    (
                        1 + \frac{\norm{A_{i} + \omega_{i}\onematrix}}{\omega_{i}}
                    )
                - 1
                    \leq
                        \prod_{i = 1}^{d}
                            (
                                1 + \frac{\norm{A_{i} + \omega_{i}\onematrix}}{\omega_{i}}
                            )
                        - 1
                    \leq (1 + 2^{1/d} - 1)^{d} - 1
                    = 1
            $
        for all $K\subseteq\{1,2,\ldots,d\}$.
        This completes the proof.
    \end{proof}

\begin{rem}
    \Cref{thm:relative-norm-bound-generators-implies-dilation:sig:article-dilation-problem-raj-dahya}
    further enables us to obtain an alternative proof
    of \Cref{cor:dilatable-up-to-exponential-modification:sig:article-dilation-problem-raj-dahya}.
    Let $A_{1},A_{2},\ldots,A_{d}$ be the generators of $T$.
    Choose any
        $
            \boldsymbol{\omega}
                \in
                    \realsPos^{d}
                    \cap \prod_{i=1}^{d}
                        [\frac{\norm{A_{i}}}{2^{1/d} - 1},\:\infty)
        $
    and set $\tilde{T} \colonequals E_{\boldsymbol{\omega}} \cdot T$.
    Then $\tilde{T}$ is a $d$-parameter $\Cnought$-semigroup
    with generators
        $\tilde{A}_{i} = A_{i} - \omega_{i}\cdot\onematrix$
    for all $i\in\{1,2,\ldots,d\}$.
    By construction one thus has
        $
            \frac{
                \norm{\tilde{A}_{i} + \omega_{i}\cdot\onematrix}
            }{\omega_{i}}
            = \frac{\norm{A_{i}}}{\omega_{i}}
            \leq 2^{1/d} - 1
        $
    for all $i\in\{1,2,\ldots,d\}$.
    Thus $\boldsymbol{\omega}$ witnesses that $\tilde{T}$ satisfies the conditions in
        \Cref{thm:relative-norm-bound-generators-implies-dilation:sig:article-dilation-problem-raj-dahya},
    which says that $\tilde{T}$ possesses a regular unitary dilation.
\end{rem}

Finally, observe that,
since the spectral radius of a bounded operator over a Hilbert space is bounded by its norm,
the norm requirement in \Cref{thm:relative-norm-bound-generators-implies-dilation:sig:article-dilation-problem-raj-dahya}
clearly implies that the generators of $T$ have strictly negative spectral bounds,
provided $d \geq 2$.
However, as shall be shown in \S{}\ref{sec:examples:non-doubly-comm},
this assumption does not suffice
for the existence of regular unitary dilations.



\subsection[Regular exponents]{Regular exponents}
\label{sec:applications:approximation}

\firstparagraph
In \Cref{cor:dilatable-up-to-exponential-modification:sig:article-dilation-problem-raj-dahya}
we showed that all multi-parameter $\Cnought$-semigroups with bounded generators
can be modified (shifted) to regularly unitarily dilatable semigroups.
In this section we provide some basic facts about the set of such shifts.

\begin{defn}
    Let $T$ be a $d$-parameter $\Cnought$-semigroup over $\HilbertRaum$.
    The \highlightTerm{regular exponents} of $T$
    is the set $\opRegRevExponents{T} \subseteq \reals^{d}$
    of all $\boldsymbol{\omega}\in\reals^{d}$
    for which the $\boldsymbol{\omega}$-shift $E_{\boldsymbol{\omega}} \cdot T$ of $T$
    has a regular unitary dilation.
\end{defn}

In other words, $\opRegRevExponents{T}$ contains all information
about how to modify the dynamical system $T$ so that it
can be embedded via regular unitary dilations into reversible processes.
Note that by \Cref{thm:classification:sig:article-dilation-problem-raj-dahya}
    $\opRegRevExponents{T}$
can be equivalently formulated as the set of all
$\boldsymbol{\omega}\in\reals^{d}$
such that the $d$ generators of $E_{\boldsymbol{\omega}} \cdot T$
are completely dissipative.

In order to study $\opRegRevExponents{T}$,
we first show that
    ${\reals^{d} \ni \boldsymbol{\omega} \mapsto \beta_{E_{\boldsymbol{\omega}} \cdot T} \in \reals}$
is continuous.

\begin{prop}
    Let $T$ be a $d$-parameter $\Cnought$-semigroup over $\HilbertRaum$
    with bounded generators $A_{1},A_{2},\ldots,A_{d}$.
    Then

        \begin{maths}[mc]{c}
        \eqtag[eq:continuity-of-beta-T:sig:article-dilation-problem-raj-dahya]
                \displaystyle
                \prod_{i=1}^{d}
                    \norm{A_{i} - \omega_{i}}
                -
                \displaystyle
                \prod_{i=1}^{d}
                    (\abs{\omega_{i}} + \norm{A_{i} - \omega_{i}})
            \leq
                \beta_{E_{\boldsymbol{\omega}} \cdot T} - \beta_{T}
            \leq
                \displaystyle
                    \displaystyle
                    \prod_{i=1}^{d}
                        (\abs{\omega_{i}} + \norm{A_{i}})
                    -
                    \displaystyle
                    \prod_{i=1}^{d}
                        \norm{A_{i}}\\
        \end{maths}

    \continueparagraph
    for all $\boldsymbol{\omega} \in \reals^{d}$.
    In particular,
        $\boldsymbol{\omega} \mapsto \beta_{E_{\boldsymbol{\omega}} \cdot T}$
    is continuous.
\end{prop}

    \begin{proof}
        Let $K\subseteq\{1,2,\ldots,d\}$.
        By \eqcref{eq:self-similarity-dissipation-operators:sig:article-dilation-problem-raj-dahya}
        in \Cref{prop:algebraic-identity-scaled-shifts:sig:article-dilation-problem-raj-dahya}
        one has
            $
                S_{E_{\boldsymbol{\omega}} \cdot T, K}
                =
                S_{T, K}
                + \sum_{K^{\prime} \subsetneq K}
                    \omega_{K \without K^{\prime}}
                    S_{T, K^{\prime}}
            $.
        Thus

            \begin{longmaths}[mc]{RCL}
                S_{T, K}
                    &=
                        &S_{E_{\boldsymbol{\omega}} \cdot T, K}
                        - \displaystyle
                        \sum_{K^{\prime} \subsetneq K}
                            \omega_{K \without K^{\prime}}
                            S_{T, K^{\prime}}\\
                    &\geq
                        &\beta_{E_{\boldsymbol{\omega}} \cdot T}\onematrix
                        - \displaystyle
                        \sum_{K^{\prime} \subsetneq K}
                            \abs{\omega_{K \without K^{\prime}}}
                            \norm{S_{T, K^{\prime}}}
                            \,\onematrix\\
                    &\overset{(\ast)}{\geq}
                        &\beta_{E_{\boldsymbol{\omega}} \cdot T}\onematrix
                        - \displaystyle
                        \sum_{K^{\prime} \subsetneq K}
                            \abs{\omega_{K \without K^{\prime}}}
                            \prod_{i \in K^{\prime}}\norm{A_{i}}
                            \,\onematrix\\
                    &\geq
                        &\beta_{E_{\boldsymbol{\omega}} \cdot T}\onematrix
                        - \displaystyle
                        \sum_{K^{\prime} \subsetneq \{1,2,\ldots,d\}}
                            \abs{\omega_{K \without K^{\prime}}}
                            \prod_{i \in K^{\prime}}\norm{A_{i}}
                            \,\onematrix\\
                    &=
                        &\beta_{E_{\boldsymbol{\omega}} \cdot T}\onematrix
                        - \Big(
                            \displaystyle
                            \prod_{i=1}^{d}
                                (\norm{A_{i}} + \abs{\omega_{i}})
                            -
                            \displaystyle
                            \prod_{i=1}^{d}
                                \norm{A_{i}}
                        \Big)
                        \onematrix,\\
            \end{longmaths}

        \continueparagraph
        where ($\ast$) holds, since each $S_{T,K^{\prime}}$ is an average
        of products of $\{A^{-}_{i},(A^{-}_{i})^{\ast}\mid i \in K^{\prime}\}$
        with each index occurring exactly once in each product,
        and the final equality holds by the binomial expansion.
        Since this inequality holds for all $K\subseteq\{1,2,\ldots,d\}$,
        it follows that

            \begin{maths}[mc]{rcl}
                \beta_{E_{\boldsymbol{\omega}} \cdot T}
                    \leq
                        \beta_{T}
                        +
                        \displaystyle
                        \prod_{i=1}^{d}
                            (\norm{A_{i}} + \abs{\omega_{i}})
                        -
                        \displaystyle
                        \prod_{i=1}^{d}
                            \norm{A_{i}},\\
            \end{maths}

        \continueparagraph
        and hence the right-hand inequality
        in \eqcref{eq:continuity-of-beta-T:sig:article-dilation-problem-raj-dahya} holds.

        Let $\tilde{T} \colonequals E_{\boldsymbol{\omega}} \cdot T$.
        From the generality of the right-hand inequality in \eqcref{eq:continuity-of-beta-T:sig:article-dilation-problem-raj-dahya},
        observing that $\tilde{T}$ has generators
            $\tilde{A}_{i} = A_{i} - \omega_{i}$
        for each $i \in \{1,2,\ldots,d\}$
        and $E_{-\boldsymbol{\omega}}\cdot \tilde{T} = T$,
        one readily obtains the left-hand inequality in \eqcref{eq:continuity-of-beta-T:sig:article-dilation-problem-raj-dahya}.

        From
            \eqcref{eq:continuity-of-beta-T:sig:article-dilation-problem-raj-dahya}
        it is clear that
            ${\abs{\beta_{E_{\boldsymbol{\omega}} \cdot T} - \beta_{T}} \longrightarrow 0}$
        for
            ${\reals^{d}\ni\boldsymbol{\omega} \longrightarrow \zerovector}$.
        Letting $\boldsymbol{\omega}\in\reals^{d}$ be arbitrary
        and $\tilde{T} \colonequals E_{\boldsymbol{\omega}}\cdot T$,
        if we consider
            ${\reals^{d}\ni\boldsymbol{\omega}^{\prime} \longrightarrow \boldsymbol{\omega}}$,
        then
            ${\reals^{d}\ni\boldsymbol{\omega}^{\prime}-\boldsymbol{\omega} \longrightarrow \zerovector}$
        and hence
            $
                \beta_{E_{\boldsymbol{\omega}^{\prime}} \cdot T}
                =\beta_{E_{\boldsymbol{\omega}^{\prime}-\boldsymbol{\omega}} \cdot \tilde{T}}
                \longrightarrow
                \beta_{\tilde{T}}
                = \beta_{E_{\boldsymbol{\omega}} \cdot T}
            $.
        Thus the map is continuous in all points.
    \end{proof}

\begin{prop}
    Let $T$ be a $d$-parameter $\Cnought$-semigroup over $\HilbertRaum$ with bounded generators.
    Then the following hold:

    \begin{kompaktenum}{\bfseries (i)}[\rtab]
        \item\punktlabel{1}
            The set $\opRegRevExponents{T}$ is non-empty and closed.
        \item\punktlabel{2}
            The inclusion
                $\opRegRevExponents{T}\supseteq\boldsymbol{\omega} + \realsNonNeg^{d}$
            holds for each $\boldsymbol{\omega} \in \opRegRevExponents{T}$.
        \item\punktlabel{3}
            For each
                $\boldsymbol{\omega} \in \reals^{d}$
                and
                $i\in\{1,2,\ldots,d\}$
            there exists $r\in\reals$,
            such that
                $\boldsymbol{\omega} - t\mathbf{e}_{i} \notin \opRegRevExponents{T}$
            for all $t\in(r,\:\infty)$.
        \item\punktlabel{4}
            The interior $\topInterior{\opRegRevExponents{T}}$
            is non-empty and consists of all $\boldsymbol{\omega}\in\reals^{d}$
            for which $E_{\boldsymbol{\omega}} \cdot T$ has completely super dissipative generators.
        \item\punktlabel{5}
            The boundary $\da\opRegRevExponents{T}$
            consists of all shifts $\boldsymbol{\omega}\in\reals^{d}$
            for which $\beta_{\boldsymbol{\omega} \cdot T}=0$.
        \item\punktlabel{6}
            The set $\opRegRevExponents{T}$
            is completely determined by the boundary via
                $\opRegRevExponents{T} = \da\opRegRevExponents{T} + \realsNonNeg^{d}$.
    \end{kompaktenum}

    \nvraum{1}

\end{prop}

    \begin{proof}
        \paragraph{\punktcref{1}:}
        By
            \Cref{cor:dilatable-up-to-exponential-modification:sig:article-dilation-problem-raj-dahya}
        $\opRegRevExponents{T}$ is non-empty.
        Since
            $%
                \opRegRevExponents{T} = \{
                    \boldsymbol{\omega}\in\reals^{d}
                    \mid
                    \beta_{E_{\boldsymbol{\omega}} \cdot T} \geq 0
                \}
            $
        and
            $\boldsymbol{\omega} \mapsto \beta_{E_{\boldsymbol{\omega}} \cdot T}$
        is continuous, it follows that $\opRegRevExponents{T}$ is closed.

        \paragraph{\punktcref{2}:}
        Fix an arbitrary $\boldsymbol{\omega}\in\opRegRevExponents{T}$.
        By closedness, it suffices to show that
            $\opRegRevExponents{T} \supseteq \boldsymbol{\omega} + \realsPos^{d}$.
        Since
            $\tilde{T} \coloneq E_{\boldsymbol{\omega}} \cdot T$
        has completely dissipative generators,
        by \Cref{lemm:self-similarities-imply-perturbations-yield-positive-beta:sig:article-dilation-problem-raj-dahya}
            $%
                E_{\boldsymbol{\omega}^{\prime} + \boldsymbol{\omega}} \cdot T
                = E_{\boldsymbol{\omega}^{\prime}} \cdot \tilde{T}
            $
        has completely (super) dissipative generators
        for $\boldsymbol{\omega}^{\prime}\in\realsPos^{d}$.
        Thus $\opRegRevExponents{T} \supseteq \boldsymbol{\omega} + \realsPos^{d}$.

        \paragraph{\punktcref{3}:}
        Let $i\in\{1,2,\ldots,d\}$ and $\boldsymbol{\omega} \in \reals^{d}$.
        Fix some
            $%
                \lambda \in \opSpectrum{\Re (A_{i} - \omega_{i}\onematrix)}
                \subseteq \reals
            $
        and set $r \colonequals -\lambda$.
        For each $t \in (r,\:\infty)$
        one has
            $%
                \opSpectrum{\Re (A_{i} - (\omega_{i} - t)\onematrix)}
                = \opSpectrum{\Re (A_{i} - \omega_{i})} + t
                \ni -r + t
            $
        and hence
            $\opSpectrum{\Re A_{i} - (\omega_{i} - t)\onematrix} \cap \realsPos \neq \emptyset$.
        It follows that
            $\Re (A_{i} - (\omega_{i} - t)\onematrix)$
        is non-dissipative
        and thus by the Lumer-Phillips form of the Hille-Yosida theorem
        (\cf \Cref{rem:lumer-phillips:sig:article-dilation-problem-raj-dahya})
        the $1$-parameter $\Cnought$-semigroup, $\tilde{T}_{i}$,
        induced by the generator $A_{i} - (\omega_{i} - t)\onematrix$
        is not contractive.
        Thus
            $E_{\boldsymbol{\omega} - t\mathbf{e}_{i}} \cdot T$,
        which has $\tilde{T}_{i}$ has its $i$th marginal,
        is not contractive and thus cannot have a regular unitary dilation.
        Hence
            $\boldsymbol{\omega} - t\mathbf{e}_{i} \notin \opRegRevExponents{T}$
        for all $t \in (r,\:\infty)$.

        \paragraph{\punktcref{4}:}
        The non-emptiness of the interior clearly follows
        the non-emptiness of $\opRegRevExponents{T}$.
        and \punktcref{2}.
        By continuity
            $\{
                \boldsymbol{\omega}\in\reals^{d}
                \mid
                \beta_{E_{\boldsymbol{\omega}} \cdot T} > 0
            \}$
        is clearly open and contained in $\opRegRevExponents{T}$.
        Conversely, consider an arbitrary point $\boldsymbol{\omega}$
        in the interior of $\opRegRevExponents{T}$.
        Then $\prod_{i=1}^{d}(\omega_{i}-2r,\:\omega_{r}+2r)\subseteq\opRegRevExponents{T}$
        for some $r\in\realsPos$.
        So $\boldsymbol{\omega}^{\prime} \colonequals (\omega_{i}-r)_{i=1}^{d} \in \opRegRevExponents{T}$.
        So the generators of
            $\tilde{T} \coloneq E_{\boldsymbol{\omega}^{\prime}} \cdot T$
        are completely dissipative.
        Since
            $%
                \boldsymbol{\omega}^{\prime\prime}
                \colonequals
                    \boldsymbol{\omega} - \boldsymbol{\omega}^{\prime}
                \in
                    \realsPos^{d}
            $,
        by \Cref{lemm:self-similarities-imply-perturbations-yield-positive-beta:sig:article-dilation-problem-raj-dahya}
            $%
                E_{\omega} \cdot T
                = E_{\boldsymbol{\omega}^{\prime\prime}} \cdot \tilde{T}
            $
        has completely super dissipative generators.
        Hence the interior $\topInterior{\opRegRevExponents{T}}$
        coincides with
            $\{
                \boldsymbol{\omega}\in\reals^{d}
                \mid
                \beta_{E_{\boldsymbol{\omega}} \cdot T} > 0
            \}$.

        \paragraph{\punktcref{5}:}
        By \punktcref{1} and \punktcref{4}
        it follows that
        $%
            \da\opRegRevExponents{T}
            = \quer{\opRegRevExponents{T}}\without\topInterior{\opRegRevExponents{T}}
            = \opRegRevExponents{T}\without\topInterior{\opRegRevExponents{T}}
            = \{\boldsymbol{\omega}\in\reals^{d} \mid \beta_{E_{\boldsymbol{\omega}} \cdot T} \geq 0~\text{and}~\beta_{E_{\boldsymbol{\omega}} \cdot T} \not > 0\}
            = \{\boldsymbol{\omega}\in\reals^{d} \mid \beta_{E_{\boldsymbol{\omega}} \cdot T} = 0\}
        $.

        \paragraph{\punktcref{6}:}
        The $\supseteq$-inclusion follows from \punktcref{1} and \punktcref{2}.
        Towards the $\subseteq$-inclusion,
        let $\boldsymbol{\omega} \in \opRegRevExponents{T}$ be arbitrary.
        Repeated applications of \punktcref{3}
        yields a sufficiently large $t_{0}\in\realsPos$,
        such that
            ${\boldsymbol{\omega} - t_{0}\onevector \notin \opRegRevExponents{T}}$.
        Thus
            ${\beta_{E_{\boldsymbol{\omega}} \cdot T} \geq 0}$
        and
            ${\beta_{E_{\boldsymbol{\omega} - t_{0}\onevector} \cdot T} < 0}$.
        Since the map
            ${t\in\reals \mapsto \beta_{E_{\boldsymbol{\omega} - t\onevector} \cdot T}}$
        is continuous, by the intermediate value theorem
        there exists some $t^{\ast}\in[0,\:t_{0}]$,
        such that
            ${\beta_{E_{\boldsymbol{\omega} - t^{\ast}\onevector} \cdot T} = 0}$.
        By \punktcref{5},
            ${\boldsymbol{\omega} - t^{\ast}\onevector \in \da\opRegRevExponents{T}}$,
        and thus
            $%
                \boldsymbol{\omega}
                = (\boldsymbol{\omega} - t^{\ast}\onevector) + t^{\ast}\onevector
                \in \da\opRegRevExponents{T} + \realsNonNeg^{d}
            $.
    \end{proof}




\section[Examples]{Examples}
\label{sec:examples}


\firstparagraph
Throughout this section we work with
$d$-parameter $\Cnought$-semigroups $T$ over a Hilbert space $\HilbertRaum$
whose marginals $T_{1},T_{2},\ldots,T_{d}$
have bounded generators $A_{1},A_{2},\ldots,A_{d}$ respectively.

This section has three goals.
Working within well understood cases of
    $d=1$
    as well as when $d \geq 1$ and the marginals of $T$ doubly commute,
we first show that
    the equivalence
        \eqcref{it:1:thm:classification:sig:article-dilation-problem-raj-dahya}%
        ~$\Leftrightarrow$~%
        \eqcref{it:2:thm:classification:sig:article-dilation-problem-raj-dahya}
    in the classification theorem (\Cref{thm:classification:sig:article-dilation-problem-raj-dahya}) holds
    without appealing to it.
Within the case of
    doubly commuting
    (\respectively normal) marginal semigroups
we demonstrate secondly that
    the notion of complete dissipativity
    (\respectively complete super dissipativity)
    corresponds to dissipativity
    (\respectively strictly negative spectral bounds).
Finally, working outside these cases,
we shall apply the classification theorem
and demonstrate that the notions of complete (super) dissipativity are not equivalent to the mentioned properties,
and in the process produce concrete examples of $d$-parameter contractive $\Cnought$-semigroups
which are not regularly unitarily dilatable.



\subsection[One-parameter contractive semigroups]{One-parameter contractive semigroups}
\label{sec:examples:simple}

\firstparagraph
Consider the case $d=1$.
Applying the definitions yields
    $B_{T,\emptyset}(t) = S_{T,\emptyset} = \onematrix$
and

    \begin{maths}[mc]{rcl}
        B_{T,\{1\}}(t)
            &= &\onematrix - T_{1}(t)^{\ast}\onematrix T_{1}(t)
            = \onematrix - T_{1}(t)^{\ast}T_{1}(t)
            ~\text{and}\\
        S_{T,\{1\}}
            &= &-\frac{1}{2}\{\Re A_{1},\:\onematrix\} + \imagunit \frac{1}{2}[\Imag A_{1},\:\onematrix]
            = -\Re A_{1}\\
    \end{maths}

\continueparagraph
for all $t\in\realsNonNeg$.

If $T$ is contractive, then clearly
    $B_{T,\emptyset}(t) \geq \zeromatrix$ and $B_{T,\{1\}}(t) \geq \zeromatrix$
for all $t\in\realsNonNeg$, whence $T$ satisfies the Brehmer positivity criterion
and thus by \Cref{thm:brehmer-iff-dilation:sig:article-dilation-problem-raj-dahya} has a regular unitary dilation.
Conversely, if $T$ has a regular unitary dilation, it is necessarily contractive.
Hence $T$ has a regular unitary dilation if and only if it is contractive.

The above expressions also demonstrate that
        the generator of $T$ is completely dissipative
    if and only if
        $\Re A_{1} \leq \zeromatrix$,
    \idest if and only if
        $A_{1}$ is dissipative,
    which, by the Lumer-Phillips form of the Hille-Yosida theorem
    (\cf \Cref{rem:lumer-phillips:sig:article-dilation-problem-raj-dahya})
    in turn holds
    if and only if
        $T$ is contractive.
We also see that
    the generator of $T$ is completely super dissipative
    if and only if
    $\Re A_{1} \leq -\beta\onematrix$ for some $\beta > 0$,
    which holds if and only if
    $T$ is super dissipative.

Hence for $d=1$
    the existence of a regular unitary dilation
    is equivalent to $T$ being contractive,
    which in turn
    is equivalent to the generator of $T$ being completely dissipative.
And one also has that complete super dissipativity and super dissipativity coincide.



\subsection[Doubly commuting generators]{Doubly commuting generators}
\label{sec:examples:doubly-comm}

\firstparagraph
Recall that operators
    $(X_{i})_{i} \subseteq \BoundedOps{\HilbertRaum}$
are said to \highlightTerm{doubly commute}
if $X_{i}$ commutes with $X_{j}$ and $X_{j}^{\ast}$
for all $i,j$ with $i \neq j$.
Suppose that the marginals $T_{1},T_{2},\ldots,T_{d}$
are doubly commuting semigroups,
\idest
    $(T_{1}(t_{1}),T_{2}(t_{2}),\ldots,T_{d}(t_{d}))$
    doubly commute for all $\mathbf{t} = (t_{i})_{i=1}^{d} \in \realsNonNeg^{d}$.
Then it is easy to see that the generators
    $(A_{1},A_{2},\ldots,A_{d})$
doubly commute
(recall that we also assume here that the $A_{i}$ are bounded).
The converse also clearly holds.
Observe further, that double commutativity trivially holds in the case of $d=1$.
Furthermore, by Fuglede's theorem
(\cf \cite[Proposition~4.4.12]{Pedersen1989analysisBook}),
if the commuting marginal semigroups are normal
(or equivalently: the commuting generators are normal),
then they doubly commute
(or equivalently: the generators doubly commute).

Under the assumption of doubly commuting marginals,
since for each
    $K\subseteq\{1,2,\ldots,d\}$ and $t\in\realsNonNeg$
the Brehmer operator
    $B_{T,K}(t)$
is an algebraic expression over
    $\{T_{i}(t),T_{i}(t)^{\ast}\mid i\in K\}$
and the dissipation operator
    $S_{T,K}$
is an algebraic expression over
    $\{A_{i},A_{i}^{\ast}\mid i\in K\}$,
one has that
    $(T_{\alpha}(t),B_{T,K}(t))$ doubly commute
and similarly
    $(A_{\alpha},S_{T,K})$ doubly commute
for all
    $\alpha\in\{1,2,\ldots,d\}\without K$.
Applying the recursive expressions in
    \eqcref{eq:recursion:brehmer:sig:article-dilation-problem-raj-dahya}
    and
    \eqcref{eq:recursion:dissipation:sig:article-dilation-problem-raj-dahya},
one thus obtains
    $B_{T,\emptyset}(t) = S_{T,\emptyset} = \onematrix$
and

    \begin{maths}[mc]{rcl}
        B_{T,K\cup\{\alpha\}}(t)
            &= &B_{T,K}(t) - T_{\alpha}(t)^{\ast}B_{T,K}(t)T_{\alpha}(t)\\
            &= &B_{T,K}(t) \cdot (\onematrix - T_{\alpha}(t)^{\ast}T_{\alpha}(t)),
        ~\text{and}\\
        S_{T,K\cup\{\alpha\}}(t)
            &= &-\frac{1}{2}(A_{\alpha}S_{T,K} + S_{T,K}A_{\alpha}^{\ast})\\
            &= &S_{T,K} \cdot (-\frac{1}{2}(A_{\alpha} + A_{\alpha}^{\ast}))\\
    \end{maths}

\continueparagraph
for all
    $K\subseteq\{1,2,\ldots,d\}$,
    $\alpha\in\{1,2,\ldots,d\}\without K$,
    and
    $t\in\realsNonNeg$.
By induction one thus obtains

    \begin{maths}[mc]{ccc}
        \eqtag[eq:brehmer-diss-operators-doubly-comm:sig:article-dilation-problem-raj-dahya]
        B_{T,K}(t)
            =
                \displaystyle
                \prod_{i \in K}
                    (\onematrix - T_{i}(t)^{\ast}T_{i}(t))
        &\text{and}
        &S_{T,K}
            =
                \displaystyle
                \prod_{i\in K}
                    (-\Re A_{i})\\
    \end{maths}

\continueparagraph
for all $K\subseteq\{1,2,\ldots,d\}$ and $t\in\realsNonNeg$.
Using these expressions yields the following results.

\begin{prop}
\makelabel{prop:doubly-comm-equivalences:sig:article-dilation-problem-raj-dahya}
    If the marginals of $T$ doubly commute (\exempli if $d=1$),
    then the following are equivalent:

    \begin{kompaktenum}{\bfseries (1)}[\rtab]
        \item\punktlabel{1}
            $T$ has a regular unitary dilation.
        \item\punktlabel{2}
            $T$ is contractive.
        \item\punktlabel{3}
            The generators of $T$ are completely dissipative.
        \item\punktlabel{4}
            The generators of $T$ are dissipative.
    \end{kompaktenum}

    \nvraum{1}
\end{prop}

    \begin{proof}
        The implication \hinRichtung{1}{2} is clear.
        And by the discussion in \S{}\ref{sec:definitions:dissipation},
        we already know that complete dissipativity clearly generalises dissipativity,
        \idest \hinRichtung{3}{4} holds.

        \paragraph{\hinherRichtung{2}{4}:}
        By the correspondence in the introduction,
            the $d$-parameter $\Cnought$-semigroup $T$ is contractive
        if and only if
            each of the marginals $T_{i}$ is contractive,
        which, by the Lumer-Phillips form of the Hille-Yosida theorem
        (\cf \Cref{rem:lumer-phillips:sig:article-dilation-problem-raj-dahya}),
        in turn holds
        if and only if
            each of the generators $A_{i}$ is dissipative.

        \paragraph{\hinRichtung{2}{1}:}
        If $T$ is contractive, or equivalently (by the correspondence mentioned in the introduction) each $T_{i}$ is contractive,
        then
            $\onematrix - T_{i}(t)^{\ast}T_{i}(t) \geq \zeromatrix$ for all $i\in\{1,2,\ldots,d\}$.
        Since by the double commutativity assumption
            $\{\onematrix - T_{i}(t)^{\ast}T_{i}(t) \mid i\in\{1,2,\ldots,d\}\}$
        is a set of commuting positive operators,
        it follows that any product of these is positive.\footref{ft:product-positive:sig:article-dilation-problem-raj-dahya}
        By \eqcref{eq:brehmer-diss-operators-doubly-comm:sig:article-dilation-problem-raj-dahya}
        it follows that
            $B_{T,K}(t) \geq \zeromatrix$
            for all $t\in\realsNonNeg$
            and all $K\subseteq\{1,2,\ldots,d\}$.
        So $T$ satisfies the Brehmer positivity criterion
        and by \Cref{thm:brehmer-iff-dilation:sig:article-dilation-problem-raj-dahya}
        it has a regular unitary dilation.

        \paragraph{\hinRichtung{4}{3}:}
        If each of the $A_{i}$ are dissipative, then $-\Re A_{i} \geq \zeromatrix$ for all $i\in\{1,2,\ldots,d\}$.
        Since by the double commutativity assumption
            $\{-\Re A_{i} \mid i\in\{1,2,\ldots,d\}\}$
        is a set of commuting positive operators,
        it follows that any product of these is positive.\footref{ft:product-positive:sig:article-dilation-problem-raj-dahya}
        By \eqcref{eq:brehmer-diss-operators-doubly-comm:sig:article-dilation-problem-raj-dahya}
        it follows that
            $S_{T,K} \geq \zeromatrix$
            for all $K\subseteq\{1,2,\ldots,d\}$.
        So the generators of $T$ are completely dissipative.
    \end{proof}

\begin{prop}
\makelabel{prop:doubly-comm-equivalences-strict:sig:article-dilation-problem-raj-dahya}
    If the marginals of $T$ are normal,
    then the following are equivalent:

    \begin{kompaktenum}{\bfseries (1)}[\rtab]
        \item\punktlabel{1}
            The generators of $T$ have strictly negative spectral bounds.
        \item\punktlabel{2}
            The generators of $T$ are completely super dissipative.
        \item\punktlabel{3}
            The generators of $T$ are super dissipative.
    \end{kompaktenum}

    \continueparagraph
    Note that the equivalence of \punktcref{2} and \punktcref{3}
    holds under the weaker assumption of double commutativity.
    In particular, if $d=1$, \hinherRichtung{2}{3} holds without any assumptions.
\end{prop}

\footnotetext[ft:product-positive:sig:article-dilation-problem-raj-dahya]{%
    Let $\cal{A}$ be a unital $C^{\ast}$-algebra
    and $a_{1},a_{2},\ldots,a_{n}\in\cal{A}$ be commuting positive elements.
    Then the elements the $C^{\ast}$-subalgebra
    generated by these elements is commutative.
    In particular $\sqrt{a_{1}},\sqrt{a_{2}},\ldots,\sqrt{a_{n}}$
    are commutative positive elements of $\cal{A}$,
    whence
        $%
            \prod_{i=1}^{n}a_{i}
            = \prod_{i=1}^{n}(\sqrt{a}_{i}^{\ast}\sqrt{a}_{i})
            = \prod_{i=1}^{n}\sqrt{a}_{i}^{\ast}
            \prod_{j=1}^{n}\sqrt{a}_{j}
            = (\prod_{i=1}^{n}\sqrt{a}_{i})^{\ast}
            \prod_{j=1}^{n}\sqrt{a}_{j}
        $,
    which is positive.
    Note that positive elements are strictly positive
    if and only if they are invertible (\cf \cite[Proposition~3.2.12]{Pedersen1989analysisBook}).
    Since the product of invertible (\respectively commuting positive) elements
    is invertible (\respectively positive),
    it follows that the product of commuting strictly positive elements
    is strictly positive.%
}
    \begin{proof}
        First note that normality of the commuting marginal semigroups
        implies that the marginals double commute (see above).

        \paragraph{\hinherRichtung{1}{3}:}
        Let $i\in\{1,2,\ldots,d\}$.
        The generator, $A_{i}$, of $T_{i}$ has a strictly negative spectral bound,
        \idest
            $
                \opSpectrum{A_{i}}
                \subseteq
                \{z\in\complex \mid \Re z \leq -\omega\}
            $
        for some $\omega\in\realsPos$,
        if and only if (by the assumption of normality)
            $\Re A_{i} \leq -\omega\onematrix$
        for some $\omega\in\realsPos$,
        which in turn holds if and only if $A_{i}$ is super dissipative
        (\cf \S{}\ref{sec:definitions:dissipation}).

        \paragraph{\hinRichtung{2}{3}:}
        Complete super dissipativity clearly generalises super dissipativity
        (\cf \S{}\ref{sec:definitions:dissipation}).

        \paragraph{\hinRichtung{3}{2}:}
        If the generators of $T$ are super dissipative,
        then for some $\boldsymbol{\omega}\in\realsPos^{d}$
        we have
            $-\Re A_{i} \geq \omega_{i}\onematrix$
        for each $i\in\{1,2,\ldots,d\}$.
        Since by double commutativity
            $\{-\Re A_{i} \mid i\in\{1,2,\ldots,d\}\}$
        is a set of commuting strictly positive operators,
        any product of these is strictly positive.\footref{ft:product-positive:sig:article-dilation-problem-raj-dahya}
        By \eqcref{eq:brehmer-diss-operators-doubly-comm:sig:article-dilation-problem-raj-dahya}
        it follows that
            $S_{T,K} > 0$
        for each $K\subseteq\{1,2,\ldots,d\}$.
        (More precisely, $S_{T,K} \geq \omega_{K}\onematrix$
        for each $K\subseteq\{1,2,\ldots,d\}$.)
        It follows that $\beta_{T} > 0$
        and the generators of $T$ are completely super dissipative.
    \end{proof}

Observe that
    \eqcref{it:1:prop:doubly-comm-equivalences:sig:article-dilation-problem-raj-dahya}%
    ~$\Leftrightarrow$~%
    \eqcref{it:3:prop:doubly-comm-equivalences:sig:article-dilation-problem-raj-dahya}
    in
    \Cref{prop:doubly-comm-equivalences:sig:article-dilation-problem-raj-dahya}
agrees with
    \eqcref{it:1:thm:classification:sig:article-dilation-problem-raj-dahya}%
    ~$\Leftrightarrow$~%
    \eqcref{it:2:thm:classification:sig:article-dilation-problem-raj-dahya}
in the classification theorem (\Cref{thm:classification:sig:article-dilation-problem-raj-dahya}).
Moreover,
    \Cref{prop:doubly-comm-equivalences:sig:article-dilation-problem-raj-dahya}
    and
    \Cref{prop:doubly-comm-equivalences-strict:sig:article-dilation-problem-raj-dahya}
demonstrate under the assumption of double commutativity
the equivalence of
    complete dissipativity of the generators of $T$
    and
    the dissipativity of the individual generators,
as well as the equivalence of
    complete super dissipativity,
    super dissipativity,
    and
    (under the assumption of normality)
    strictly negative spectral bounds.

Further note that if $T$ is a $d$-parameter unitary $\Cnought$-semigroup,
then its generators are purely imaginary,
\idest $\Re A_{i} = \zeromatrix$ for each $i\in\{1,2,\ldots,d\}$.
By \eqcref{eq:brehmer-diss-operators-doubly-comm:sig:article-dilation-problem-raj-dahya}
it follows that
    $S_{T,K}=\zeromatrix$ for all non-empty $K\subseteq\{1,2,\ldots,d\}$,
whence $\beta_{T}=0$.
So the generators are completely dissipative but not completely super dissipative.
Hence the latter property is strictly stronger than the former.



\subsection[Non-doubly commuting generators]{Non-doubly commuting generators}
\label{sec:examples:non-doubly-comm}

\firstparagraph
Suppose that $T$ is contractive and has bounded generators,
which furthermore have strictly negative spectral bounds.
Then, if the marginals of $T$ doubly commute,
by \Cref{prop:doubly-comm-equivalences:sig:article-dilation-problem-raj-dahya}
    $T$ is completely dissipative
    and thus by the classification theorem (\Cref{thm:classification:sig:article-dilation-problem-raj-dahya})
    has a regular unitary dilation.
We shall show that the first three assumptions do not together suffice
to yield regular unitary dilatability.
In order to achieve this, we clearly need to consider non-doubly commuting semigroups.
A typical place to look is \emph{upper} (or lower) \emph{triangular matrices}.

\begin{prop}
\makelabel{prop:non-doubly-comm-example:sig:article-dilation-problem-raj-dahya}
    Let $\HilbertRaum$ be a Hilbert space with $\dim(\HilbertRaum) \geq 2$,
    and let $d\in\naturals$ with $d \geq 2$.
    Then there exists a
        $d$-parameter contractive $\Cnought$-semigroup $T$ over $\HilbertRaum$
        with bounded generators
            which have strictly negative spectral bounds,
    such that $T$ has no regular unitary dilation.
\end{prop}

    \begin{proof}
        We first construct a space of operators over $\HilbertRaum$,
        from which we shall pick our generators.
        Since $\dim(\HilbertRaum) \geq 2$
        one can find orthonormal closed subspaces
            $\HilbertRaum_{1},\HilbertRaum_{2}\subseteq\HilbertRaum$
        with
            $0 < \dim(\HilbertRaum_{2}) \leq \dim(\HilbertRaum_{1})$
        and $\HilbertRaum=\HilbertRaum_{1}\bigoplus\HilbertRaum_{2}$.
        Consider the following collection of operators

            \begin{maths}[mc]{c}
                \AlgebraUpperTr
                    =
                        \displaystyle
                        \bigcup_{c\in\complex^{\times}}
                            \underbrace{
                            \{
                                c \cdot \onematrix
                                +
                                \begin{smatrix}
                                \zeromatrix &D\\
                                \zeromatrix &\zeromatrix\\
                                \end{smatrix}
                                \mid
                                D \in \BoundedOps{\HilbertRaum_{2}}{\HilbertRaum_{1}}
                            \}
                        }_{\equalscolon \AlgebraUpperTr_{c}}.\\
            \end{maths}

        \continueparagraph
        and observe for
            $D, D_{1}, D_{2} \in \BoundedOps{\HilbertRaum_{2}}{\HilbertRaum_{1}}$
            and
            $E\in\BoundedOps{\HilbertRaum_{2}}$
        the following algebraic relations:

            \begin{maths}[mc]{c}
                \eqtag[eq:0:\beweislabel]
                \begin{smatrix}
                    \zeromatrix &D_{1}\\
                    \zeromatrix &\zeromatrix\\
                \end{smatrix}
                \begin{smatrix}
                    \zeromatrix &D_{2}\\
                    \zeromatrix &\zeromatrix\\
                \end{smatrix}
                =
                    \begin{smatrix}
                        \zeromatrix &D_{1}\\
                        \zeromatrix &\zeromatrix\\
                    \end{smatrix}^{\ast}
                    \begin{smatrix}
                        \zeromatrix &D_{2}\\
                        \zeromatrix &\zeromatrix\\
                    \end{smatrix}^{\ast}
                =
                    \begin{smatrix}
                        \zeromatrix &\zeromatrix\\
                        \zeromatrix &E\\
                    \end{smatrix}
                    \begin{smatrix}
                        \zeromatrix &D\\
                        \zeromatrix &\zeromatrix\\
                    \end{smatrix}
                =
                    \begin{smatrix}
                        \zeromatrix &D\\
                        \zeromatrix &\zeromatrix\\
                    \end{smatrix}^{\ast}
                    \begin{smatrix}
                        \zeromatrix &\zeromatrix\\
                        \zeromatrix &E\\
                    \end{smatrix}
                = \zeromatrix\\
            \end{maths}

        \continueparagraph
        from which one can derive

            \begin{maths}[mc]{rcl}
            \eqtag[eq:1:\beweislabel]
                \Big(
                    \onematrix
                    + \begin{smatrix}
                    \zeromatrix &D\\
                    \zeromatrix &\zeromatrix\\
                    \end{smatrix}
                \Big)^{\ast}
                    &=
                    &\onematrix
                    + \begin{smatrix}
                        \zeromatrix &\zeromatrix\\
                        D^{\ast} &\zeromatrix\\
                    \end{smatrix},\\
                \Big(
                    \onematrix
                    + \begin{smatrix}
                    \zeromatrix &D_{1}\\
                    \zeromatrix &\zeromatrix\\
                    \end{smatrix}
                \Big)
                \cdot
                \Big(
                    \onematrix
                    + \begin{smatrix}
                    \zeromatrix &D_{2}\\
                    \zeromatrix &\zeromatrix\\
                    \end{smatrix}
                \Big)
                    &=
                    &\onematrix
                    +
                    \begin{smatrix}
                        \zeromatrix &D_{1} + D_{2}\\
                        \zeromatrix &\zeromatrix\\
                    \end{smatrix},\\
                \Big(
                    \onematrix
                    + \begin{smatrix}
                    \zeromatrix &D_{1}\\
                    \zeromatrix &\zeromatrix\\
                    \end{smatrix}
                \Big)^{\ast}
                \cdot
                \Big(
                    \onematrix
                    + \begin{smatrix}
                    \zeromatrix &D_{2}\\
                    \zeromatrix &\zeromatrix\\
                    \end{smatrix}
                \Big)
                    &=
                    &\onematrix
                    +
                    \begin{smatrix}
                        \zeromatrix &D_{2}\\
                        D_{1}^{\ast} &D_{1}^{\ast}D_{2}\\
                    \end{smatrix},\\
                \Big(
                    \onematrix
                    + \begin{smatrix}
                    \zeromatrix &D_{2}\\
                    \zeromatrix &\zeromatrix\\
                    \end{smatrix}
                \Big)
                \cdot
                \Big(
                    \onematrix
                    + \begin{smatrix}
                    \zeromatrix &D_{1}\\
                    \zeromatrix &\zeromatrix\\
                    \end{smatrix}
                \Big)^{\ast}
                    &=
                    &\onematrix
                    +
                    \begin{smatrix}
                        D_{2}D_{1}^{\ast} &D_{2}\\
                        D_{1}^{\ast} &\zeromatrix\\
                    \end{smatrix},\\
            \end{maths}

        \continueparagraph
        and thus

            \begin{maths}[mc]{rcl}
            \eqtag[eq:2:\beweislabel]
                \Big(
                    \onematrix
                    + \begin{smatrix}
                    \zeromatrix &D\\
                    \zeromatrix &\zeromatrix\\
                    \end{smatrix}
                \Big)
                \cdot
                \Big(
                    \onematrix
                    + \begin{smatrix}
                    \zeromatrix &-D\\
                    \zeromatrix &\zeromatrix\\
                    \end{smatrix}
                \Big)
                    &=
                    &\onematrix
                    +
                    \begin{smatrix}
                        \zeromatrix &D + (-D)\\
                        \zeromatrix &\zeromatrix\\
                    \end{smatrix}
                    = \onematrix.\\
            \end{maths}

        By \eqcref{eq:1:\beweislabel} and \eqcref{eq:2:\beweislabel},
        $\AlgebraUpperTr$ forms a commutative group under multiplication
        and the elements in $\AlgebraUpperTr$ in general are not normal
        nor do they doubly commute.

        \paragraph{Properties of the generators:}
        The spectra of the elements are straightforward to compute:
        For $c\in\complex^{\times}$ and
        ${%
            A \colonequals
            c \cdot \onematrix
            + \begin{smatrix}
                \zeromatrix &D\\
                \zeromatrix &\zeromatrix\\
                \end{smatrix}
            \in \AlgebraUpperTr_{c}
        }$
        and $\lambda \in \complex \without \{c\}$
        one clearly has
        $
            \lambda \cdot \onematrix - A
            \in \AlgebraUpperTr_{\lambda - c}
            \subseteq \AlgebraUpperTr
        $,
        and thus $\lambda \cdot \onematrix - A$ is invertible.
        So $\opSpectrum{A} \subseteq \{c\}$.
        Since the spectrum of bounded operators is non-empty,
        it follows that $\opSpectrum{A} = \{c\}$
        for all $A \in \AlgebraUpperTr_{c}$
        and all $c\in\complex^{\times}$.
        In particular,
        $\opSpectrum{A} = \{-1\} \subseteq\{z\in\complex \mid \Re z \leq -1\}$,
        for all $A \in \AlgebraUpperTr_{-1}$,
        whence each operator in $\AlgebraUpperTr_{-1}$
        generates a $\Cnought$-semigroup
        whose generator has a strictly negative spectral bound.

        For
        ${%
            A \colonequals
            -1 \cdot \onematrix
            + \begin{smatrix}
                \zeromatrix &D\\
                \zeromatrix &\zeromatrix\\
                \end{smatrix}
            \in \AlgebraUpperTr_{-1}
        }$
        one has

            \begin{maths}[mc]{rcl}
                \Re \BRAKET{A(x\oplus y)}{x\oplus y}
                    &=
                        &-\norm{x\oplus y}^{2} + \Re \BRAKET{Dy}{x}\\
                    &\leq
                        &-\norm{x}^{2} - \norm{y}^{2}
                        + \norm{D}\norm{x}\norm{y}\\
                    &=
                        &-(\norm{x} - \norm{y})^{2}
                            -(2 - \norm{D})\norm{x}\norm{y}\\
            \end{maths}

        \continueparagraph
        for $x\in\HilbertRaum_{1}$ und $y\in\HilbertRaum_{2}$.
        Thus $A$ is dissipative, provided $\norm{D}\leq 2.$

        This establishes a broad class of possibilities for generating
        our $d$-parameter $\Cnought$-semigroups.
        Let
            $A_{1},A_{2},\ldots,A_{d} \in \AlgebraUpperTr_{-1}$
        with
            ${%
            A_{i}
            = -\onematrix
            + \begin{smatrix}
                \zeromatrix &-2D_{i}\\
                \zeromatrix &\zeromatrix\\
            \end{smatrix}
            }$
        for some contractive
            $D_{i} \in \BoundedOps{\HilbertRaum_{2}}{\HilbertRaum_{1}}$
        and let $T$ be the $d$-parameter $\Cnought$-semigroup whose marginals have
        the commuting operators $A_{1},A_{2},\ldots,A_{d}$ as generators.
        By the above, each $A_{i}$ is bounded, dissipative,
        and has a strictly negative spectral bound.
        By dissipativity and the Lumer-Phillips form of the Hille-Yosida theorem
        (\cf \Cref{rem:lumer-phillips:sig:article-dilation-problem-raj-dahya}),
        the marginals of $T$ and thus $T$ itself are contractive.
        By applying the algebraic relations in \eqcref{eq:1:\beweislabel}
        one can readily verify that $T$ is explicitly given by:

            \begin{maths}[mc]{c}
            \eqtag[eq:form-of-non-doubly-comm-semigroup:sig:article-dilation-problem-raj-dahya]
                T(\mathbf{t})
                    =
                        e^{-\sum_{i=1}^{d}t_{i}}
                        \cdot
                        \Big(
                            \onematrix
                            + \displaystyle
                            \sum_{i=1}^{d}
                                t_{i}(\onematrix + A_{i})
                        \Big)\\
            \end{maths}

        \continueparagraph
        for each $\mathbf{t}=(t_{i})_{i=1}^{d}\in\realsNonNeg^{d}$.

        \paragraph{Complete dissipativity:}
        To prove that $T$ has no regular unitary dilation,
        by the classification theorem (\Cref{thm:classification:sig:article-dilation-problem-raj-dahya})
        it suffices to show that the generators of $T$ are not completely dissipative.
        To assist with the computation of
        the dissipation operators associated to the generators of $T$,
        define the following aggregates

            \begin{maths}[mc]{cqc}
                \brkt{D}_{K}
                \colonequals
                    \displaystyle
                    \sum_{i\in K}D_{i},
                &\brkt{\abs{D}^{2}}_{K}
                \colonequals
                    \displaystyle
                    \sum_{i\in K}D_{i}^{\ast}D_{i}\\
            \end{maths}

        \continueparagraph
        for each $K \subseteq \{1,2,\ldots,d\}$.
        We now claim that the dissipativity operators are given by

            \begin{maths}[mc]{rcl}
            \eqtag[eq:form-of-dissipation-operators:sig:article-dilation-problem-raj-dahya]
                S_{T,K}
                    &=
                        &\Big(
                            \onematrix
                            +
                            \begin{smatrix}
                                \zeromatrix &\brkt{D}_{K}\\
                                \zeromatrix &\zeromatrix\\
                            \end{smatrix}
                        \Big)^{\ast}
                        \Big(
                            \onematrix
                            -
                            \underbrace{
                            \begin{smatrix}
                                \zeromatrix &\zeromatrix\\
                                \zeromatrix &\brkt{\abs{D}^{2}}_{K}\\
                            \end{smatrix}
                            }_{\equalscolon V_{K}}
                        \Big)
                        \underbrace{
                            \Big(
                                \onematrix
                                + \begin{smatrix}
                                    \zeromatrix &\brkt{D}_{K}\\
                                    \zeromatrix &\zeromatrix\\
                                \end{smatrix}
                            \Big)
                        }_{\equalscolon C_{K}}\\
                    &= &C_{K}^{\ast}C_{K} - C_{K}^{\ast}V_{K}C_{K}\\
                    &= &C_{K}^{\ast}C_{K} - V_{K}\\
            \end{maths}

        \continueparagraph
        holds for all $K \subseteq \{1,2,\ldots,d\}$.
        (%
            Note that the final equality holds by applying
            the identities in \eqcref{eq:0:\beweislabel} to
                $\brkt{D}_{K}$
                and
                $\brkt{\abs{D}^{2}}_{K}$.%
        )
        We prove this by induction over the size of $K$
        using the algebraic relations in
            \eqcref{eq:1:\beweislabel}
        and the recursions in
            \eqcref{eq:recursion:dissipation:sig:article-dilation-problem-raj-dahya}.
        For $K=\emptyset$ one has
            $%
                C_{\emptyset}^{\ast}C_{\emptyset} - V_{\emptyset}
                = (\onematrix + \zeromatrix)^{\ast}(\onematrix + \zeromatrix) - \zeromatrix
                = \onematrix
                = S_{T,\emptyset}
            $,
        so \eqcref{eq:form-of-dissipation-operators:sig:article-dilation-problem-raj-dahya} holds.
        And if \eqcref{eq:form-of-dissipation-operators:sig:article-dilation-problem-raj-dahya}
        holds for some $K \subseteq \{1,2,\ldots,d\}$,
        then for $\alpha \in \{1,2,\ldots,d\} \setminus K$

            \begin{longmaths}[mc]{RCL}
                S_{T,K \cup \{\alpha\}}
                    &\eqcrefoverset{eq:recursion:dissipation:sig:article-dilation-problem-raj-dahya}{=}
                        &\frac{1}{2}(-A_{\alpha}^{\ast})S_{T,K} + \frac{1}{2}S_{T,K}(-A_{\alpha})\\
                    &=
                        &\frac{1}{2}
                        \Big(
                            \onematrix
                            +
                            2
                            \begin{smatrix}
                                \zeromatrix &D_{\alpha}\\
                                \zeromatrix &\zeromatrix\\
                            \end{smatrix}^{\ast}
                        \Big)
                        S_{T,K}
                        +
                        \frac{1}{2}
                        S_{T,K}
                        \Big(
                            \onematrix
                            +
                            2
                            \begin{smatrix}
                                \zeromatrix &D_{\alpha}\\
                                \zeromatrix &\zeromatrix\\
                            \end{smatrix}
                        \Big)\\
                    &\textoverset{ind.}{=}
                        &(C_{K}^{\ast}C_{K} - V_{K})
                        +
                        \begin{smatrix}
                            \zeromatrix &D_{\alpha}\\
                            \zeromatrix &\zeromatrix\\
                        \end{smatrix}^{\ast}
                        (C_{K}^{\ast}C_{K} - V_{K})
                        +
                        (C_{K}^{\ast}C_{K} - V_{K})
                        \begin{smatrix}
                            \zeromatrix &D_{\alpha}\\
                            \zeromatrix &\zeromatrix\\
                        \end{smatrix}\\
                    &\textoverset{($\ast$)}{=}
                        &(C_{K}^{\ast}C_{K} - V_{K})
                        +
                        \Big(
                            \begin{smatrix}
                                \zeromatrix &D_{\alpha}\\
                                \zeromatrix &\zeromatrix\\
                            \end{smatrix}^{\ast}
                            C_{K}
                            - \zeromatrix
                        \Big)
                        +
                        \Big(
                            C_{K}^{\ast}
                            \begin{smatrix}
                                \zeromatrix &D_{\alpha}\\
                                \zeromatrix &\zeromatrix\\
                            \end{smatrix}
                            - \zeromatrix
                        \Big)\\
                    &=
                        &(C_{K}^{\ast}C_{K} - V_{K})
                        +
                        \begin{smatrix}
                            \zeromatrix &D_{\alpha}\\
                            \zeromatrix &\zeromatrix\\
                        \end{smatrix}^{\ast}
                        C_{K}
                        +
                        C_{K}^{\ast}
                        \begin{smatrix}
                            \zeromatrix &D_{\alpha}\\
                            \zeromatrix &\zeromatrix\\
                        \end{smatrix}\\
                    &=
                        &\Big(
                            C_{K}
                            +
                            \begin{smatrix}
                                \zeromatrix &D_{\alpha}\\
                                \zeromatrix &\zeromatrix\\
                            \end{smatrix}
                        \Big)^{\ast}
                        \Big(
                            C_{K}
                            +
                            \begin{smatrix}
                                \zeromatrix &D_{\alpha}\\
                                \zeromatrix &\zeromatrix\\
                            \end{smatrix}
                        \Big)
                        - V_{K}
                        - \begin{smatrix}
                            \zeromatrix &D_{\alpha}\\
                            \zeromatrix &\zeromatrix\\
                        \end{smatrix}^{\ast}
                        \begin{smatrix}
                            \zeromatrix &D_{\alpha}\\
                            \zeromatrix &\zeromatrix\\
                        \end{smatrix}\\
                    &=
                        &C_{K \cup \{\alpha\}}^{\ast}C_{K \cup \{\alpha\}}
                        - \Big(
                            V_{K}
                            +
                            \begin{smatrix}
                                \zeromatrix &\zeromatrix\\
                                \zeromatrix &D_{\alpha}^{\ast}D_{\alpha}\\
                            \end{smatrix}
                        \Big)\\
                    &=
                        &C_{K \cup \{\alpha\}}^{\ast}
                        C_{K \cup \{\alpha\}}
                        - V_{K \cup \{\alpha\}},\\
            \end{longmaths}

        \continueparagraph
        where the cancellations in ($\ast$) hold by applying
        \eqcref{eq:0:\beweislabel} to
            $\brkt{D}_{K}$,
            $D_{\alpha}$,
            and
            $\brkt{\abs{D}^{2}}_{K}$.
        This computation shows that
            \eqcref{eq:form-of-dissipation-operators:sig:article-dilation-problem-raj-dahya} holds for $K\cup\{\alpha\}$.
        It follows by induction that \eqcref{eq:form-of-dissipation-operators:sig:article-dilation-problem-raj-dahya}
        holds for all $K \subseteq \{1,2,\ldots,d\}$.

        Let $K \subseteq \{1,2,\ldots,d\}$ be arbitrary.
        Observe that
            $C_{K} \in \AlgebraUpperTr_{1}$,
        and thus by the above discussions, $C_{K}$ is invertible.
        By \eqcref{eq:form-of-dissipation-operators:sig:article-dilation-problem-raj-dahya}
        one has
            $S_{T,K} = C_{K}^{\ast}(\onematrix - V_{K})C_{K}$.
        Since $S_{T,K}$ and ${\onematrix - V_{K}}$ are self-adjoint
        and $C_{K}$ is invertible,
        it follows that the signs of the minimal spectral values
        of $S_{T,K}$ and ${\onematrix - V_{K}}$ coincide.\footnote{
            Let $\cal{A}$ be a unital $C^{\ast}$-algebra,
            $a\in\cal{A}$ be self-adjoint,
            and $c\in\cal{A}$ be invertible.
            Set
                ${\lambda  \colonequals \min(\opSpectrum{a})}$,
                ${\lambda^{\prime} \colonequals \min(\opSpectrum{c^{\ast}ac})}$,
                ${r \colonequals \min(\opSpectrum{c^{\ast}c}) > 0}$,
                and
                ${r^{\prime} \colonequals \min(\opSpectrum{(c^{-1})^{\ast}c^{-1}}) > 0}$.
            Then
                $%
                    c^{\ast} a c
                    \geq
                        c^{\ast}\cdot \lambda \onematrix \cdot c
                    =
                        \lambda c^{\ast}c
                    \geq
                        \lambda r\onematrix
                $
            so $\lambda^{\prime} = \min(\opSpectrum{c^{\ast}ac}) \geq \lambda r$.
            Since $a = (c^{-1})^{\ast}(c^{\ast} a c)c^{-1}$,
            one similarly obtains
                $\lambda \geq \lambda^{\prime}r^{\prime}$.
            Since $r,r^{\prime}>0$, it follows that $\signum{\lambda}=\signum{\lambda^{\prime}}$.
        }
        To compute the latter observe that
            $%
                \opSpectrum{\onematrix - V_{K}}
                = 1 - \opSpectrum{V_{K}}
                = 1 - \{0\} \cup \opSpectrum{\brkt{\card{D}^{2}}_{K}}
                = \{1\} \cup (1 - \opSpectrum{\brkt{\card{D}^{2}}})
            $,
        whence by the positivity of
        ${%
            \brkt{\card{D}^{2}}_{K}
            = \sum_{i \in K}D_{i}^{\ast}D_{i}
        }$
        the spectral theorem yields

            \begin{maths}[mc]{c}
                \min\opSpectrum{\onematrix - V_{K}}
                    = 1 - \max(\opSpectrum{\brkt{\card{D}^{2}}_{K}})
                    = 1 - \norm{\brkt{\card{D}^{2}}_{K}}
                    = 1 - \normLong{
                        \displaystyle
                        \sum_{i \in K}D_{i}^{\ast}D_{i}
                    }.
            \end{maths}

        For each $i\in\{1,2,\ldots,d\}$
        we now choose $D_{i}=\alpha V_{i}$
        where $V_{i}$ is an isometry and $\alpha\in\complex$
        with
            $%
                \abs{\alpha}
                \in (\frac{1}{\sqrt{d}},\:\frac{1}{\sqrt{d-1}})
                \subseteq [0, 1]
            $.
        This can be achieved in the above construction,
        as we simply required that each
            $D_{i}\in\BoundedOps{\HilbertRaum_{2}}{\HilbertRaum_{1}}$
        be contractive and since $0 < \dim(\HilbertRaum_{2}) \leq \dim(\HilbertRaum_{1})$.
        For $K\subseteq\{1,2,\ldots,d\}$
        our choice yields
            $%
                \norm{\sum_{i \in K} D_{i}^{\ast}D_{i}}
                = \norm{\sum_{i \in K} \abs{\alpha}^{2}\onematrix}
                = \abs{\alpha}^{2}\card{K}
            $
        and thus by the above computation
            $%
                \signum{\min(\opSpectrum{S_{T,K}})}
                = \signum{\min(\opSpectrum{\onematrix - V_{K}})}
                = \signum{1 - \abs{\alpha}^{2}\card{K}}
            $.
        By our choice of $\alpha$ we have that
            $\signum{\min(\opSpectrum{S_{T,K}})} = +1$
        for $K\subsetneq\{1,2,\ldots,d\}$
        and $\signum{\min(\opSpectrum{S_{T,K}})} = -1$
        for $K=\{1,2,\ldots,d\}$.
        In particular,
            $%
                \beta_{T}
                =
                    \min_{K\subseteq\{1,2,\ldots,d\}}
                    \min(\opSpectrum{S_{T,K}})
                < 0
            $.

        For any such construction,
        one thus has that $T$ satisfies all the assumptions
        and $\beta_{T} < 0$.
        That is, the generators of $T$ are not completely dissipative,
        and hence by \Cref{thm:classification:sig:article-dilation-problem-raj-dahya}
        $T$ has no regular unitary dilation.
    \end{proof}

\begin{rem}
    The above construction satisfied
        $\min(\opSpectrum{S_{T,K}}) \geq 0$
        for $\card{K} < d$
        and
        $\min(\opSpectrum{S_{T,K}}) < 0$
        for $\card{K} = d$.
    This shows that,
    to determine whether a $d$-parameter $\Cnought$-semigroup $T$
    has completely dissipative generators,
    the higher order dissipation operators are not redundant.
\end{rem}

\begin{rem}
\makelabel{rem:regular-dilation-stronger-than-unitary-dilation:sig:article-dilation-problem-raj-dahya}
    It is well known that all $2$-parameter contractive $\Cnought$-semigroups
    have unitary dilations (\cf
        \cite{Slocinski1974},
        \cite[Theorem~2]{Slocinski1982},
        and
        \cite[Theorem~2.3]{Ptak1985}%
    ).
    Thus \Cref{prop:non-doubly-comm-example:sig:article-dilation-problem-raj-dahya} confirms that
        the existence of regular unitary dilations
        is a strictly stronger condition than
        the existence of unitary dilations.
\end{rem}




\section[Application]{Application to the von Neumann polynomial inequality}
\label{sec:applications}


\firstparagraph
We conclude this paper with an application of the complete dissipativity condition
introduced in this paper.
When studying the aspects of multiple operators,
it is natural to consider algebraic combinations and their bounds.

\begin{defn}
    Let $S_{1},S_{2},\ldots,S_{d} \in \BoundedOps{\HilbertRaum}$ be commuting operators.
    We define the map
        $%
            \complex[X_{1},X_{1}^{-1},X_{2},X_{2}^{-1},\ldots,X_{d},X_{d}^{-1}] \ni p
            \mapsto p(S_{1},S_{2},\ldots,S_{d}) \in \BoundedOps{\HilbertRaum}
        $
    as the unique linear map satisfying
        $%
            p(S_{1},S_{2},\ldots,S_{d})
            \colonequals
            (\prod_{i\in\supp(\mathbf{n}^{-})}S_{i}^{-n_{i}})^{\ast}
            (\prod_{i\in\supp(\mathbf{n}^{+})}S_{i}^{n_{i}})
        $
    for all monomials of the form
        $p = \prod_{i=1}^{d}X_{i}^{n_{i}}$
    where $\mathbf{n} \in \integers^{d}$.
    We refer to this map as the \highlightTerm{regular polynomial evaluation}.
\end{defn}

It is easy to see that the commutativity of the operators guarantees
that regular polynomial evaluation is well-defined and linear.
Note that, unless the operators are normal
(which by Fuglede's theorem implies double commutativity),
regular polynomial evaluation is not multiplicative.
We may also consider $\complex[X_{1},X_{2},\ldots,X_{d}]$
as a subalgebra of the ring $\complex[X_{1},X_{1}^{-1},X_{2},X_{2}^{-1},\ldots,X_{d},X_{d}^{-1}]$.
The restriction of the above map to $\complex[X_{1},X_{2},\ldots,X_{d}]$
yields the usual \highlightTerm{polynomial evaluation} for tuples of operators.

Now in
    \cite{vonNeumann1951specThmContractions},
    \cite[Proposition~I.8.3 and Notes, p.~54]{Nagy1970},
    \cite[Theorem~1.2]{Pisier2001bookCBmaps},
    and
    \cite[Theorem~1.1]{Shalit2021DilationBook}
the following bound is proved for
commuting $d$-tuples $(S_{1},S_{2},\ldots,S_{d})$ of contractions
and $p\in\complex[X_{1},X_{2},\ldots,X_{d}]$:

    \begin{maths}[mc]{c}
        \norm{p(S_{1},S_{2},\ldots,S_{d})}
            \leq
                \displaystyle
                \sup_{\boldsymbol{\lambda}\in\Torus^{d}}
                    \abs{p(\lambda_{1},\lambda_{2},\ldots,\lambda_{d})},\\
    \end{maths}

\continueparagraph
provided $d\in\{1,2\}$.
In \cite[pp.~488--489]{Parrott1970counterExamplesDilation} an example of ${d=3}$ commuting contractions
is provided, which satisfies the above bounds,
but admits no \highlightTerm{simultaneous unitary power dilation}.\footnote{
    A tuple
        $(S_{1},S_{2},\ldots,S_{d}) \in \BoundedOps{\HilbertRaum}$
    of commuting operators,
    is said to have a \highlightTerm{simultaneous unitary power dilation},
    if a Hilbert space $\HilbertRaum^{\prime}$ exists,
    as well as $d$ commuting unitaries
        $(U_{1},U_{2},\ldots,U_{d}) \in \BoundedOps{\HilbertRaum^{\prime}}$
    and $r\in\BoundedOps{\HilbertRaum}{\HilbertRaum^{\prime}}$ (necessarily isometric),
    such that
        $\prod_{i=1}^{d}S_{i}^{n_{i}} = r^{\ast}(\prod_{i=1}^{d}U_{i}^{n_{i}})r$
    for all $\mathbf{n}=(n_{i})_{i=1}^{d}\in\naturalsZero^{d}$.
}
The proper correspondence between polynomial bounds and unitary dilations
can be found in \cite[Corollary~4.9]{Pisier2001bookCBmaps} (see also Corollary~4.13 in this book).
There it is shown (in particular for $d \geq 3$)
that a simultaneous unitary power dilation exists
if and only if
a strengthened version of the above inequalities hold
with $\complex$-valued polynomials replaced by
    ${n \times n}$ \emph{matrices} $(p_{ij})_{ij}$ with polynomial entries
    for all $n\in\naturalsPos$
and norms/absolute values replaced by
    appropriate matrix norms.

Further interesting connections can be found in recent literature.
For example, in \cite{Arcozzi2021vNpoly}
the authors define the notion of \emph{Siegel-dissipativity}
(which bears no relation to complete dissipativity),
and show that this condition suffices for a certain variation of the above bounds.
And in \cite{Hartz2021dilationFiniteDim,Hartz2022vNrowcontractions} finite dimensional dilations for operator systems
and polynomial bounds for finite dimensional commuting matrices
are investigated.



Moving to the continuous setting of multi-parameter (contractive) $\Cnought$-semigroups,
we can define a similar problem.

\begin{defn}
\makelabel{defn:polynomial-inequalities-semigroups:sig:article-dilation-problem-raj-dahya}
    Let $T$ be a $d$-parameter $\Cnought$-semigroup over $\HilbertRaum$.
    Say that $T$ satisfies \highlightTerm{polynomial bounds} if

    \begin{maths}[mc]{c}
    \eqtag[eq:vN-polynomial-inequality:sig:article-dilation-problem-raj-dahya]
        \norm{p(T_{1}(t_{1}),T_{2}(t_{2}),\ldots,T_{d}(t_{d}))}
            \leq
                \displaystyle
                \sup_{\boldsymbol{\lambda}\in\Torus^{d}}
                    \abs{p(\lambda_{1},\lambda_{2},\ldots,\lambda_{d})}\\
    \end{maths}

    \continueparagraph
    for all polynomials $p \in \complex[X_{1},X_{2},\ldots,X_{d}]$
    and all $\mathbf{t}=(t_{i})_{i=1}^{d}\in\realsNonNeg^{d}$.
    We say $T$ satisfies
        \highlightTerm{regular polynomial bounds},
    if the above holds
        for all $p \in \complex[X_{1},X_{1}^{-1},X_{2},X_{2}^{-1},\ldots,X_{d},X_{d}^{-1}]$
        and all $\mathbf{t} \in \realsNonNeg^{d}$.
    We say that $T$ satisfies
        \highlightTerm{regular polynomial bounds in a neighbourhood of $\zerovector$},
    if for some neighbourhood $W\subseteq\realsNonNeg^{d}$ of $\zerovector$,
    the above holds
        for all $p \in \complex[X_{1},X_{1}^{-1},X_{2},X_{2}^{-1},\ldots,X_{d},X_{d}^{-1}]$
        and all $\mathbf{t} \in W$.

    The \highlightTerm{%
        (regular) von Neumann polynomial inequality problem
        for $d$-parameter $\Cnought$-semigroups%
    }
    shall refer to determining whether
    $T$ satisfies (regular) polynomial bounds
    for all $d$-parameter contractive $\Cnought$-semigroups $T$ over a Hilbert space $\HilbertRaum$.
\end{defn}



Clearly, satisfaction of regular polynomial bounds
implies satisfaction of polynomial bounds.
Simple examples of the stronger condition are as follows:

\begin{prop}
\makelabel{prop:unitary-semigroups-satisfy-regular-vN-polynomial-inequalities:sig:article-dilation-problem-raj-dahya}
    Every $d$-parameter unitary $\Cnought$-semigroup $U$ over $\HilbertRaum$
    satisfies regular polynomial bounds.
\end{prop}

    \begin{proof}
        Let $p \in \complex[X_{1},X_{1}^{-1},X_{2},X_{2}^{-1},\ldots,X_{d},X_{d}^{-1}]$
            and
            $\mathbf{t}=(t_{i})_{i=1}^{d}\in\realsNonNeg^{d}$
        be arbitrary and fixed.
        Since $\{ U_{i}(t_{i}) \mid i \in \{1,2,\ldots,d\} \}$
        are doubly commuting unitary operators,
        one can apply the spectral mapping theorem for commutative $C^{\ast}$-algebras
        to simultaneously diagonalise these to multiplication operators over a semi-finite measure space
        (%
            see
            \cite[Theorem~1.3.6]{Murphy1990},
            \cite[3.3.1 and 3.4.1]{Pedersen2018cstaralg},
            and
            \cite{Haase2020spectralTheory}%
        ).
        That is, one can find
            a semi-finite measure space $(X,\mu)$,
            measurable $\reals$-valued functions
            $\theta_{1},\theta_{2},\ldots,\theta_{d}\in L^{\infty}(X,\mu)$,
            and
            a unitary operator $u\in\BoundedOps{\HilbertRaum}{L^{2}(X,\mu)}$,
        such that
            $%
                U_{i}(t_{i}) = u^{\ast} M_{e^{\imagunit \theta_{i}(\cdot)}} u
            $
            for all $i\in\{1,2,\ldots,d\}$.
        One can readily verify that
            $
                p(U_{1}(t_{1}),U_{2}(t_{2}),\ldots,U_{d}(t_{d}))
                = u^{\ast} M_{
                    p(%
                        e^{\imagunit \theta_{1}(\cdot)},
                        e^{\imagunit \theta_{2}(\cdot)},
                        \ldots
                        e^{\imagunit \theta_{d}(\cdot)}%
                    )
                } u
            $
        and thus

            \begin{maths}[mc]{rcl}
                \norm{p(U_{1}(t_{1}),U_{2}(t_{2}),\ldots,U_{d}(t_{d}))}
                    &=
                        &\norm{%
                            M_{
                                p(%
                                    e^{\imagunit \theta_{1}(\cdot)},
                                    e^{\imagunit \theta_{2}(\cdot)},
                                    \ldots
                                    e^{\imagunit \theta_{d}(\cdot)}%
                                )
                            }
                        }_{L^{2}(X,\mu)}\\
                    &=
                        &\norm{%
                            p(%
                                e^{\imagunit \theta_{1}(\cdot)},
                                e^{\imagunit \theta_{2}(\cdot)},
                                \ldots
                                e^{\imagunit \theta_{d}(\cdot)}%
                            )
                        }_{L^{\infty}(X,\mu)}\\
                    &\leq
                        &\displaystyle
                        \sup_{\boldsymbol{\lambda}\in\Torus^{d}}
                            \abs{p(\lambda_{1}, \lambda_{2}, \ldots \lambda_{d})}.\\
            \end{maths}

        \continueparagraph
        Thus $U$ satisfies regular polynomial bounds.
    \end{proof}

\begin{prop}
\makelabel{prop:regular-unitary-dilation-satisfy-regular-vN-polynomial-inequalities:sig:article-dilation-problem-raj-dahya}
    Let $T$ be a $d$-parameter $\Cnought$-semigroup over $\HilbertRaum$.
    If $T$ has a regular unitary dilation,
    then it satisfies regular polynomial bounds.
\end{prop}

    \begin{proof}
        For any monomial
            $p = \prod_{i=1}^{d}X_{i}^{n_{i}}$
        with $\mathbf{n} \in \integers^{d}$,
        and for any $\mathbf{t} \in \realsNonNeg^{d}$, one computes

            \begin{maths}[mc]{rcl}
            \eqtag[eq:polynomial-applied-to-semigroup:sig:article-dilation-problem-raj-dahya]
                p(T_{1}(t_{1}),T_{2}(t_{2}),\ldots,T_{d}(t_{d}))
                    &= &(\prod_{i\in\supp(\mathbf{n}^{-})}T_{i}(t_{i})^{-n_{i}})^{\ast}
                        (\prod_{i\in\supp(\mathbf{n}^{+})}T_{i}(t_{i})^{n_{i}})\\
                    &= &(\prod_{i\in\supp(\mathbf{n}^{-})}T_{i}(-n_{i}t_{i}))^{\ast}
                        (\prod_{i\in\supp(\mathbf{n}^{+})}T_{i}(n_{i}t_{i}))\\
                    &= &T(\mathbf{n}^{-}\odot\mathbf{t})^{\ast}T(\mathbf{n}^{+}\odot\mathbf{t})\\
                    &= &T((\mathbf{n}\odot\mathbf{t})^{-})^{\ast}T((\mathbf{n}\odot\mathbf{t})^{+}),\\
            \end{maths}

        \continueparagraph
        where we define $\odot : \reals^{d} \times \reals^{d} \to \reals^{d}$
        via
            $\mathbf{s} \odot \mathbf{s}^{\prime} \colonequals (s_{i}s^{\prime}_{i})_{i=1}^{d}$
        for $\mathbf{s}, \mathbf{s}^{\prime} \in \reals^{d}$.

        By assumption, $T$ has a regular unitary dilation $(U,\HilbertRaum^{\prime},r)$.
        Applying the definitions yields

            \begin{longmaths}[mc]{RCL}
                p(T_{1}(t_{1}),T_{2}(t_{2}),\ldots,T_{d}(t_{d}))
                &= &T((\mathbf{n}\odot\mathbf{t})^{-})^{\ast}T((\mathbf{n}\odot\mathbf{t})^{+})\\
                &= &r^{\ast}U(\mathbf{n}\odot\mathbf{t})r\\
                &= &r^{\ast}
                    U((\mathbf{n}\odot\mathbf{t})^{-})^{\ast}
                    U((\mathbf{n}\odot\mathbf{t})^{+})
                    r\\
                &= &r^{\ast}p(U_{1}(t_{1}),U_{2}(t_{2}),\ldots,U_{d}(t_{d}))r,\\
            \end{longmaths}

        \continueparagraph
        whereby the final equality holds by applying
            \eqcref{eq:polynomial-applied-to-semigroup:sig:article-dilation-problem-raj-dahya}
        to $U$ instead of $T$.
        Since the polynomial map is linear
        and the monomials linearly span
            $\complex[X_{1},X_{1}^{-1},X_{2},X_{2}^{-1},\ldots,X_{d},X_{d}^{-1}]$,
        the above computation implies that

            \begin{maths}[mc]{c}
                p(T_{1}(t_{1}),T_{2}(t_{2}),\ldots,T_{d}(t_{d}))
                = r^{\ast}p(U_{1}(t_{1}),U_{2}(t_{2}),\ldots,U_{d}(t_{d}))r\\
            \end{maths}

        \continueparagraph
        for all $p \in \complex[X_{1},X_{1}^{-1},X_{2},X_{2}^{-1},\ldots,X_{d},X_{d}^{-1}]$
        and all $\mathbf{t} \in \realsNonNeg^{d}$.

        Since $U$ satisfies regular polynomial bounds (see \Cref{prop:unitary-semigroups-satisfy-regular-vN-polynomial-inequalities:sig:article-dilation-problem-raj-dahya})
        and $r \in \BoundedOps{\HilbertRaum}{\HilbertRaum^{\prime}}$
        is an isometry and thus contractive,
        one obtains

            \begin{maths}[mc]{rcl}
                \norm{p(T_{1}(t_{1}),T_{2}(t_{2}),\ldots,T_{d}(t_{d}))}
                    &\leq
                        &\norm{r^{\ast}}
                        \norm{p(U_{1}(t_{1}),U_{2}(t_{2}),\ldots,U_{d}(t_{d}))}
                        \norm{r}\\
                    &\leq
                        &1
                        \cdot \displaystyle
                        \sup_{\boldsymbol{\lambda}\in\Torus^{d}}
                            \abs{p(\lambda_{1},\lambda_{2},\ldots,\lambda_{d})}
                        \cdot 1.\\
            \end{maths}

        Since this holds for all polynomials $p \in \complex[X_{1},X_{1}^{-1},X_{2},X_{2}^{-1},\ldots,X_{d},X_{d}^{-1}]$,
        it follows that $T$ satisfies regular polynomial bounds.
    \end{proof}



Now, to prove
    \eqcref{it:3:thm:regular-polynomial-inequality-iff-completely-dissipative:sig:article-dilation-problem-raj-dahya}%
    ~$\Rightarrow$~%
    \eqcref{it:1:thm:regular-polynomial-inequality-iff-completely-dissipative:sig:article-dilation-problem-raj-dahya}
of \Cref{thm:regular-polynomial-inequality-iff-completely-dissipative:sig:article-dilation-problem-raj-dahya},
observe that the dissipation operators involve expressions of the form
    $
        (\prod_{i \in C_{1}}A^{-}_{i})^{\ast}(\prod_{j\in C_{2}}A^{-}_{j})
    $
for $\isPartition{(C_{1},C_{2})}{K}$ and $K\subseteq\{1,2,\ldots,d\}$,
which, in terms of their indices, are fitting for the kinds of expressions
that arise in regular polynomial evaluations.
However, regular polynomial evaluation captures values of the semigroups and not their generators.
Using appropriate approximation arguments, this issue can nonetheless be overcome.
In this way, one may prove \Cref{thm:regular-polynomial-inequality-iff-completely-dissipative:sig:article-dilation-problem-raj-dahya}.

    \def\beweislabel{thm:regular-polynomial-inequality-iff-completely-dissipative:sig:article-dilation-problem-raj-dahya}
    \begin{proof}[of \Cref{\beweislabel}]
        The direction \hinRichtung{1}{2}
            has already been proved in \Cref{prop:regular-unitary-dilation-satisfy-regular-vN-polynomial-inequalities:sig:article-dilation-problem-raj-dahya}
        and \hinRichtung{2}{3} holds trivially.
        To prove \hinRichtung{3}{1},
        assume that $T$ satisfies regular polynomial bounds in a neighbourhood
            $W\subseteq\realsNonNeg^{d}$
        of $\zerovector$.
        By \Cref{thm:classification:sig:article-dilation-problem-raj-dahya},
        to show that $T$ has a regular unitary dilation,
        it is necessary and sufficient to show that
        the generators of $T$ are completely dissipative.
        To this end fix an arbitrary $K \subseteq \{1,2,\ldots,d\}$.
        We need to show that $S_{T,K} \geq \zeromatrix$.

        Now, for each $i \in K$, we know that
            ${t_{i}^{-1}(\onematrix - T_{i}(t_{i})) \longrightarrow -A_{i}}$
        in norm for ${\realsPos \ni t_{i} \longrightarrow 0}$.
        Thus, defining
        for $\mathbf{t} = (t_{i})_{i=1}^{d} \in\realsNonNeg^{d}$
        the product $t_{K} \colonequals \prod_{i \in K}t_{i}$
        and the operator

            \begin{maths}[mc]{c}
                P(\mathbf{t})
                    \colonequals
                        \displaystyle
                        \sum_{\mathclap{
                            \isPartition{(C_{1},C_{2})}{K}
                        }}
                            \displaystyle
                            \prod_{i \in C_{1}}
                                (\onematrix - T_{i}(t_{i})^{\ast})
                            \cdot
                            \displaystyle
                            \prod_{j \in C_{2}}
                                (\onematrix - T_{j}(t_{j})),\\
            \end{maths}

        \continueparagraph
        one has the norm convergence
            ${\frac{1}{2^{\card{K}}t_{K}}P(\mathbf{t}) \longrightarrow S_{T,K}}$
        for ${\realsPos^{d} \ni \mathbf{t} \longrightarrow 0}$.
        Since the subspace $\BoundedOps{\HilbertRaum}_{\geq 0}$
        of $\BoundedOps{\HilbertRaum}$ consisting of positive operators is norm-closed,
        in order to prove that $S_{T,K} \geq \zeromatrix$,
        it suffices to prove that $P(\mathbf{t}) \geq \zeromatrix$
        for all $\mathbf{t}\in\realsPos^{d}$ sufficiently close to $\zerovector$.

        So fix an arbitrary $\mathbf{t} \in W \cap \realsPos^{d}$.
        We shall show that $P(\mathbf{t}) \geq \zeromatrix$.
        Since $P(\mathbf{t}) \in \BoundedOps{\HilbertRaum}$ is clearly a self-adjoint operator,
        it suffices to prove that $\norm{\onematrix - \alpha \, P(\mathbf{t})} \leq 1$
        for sufficiently small $\alpha > 0$.%
        \footnote{
            Let $\cal{A}$ be a unital $C^{\ast}$-algebra
            and $a\in\cal{A}$ be self-adjoint.
            Let
                ${\lambda_{\min}\colonequals\min(\opSpectrum{a})}$
                and
                ${\lambda_{\max}\colonequals\max(\opSpectrum{a})}$
            and
                ${\alpha^{\ast}\colonequals\frac{1}{\max\{\abs{\lambda_{\min}},\,\abs{\lambda_{\max}}\}+1}}$.
            For $\alpha\in(0,\:\alpha^{\ast})$
            one has
                $\min(\opSpectrum{\onematrix - \alpha \, a}) = 1-\alpha\lambda_{\max}$
                and
                $\max(\opSpectrum{\onematrix - \alpha \, a}) = 1-\alpha\lambda_{\min}$,
            which are both positive by the choice of $\alpha^{\ast}$.
            By the spectral theorem it follows that
                $%
                    \norm{\onematrix - \alpha \, a}
                    = 1 - \alpha\lambda_{\min}
                $,
            which does not exceed $1$ if and only if $\lambda_{\min} \geq 0$.
            Now, $a$ is a positive element
            if and only if
                $\lambda_{\min} \geq 0$,
            which by the preceding argument holds
            if and only if
                $\norm{\onematrix - \alpha \, a} \leq 1$
                for all sufficiently small $\alpha\in\realsPos$.
        }

        Now, applying regular polynomial evaluation to

            \begin{maths}[mc]{c}
                p
                    \colonequals
                        \displaystyle
                        \sum_{\mathclap{
                            \isPartition{(C_{1},C_{2})}{K}
                        }}
                            \displaystyle
                            \prod_{i \in C_{1}}
                                (1 - X_{i}^{-1})
                            \cdot
                            \displaystyle
                            \prod_{j \in C_{2}}
                                (1 - X_{j})\\
            \end{maths}

        \continueparagraph
        yields

            \begin{maths}[mc]{c}
                p(T_{1}(t_{1}),T_{2}(t_{2}),\ldots,T_{d}(t_{d}))
                    = P(\mathbf{t}),\\
            \end{maths}

        \continueparagraph
        and since $T$ satisfies regular polynomial bounds in $W$,
        it follows that

            \begin{maths}[mc]{c}
                \norm{\onematrix - \alpha \, P(\mathbf{t})}
                    =
                        \norm{(1 - \alpha \, p)(T_{1}(t_{1}),T_{2}(t_{2}),\ldots,T_{d}(t_{d}))}
                    \leq
                        \displaystyle
                        \sup_{\boldsymbol{\lambda} \in \Torus^{d}}
                            \abs{1 - \alpha \, p(\lambda_{1},\lambda_{2},\ldots,\lambda_{d})}\\
            \end{maths}

        \continueparagraph
        for $\alpha > 0$.
        Thus, to prove that $P(\mathbf{t}) \geq \zeromatrix$,
        it suffices to prove for sufficiently small $\alpha > 0$,
        that
            $\abs{1 - \alpha \, p(\lambda_{1},\lambda_{2},\ldots,\lambda_{d})} \leq 1$
        for all $\boldsymbol{\lambda} \in \Torus^{d}$.
        This in turn holds if $p$,
        when viewed as a function ${p:\Torus^{d}\to\complex}$,
        is bounded, $\reals$-valued, and non-negative.
        Indeed, for each $\boldsymbol{\lambda} \in \Torus^{d}$

            \begin{maths}[mc]{rcl}
                p(\lambda_{1},\lambda_{2},\ldots,\lambda_{d})
                    &= &\displaystyle
                        \sum_{\isPartition{(C_{1},C_{2})}{K}}
                            \displaystyle
                            \prod_{i \in C_{1}}
                                (1 - \lambda_{i}^{-1})
                            \cdot
                            \displaystyle
                            \prod_{j \in C_{2}}
                                (1 - \lambda_{j})\\
                    &= &\displaystyle
                        \prod_{i \in K}
                        \underbrace{
                            \Big(
                                (1 - \lambda_{i}^{-1})
                                +
                                (1 - \lambda_{i})
                            \Big)
                        }_{
                            = 2\Re (1-\lambda_{i})
                            \in [0,\:4]
                        }
                    \:\in\:[0,\:4^{\card{K}}],\\
            \end{maths}

        \continueparagraph
        which completes the proof.
    \end{proof}

\begin{rem}
    The use of the complete dissipativity condition was brought to bear in the above argument.
    If one did not work with this concept, one might attempt to prove that
    satisfaction of regular polynomial bounds
    implies the existence of a regular unitary dilation,
    by showing that the Brehmer positivity criterion holds.
    To achieve this, one could attempt to capture the Brehmer operators
    via regular polynomial evaluation and then again rely on polynomial bounds.
    However, the expressions occurring in the Brehmer operators involve monomials of the form
        $
            (\prod_{i \in C}T_{i}(t))^{\ast}(\prod_{i\in C}T_{i}(t))
        $
    for $C\subseteq\{1,2,\ldots,d\}$,
    which do not seem to be expressions that can be captured by regular polynomial evaluation.
\end{rem}

Combining the examples in \S{}\ref{sec:examples:non-doubly-comm}
with \Cref{thm:regular-polynomial-inequality-iff-completely-dissipative:sig:article-dilation-problem-raj-dahya}
we can prove
    \Cref{cor:counter-examples-regular-polynomial-inequality:sig:article-dilation-problem-raj-dahya},
which negatively solves
    the regular von Neumann polynomial problem for multi-parameter $\Cnought$-semigroups:

    \def\beweislabel{cor:counter-examples-regular-polynomial-inequality:sig:article-dilation-problem-raj-dahya}
    \begin{proof}[of \Cref{\beweislabel}]
        By \Cref{prop:non-doubly-comm-example:sig:article-dilation-problem-raj-dahya}
        there exist $d$-parameter contractive $\Cnought$-semigroups
        whose generators are bounded and have strictly negative spectral bounds,
        and which are not completely dissipative.
        By the classification theorem (\Cref{thm:classification:sig:article-dilation-problem-raj-dahya}),
        these have no regular unitary dilations,
        and thus by \Cref{thm:regular-polynomial-inequality-iff-completely-dissipative:sig:article-dilation-problem-raj-dahya}
        do not satisfy regular polynomial bounds.
    \end{proof}

\begin{rem}
    It would be of interest to know whether the characterisation
    in \Cref{thm:regular-polynomial-inequality-iff-completely-dissipative:sig:article-dilation-problem-raj-dahya}
    also holds without the assumption of bounded generators.
    Certainly the implications
        \eqcref{it:1:thm:regular-polynomial-inequality-iff-completely-dissipative:sig:article-dilation-problem-raj-dahya}%
        ~$\Rightarrow$~%
        \eqcref{it:2:thm:regular-polynomial-inequality-iff-completely-dissipative:sig:article-dilation-problem-raj-dahya}%
        ~$\Rightarrow$~%
        \eqcref{it:3:thm:regular-polynomial-inequality-iff-completely-dissipative:sig:article-dilation-problem-raj-dahya}
    hold, as neither
        \Cref{prop:regular-unitary-dilation-satisfy-regular-vN-polynomial-inequalities:sig:article-dilation-problem-raj-dahya}
    nor
        \Cref{prop:unitary-semigroups-satisfy-regular-vN-polynomial-inequalities:sig:article-dilation-problem-raj-dahya}
    required the semigroups to have bounded generators.
    But our proof of the implication
        \eqcref{it:3:thm:regular-polynomial-inequality-iff-completely-dissipative:sig:article-dilation-problem-raj-dahya}%
        ~$\Rightarrow$~%
        \eqcref{it:1:thm:regular-polynomial-inequality-iff-completely-dissipative:sig:article-dilation-problem-raj-dahya}
    relied heavily on the complete dissipativity condition.
    The latter notion might need to be reformulated for the unbounded setting.
\end{rem}







\firstparagraph
\paragraph{Acknowledgement.}
The author is grateful to
    Tanja Eisner
    for her patience and feedback,
to
    Leonardo Goller
    and
    Rainer Nagel
    for useful discussions,
and to the referee for their constructive feedback.


\bibliographystyle{abbrv}
\def\bibname{References}
\bgroup

\egroup


\addresseshere
\end{document}
